\documentclass[oneside,reqno,11pt]{amsart}
\usepackage{preamble}

\title{Periods detecting Eisenstein series and sums of $L$-values I}

\author{Weixiao Lu and Guodong Xi}

\begin{document}

\maketitle

\begin{abstract}
    We study the automorphic period associated to a $G$-Hamiltonian variety $M$ whose dual is $\check{M} = T^*(\check{G}/\check{L})$, where $\check{G}$ is a general linear group and $\check{L}$ is a Levi subgroup. For certain cuspidal Eisenstein series, we prove that their period is equal to a finite sum of special values of $L$-functions. This sum is indexed by the fixed points of the associated extended $L$-parameter on $\check{M}$, confirming a conjecture by Ben-Zvi-Sakellaridis-Venkatesh in this case.
\end{abstract}

\tableofcontents

\section{Introduction} \label{sec: introduction}

\subsection{Relative Langlands conjectures of BZSV}

The relative Langlands program investigates the relation between periods of automorphic forms and special values of $L$-functions. In their seminal paper \cite{BZSV}, Ben--Zvi, Sakellaridis and Venkatesh proposed a general framework for this relationship. Their central idea is that periods are associated with Hamiltonian $G$-varieties, and each Hamiltonian variety should admit a ``dual" $\check{M}$, which is a $\check{G}$-Hamiltonian variety. With this framework, the period $\cP_M$ associated to $M$ is conjectured to an $L$-value attached to $\check{M}$.

To formulate a precise conjecture, the authors of \cite{BZSV} work in the context of function fields with everywhere unramified data. A part of their conjecture \cite[Conjecture 14.2.1]{BZSV} can be summarized as follows:
\begin{conjecture}[Ben-Zvi-Sakellaridis-Venkatesh] \label{conj:BZSV}
    Assume that $\check{M} = T^* \check{X}$, where $\check{X}$ is a $\check{G}$-spherical variety. Let $\pi$ be a tempered, everywhere unramified automorphic representation of $G(\bA)$ with $L$-parameter $\phi$. If $\phi$ has only finitely many fixed points $\{x_1,\cdots,x_r\}$ on $\check{X}$. Then for a suitably normalized spherical vector $f \in \pi$, we have
    \begin{equation*}
        \cP_M(f) \sim \sum_i L(0,(T_{x_i} \check{X})^{\shear})
    \end{equation*}
\end{conjecture}

In the number field setting, Mao, Wan and Zhang \cite{MWZ1} formulated an analog of the Conjecture \ref{conj:BZSV}, under the assumption that the hypothetical extended $L$-parameter of $\pi$ only has at most one fixed point on $\check{M}$. The goal of this paper is to prove specific cases of Conjecture \ref{conj:BZSV} in the style of \cite{MWZ1} for number fields. We focus on the case where $\check{X} = \check{G}/\check{L}$, with $\check{G}$ a general linear group and $\check{L}$ a Levi subgroup. A key feature of this case is that the set of fixed points is not necessarily a singleton, leading to an equality of the form
\begin{equation} \label{eq:P=L}
    \text{Automorphic period} \quad ``=" \quad \text{Sum of $L$-values}
\end{equation}
We note that the related work \cite{Wan} of Wan, who gavemain another example of \eqref{eq:P=L}, that the period associated to $\U(2) \backslash \SO(5)$ equals to a sum of two $L$-values.

\subsection{The main result}

\subsubsection{The period}

Throughout the article, we fix a number field $F$. Let $\bA := \bA_F$ and fix a non-trivial additive character $\psi$ of $F \backslash \bA$. We denote the general linear group $\GL_k$ over $F$ by $G_k$. 

For the introduction, we fix integers $n>0$ and $m>0$. Let $G=G_{2n+m}$. Let $N$ be the upper triangular unipotent subgroup of $G$. Let $Q$ be the standard parabolic subgroup of $G$ with Levi component $G_{2n+1} \times G_1^{m-1}$. Denote the unipotent radical of $Q$ by $U$. We define a character $\psi_U$ of $U(\bA)$ by 
\begin{equation*}
    \psi_U(u_{ij}) = \psi(u_{2n+1,2n+2} + \cdots + u_{2n+m-1,2n+m}).
\end{equation*}

Let $H$ denote the symplectic group $\Sp_{2n}$ preserving the symplectic form $\begin{pmatrix}
    & J \\ -J &
\end{pmatrix}$ with 
$J = \begin{psmallmatrix}
    & & 1 \\ & \cdots & \\1 & &
\end{psmallmatrix}$. We embed $H$ into $G$ as a subgroup in the upper-left corner. Note that $H$ normalizes $U$, and the character $\psi_U$ is invariant under the conjugation action of $H(\bA)$.

For an automorphic form $f$ on $G(\bA)$, we define its Fourier coefficient along $U$ by
    \begin{equation*}
        f_{U,\psi}(g) := \int_{[U]} f(ug) \psi_U^{-1}(u) \rd u
    \end{equation*}
We then define the period integral
    \begin{equation} \label{eq:intro_1}
        \cP(f) := \int_{[H]} f_{U,\psi}(h) \rd h.
    \end{equation}
We remark that the integral defining $\cP$ is not necessarily absolutely convergent. Thus, to define this period for a broad class of automorphic forms, a regularization of the integral \eqref{eq:intro_1} is required. 

The period $\cP$ is the period associated to the $G$-hyperspherical Hamiltonian variety $M=T^*(G/HU,\psi_U)$, whose conjectural dual variety is $\check{M} = T^* \check{X}$, where $\check{X} = G_{2n+m}/G_n \times G_{n+m}$ (see \cite[\S 4]{BZSV}, \cite[Appendix A]{Sakellaridis}).

\subsubsection{The main result}
\label{sssec:intro_main_result}

We will study the period $\cP(f)$ when $f$ is a cuspidal Eisenstein series. Let $P$ be a standard parabolic subgroup of $G$, let $\pi$ be a unitary cuspidal automorphic representation of $M_P$. For concreteness, we assume $M_P = G_{n_1} \times \cdots \times G_{n_k}$ and $\pi = \pi_1 \boxtimes \cdots \boxtimes \pi_k$, where each $\pi_i$ is a unitary cuspidal automorphic representation of $G_{n_i}$.

Let $A_{M_P}^{\infty} \cong \R_{>0}^k$ denote the central subgroup of $M_P(\bA)$, consisting of elements of the form $\begin{psmallmatrix}
    t_1 I_{n_1} & & \\ & \cdots &  \\ & & t_k I_{n_k} \end{psmallmatrix}$ with $t_i \in \R_{>0}$. 
    
We write $\pi_0$ for the unique unramified twist of $\pi$ such that the central character of $\pi_0$ is trivial on $A_{M_P}^{\infty}$. By an unramified twist, we mean $\pi_0$ is of the form $\pi_1 \lvert \cdot \rvert^{\lambda_1} \boxtimes \cdots \boxtimes \pi_k \lvert \cdot \rvert^{\lambda_k}$ with $\lambda_j \in i\R$ for $1 \le j \le k$.

Let $\varphi \in \opn{Ind}_{P(\bA)}^{G(\bA)} \pi$, which we regard as a function on $N_P(\bA)M_P(F) \backslash G(\bA)$. Let $E(\varphi) := E(\cdot,\varphi,0)$ be the associated Eisenstein series. We denote by $\opn{Fix}(\pi)$ the fixed point set of the (hypothetical) $L$-parameter of $\opn{Ind}_{P(\bA)}^{G(\bA)} \pi$ acting on $\check{X}$.

\begin{theorem}[Theorem \ref{thm:n_m_Eisenstein}, rough form] \label{thm:intro}
    Assume that $\opn{Fix}(\pi_0)$ is finite, then we have $\opn{Fix}(\pi_0) = \opn{Fix}(\pi)$ and
    \begin{enumerate}
        \item The period $\cP(E(\varphi))$ can be defined canonically.
        \item Let $\tS$ be a sufficiently large set of places of $F$, then we have
        \begin{equation} \label{eq:intro_2}
            \begin{split}
            \cP(E(\varphi)) = &(\Delta_H^{\tS,*})^{-1} \prod_{1 \le i < j \le k}L(1,\pi_i \times \pi_j^{\vee})^{-1} \times \\
            & \sum_{\sigma \in \opn{Fix}(\pi)}
            L^{\tS}(1,T_\sigma \check{X}) \cdot ( \text{local zeta integral of } W_{\varphi,\tS}^{M_P} \text{at $\tS$} ),
            \end{split}
        \end{equation}
        where 
        \begin{itemize}
            \item $\Delta_H^{\tS,*}$ is a constant related to the Tamagawa measure (see \S \ref{sssec:prelim_Tamagawa})
            \item We define
            \begin{equation*}
                W_{\varphi}^{M_P}(g) := \int_{[M_P^N]} \varphi(ug) \psi^{-1}(u) \rd u
            \end{equation*}
            to be the Whittaker coefficient of the section $\varphi$, where $M_P^N = M_P \cap N$. And we decompose $W_\varphi^{M_P}$ as $W_{\varphi}^{M_P} = W_{\varphi}^{M_P,\tS} W_{\varphi,\tS}^{M_P}$ with $W_{\varphi}^{M_P,\tS}$ is spherical and equals 1 at $g=1$.
        \end{itemize}
    \end{enumerate}
\end{theorem}

We will address the definition of $\cP(E(\varphi))$ in \S \ref{ssec:intro_regularization}. For now, let us focus on the identity \eqref{eq:intro_2}.
\begin{remark}
    \begin{itemize}
        \item If $\pi_i = \pi_j$ for some $i \ne j$, the $L$-function $L(s,\pi_i \times \pi_j^{\vee})$ has a pole at $s=1$. In this case, the right-hand side of \eqref{eq:intro_2} is interpreted as $0$, so $\cP(E(\varphi))=0$.
        \item Consider the function field analogue of Theorem \ref{thm:intro} under the assumption that all data are unramified. We may take $\tS = \varnothing$. Then if we normalize so that $W^{M_P}_{\varphi}(1)=1$, equation \eqref{eq:intro_2} becomes
        \begin{equation*}
            \cP(E(\varphi)) = \Delta_H^{*,-1} \prod_{1\le i<j\le k} L(1,\pi_i \times \pi_j^{\vee}) \sum_{\sigma \in \opn{Fix}(\pi)} L(1,T_{\sigma} \check{X}).
        \end{equation*}
        By a result of Shahidi \cite[\S 4]{Shahidi}, the Whittaker coefficient $W_{E(\varphi)}(g) := \int_{[N]} E(\varphi)(u)\psi^{-1}(u) \rd u$ of $E(\varphi)$ satisfies $W_{E(\varphi)}(1) = \prod_{1 \le i<j \le k} L(1,\pi_i \times \pi_j^{\vee})^{-1}$. Therefore, if we normalize so that $W_{E(\varphi)}(1)=1$, then
        \begin{equation*}
            \cP(E(\varphi)) = \Delta_H^{*,-1} \sum_{\sigma \in \opn{Fix}(\pi)} L(1,T_\sigma \check{X}).
        \end{equation*}
        This is the form which exactly looks like Conjecture \ref{conj:BZSV}.
        \item The $L$ function $L(1,T_{\sigma} \check{X})$ is an example of \emph{non-linear} $L$-function; see \cite[Remark 14.2.4]{BZSV}, \cite{CV}.
        \item A version of Theorem \ref{thm:intro} also holds for $\check{X} = \GL_{2n}/\GL_n \times \GL_n$ (see \S \ref{sec:n_n}). In this case, the period associated to the dual variety $\check{M}$ (rather than $M$) is the \emph{Friedberg-Jacquet period} (or \emph{linear period}) studied in \cite{FJ}.
        \item In our earlier work \cite{LX}, we proved a special case of Theorem \ref{thm:intro} when $m=1$ and $P$ is a maximal parabolic. We note that the method in the present article differs from the that of \emph{loc .cit.}. We also note that the period associated to $\check{M}$ is also studied in \cite{FJ}.
    \end{itemize}
\end{remark}

\subsubsection{A more precise formulation}
\label{sssec:intro_precise}

The statement of Theorem \ref{thm:intro} involves the hypothetical global Langlands correspondence. To avoid this, we now describe the fixed point set $\opn{Fix}(\pi)$ and the $L$-function $L(1,T_{\sigma} \check{X})$ solely in terms of the representation $\pi$. In Appendix \S \ref{appendix:fixed_point}, we verify that, assuming global Langlands correspondence, this description coincides with the definition given by the $L$-parameter.

Let $P$ be a standard parabolic and let $\pi$ be a unitary cuspidal automorphic representation of $M_P$ as above. The condition for $\opn{Fix}(\pi)$ to be discrete is equivalent to the following:

\begin{num}
    \item \label{eq:intro_3} For any subset $I \subset \{1,2,\cdots,k\}$ such that $\sum_{i \in I} n_i = n$, we have $\pi_i \ne \pi_j$ for any $i \in I$ and $j \in I^c$, where $I^c$ denote the complement of $I$.
\end{num}

From this description, we see that the condition that $\opn{Fix}(\pi_0)$ is discrete is stronger than the condition that $\opn{Fix}(\pi)$ is discrete. Therefore, Theorem \ref{thm:intro} is slightly weaker than the expectation from Conjecture \ref{conj:BZSV}.
 
Henceforth, we assume that \eqref{eq:intro_3} holds for $\pi_0$. We define $\opn{Fix}({\pi})$ as the set of permutations $\sigma: \{1,2,\cdots,k\} \to \{1,2,\cdots,k\}$, for which there exists $1 \le t \le k$ with:
\begin{enumerate}
    \item $n_{\sigma(1)} + \cdots + n_{\sigma(t)}=n, n_{\sigma(t+1)} + \cdots + n_{\sigma(k)}=n+m$.
    \item $\sigma(1) < \cdots < \sigma(t)$ and $\sigma(t+1) < \cdots < \sigma(k)$.
\end{enumerate}
Note that $t$ is uniquely determined by $\sigma$ and $\opn{Fix}(\pi)$ is in bijection with the set
\begin{equation*}
    \left\{ I \subset \{1,2,\cdots,k\} \mid \sum_{i \in I} n_i = n \right\}.
\end{equation*}

In particular $\opn{Fix}(\pi)$ is non-empty if and only if, up to permutation, $E(\varphi)$ is an Eisenstein series ``passing through" the maximal Levi subgroup $G_n \times G_{n+m}$. In other words, the period $\cP$ \emph{detects $(n,n+m)$-Eisenstein series}.

For $\sigma \in \opn{Fix}(\pi)$ corresponds to the subset $I$ above, we put
\begin{equation*}
    L(s,T_{\sigma} \check{X}) = \prod_{(i,j) \in I \times I^c} L(s,\pi_i^{\vee} \times \pi_j) L(s,\pi_i \times \pi_j^{\vee}).
\end{equation*}
Note that the condition \eqref{eq:intro_3} ensures that this $L$-function does not have a pole at $s=1$. This completes the description of $\opn{Fix}(\pi)$ and $L(1, T_\sigma \check{X})$ in \eqref{eq:intro_2}.

We now describe the local zeta integral. Let $v$ be a place of $F$ and let $R = M_RN_R$ a standard parabolic subgroup of $G$. Let $\Pi$ be an irreducible generic representation of $M_R$ and let $\cW(\Pi,\psi_v)$ denote the Whittaker model of $\Pi_R$. We define $\opn{Ind}_{R(F_v)}^{G(F_v)} \cW(\Pi,\psi_v)$ to be the space of functions $W^{M_R}:G(F_v) \to \C$ such that for any $g \in G(F_v)$, the map $m \in M_{R}(F_v) \mapsto \delta_{R}^{-1}(m) W^{M_R}(mg)$ belongs to $\cW(\Pi,\psi_v)$. 

Let $Q_n$ denote the standard parabolic subgroup of $G$ with Levi component $G_n \times G_{n+m}$. Let $\Pi_M = \Pi_n \boxtimes \Pi_{n+m}$ be an irreducible generic representation of $M_{Q_n}(F_v)$. For $W^M \in \opn{Ind}_{Q_n(F_v)}^{G(F_v)} \cW(\Pi_M,\psi_v)$ and $\lambda \in \fa_{Q_n,\C}^*$, we define
\begin{equation*}
    Z_v(\lambda,W^M) = \int_{ N_H(F_v) \backslash H(F_v)} W^M(h) e^{\langle \lambda, H_{Q_n}(h) \rangle} \rd h,
\end{equation*}
where $N_H := N \cap H$. The integral is convergent for $\opn{Re}(\lambda)$ lies in a suitable cone and has meromorphic continuation to $\fa_{Q_n,\C}^*$.

Note that $\opn{Fix}(\pi)$ can be identified with a subset of the Weyl group $W^G$ of $G$. Specifically, we identify an element $\sigma \in \opn{Fix}(\pi)$ with the permutation that preserves the internal order of each block of $M_P$. We write $P_{\sigma}$ for the standard parabolic subgroup of $G_{2n+m}$ with $M_{P_{\sigma}} = G_{n_{\sigma(1)}} \times \cdots \times G_{n_{\sigma(k)}}$. Let $\tS$ be a finite set of places of $F$. Then $\sigma$ induces a \emph{normalized intertwining operator} (see \S \ref{sssec:prelim_intertwining_operator}) $N_{\pi,\tS}: \opn{Ind}_{P(F_{\tS})}^{G(F_{\tS})} \cW(\pi,\psi_{\tS}) \to \opn{Ind}_{P_{\sigma}(F_{\tS})}^{G(F_{\tS})} \cW(\sigma \pi,\psi_{\tS})$.

Let $\Pi_{\sigma,n} = \pi_{\sigma(1)} \boxplus \cdots \boxplus \pi_{\sigma(t)}$ (parabolic induction) and $\Pi_{\sigma,n+m} = \pi_{\sigma(t+1)} \boxplus \cdots \boxplus \pi_{\sigma(k)}$. Finally, let $\Omega_{\tS}^{Q_n}$ denote the \emph{Jacquet integral} (see \S \ref{sssec:prelim_Jacquet_integral}), it is a map from $\opn{Ind}_{P_{\sigma}(F_{\tS})}^{G(F_{\tS})} \cW(\sigma \pi,\psi_{\tS})$ to $\opn{Ind}_{Q_n(F_{\tS})}^{G(F_{\tS})} \cW(\Pi_{\sigma,n} \boxtimes \Pi_{\sigma,n+m},\psi_{\tS})$. 

With these notations, the precise form of the identity \eqref{eq:intro_2} is given by   
    \begin{equation} \label{eq:intro_4}
        \begin{split}
            \cP(E(\varphi)) = &(\Delta_H^{\tS,*})^{-1} \prod_{1 \le i < j \le k}L(1,\pi_i \times \pi_j^{\vee})^{-1} \times \\
            & \sum_{\sigma \in \opn{Fix}(\pi)}
            L^{\tS}(1,T_\sigma \check{X}) \left( \prod_{1 \le i<j \le k}  L_{\tS}(1,\pi_{\sigma(i)} \times \pi_{\sigma(j)}^{\vee}) \right) Z_{\tS}(0,\Omega_{\tS}^{Q_n}(N_{\pi,\tS}(\sigma)W^{M_P}_{\varphi.\tS})).
        \end{split}
    \end{equation}

\subsection{Definition of the period}
\label{ssec:intro_regularization}

We now discuss the definition of the period integral.

\subsubsection{Definition via continuous extension}
Let $\cS([G])$ denote the space of Schwartz functions on $[G]$ and let $\cT([G])$ denote the space of smooth functions of uniform moderate growth on $[G]$ (see \ref{sssec:prelim_function_spaces}). Both of them are vector spaces over $\C$ carrying a natural topology. When $f \in \cS([G])$, the integral defining $\cP(f)$ is absolutely convergent.

Let $\fX(G)$ denote the set of cuspidal datum of $G$. (see \S \ref{sssec:prelim_function_spaces}) We have the following coarse Langlands spectral decomposition according to cuspidal support:
\begin{equation*}
    L^2([G]) = \widehat{\bigoplus}_{\chi \in \fX(G)} L^2_{\chi}([G]).
\end{equation*}

For a subset $\fX \subset \fX(G)$, we put $L^2_{\fX}([G]) = \widehat{\bigoplus}_{\chi \in \fX} L^2_{\chi}([G])$, and let $\cS_{\fX}([G]) = \cS([G]) \cap L^2_{\fX}([G])$. These are Schwartz functions with cuspidal support in $\fX$. Let $\cT_{\fX}([G])$ denote the orthogonal complement of $\cS_{\fX^c}([G])$ in $\cT([G])$. When $\cT_{\fX}([G])$ carries the subspace topology inherited from $\cT([G])$, $\cS_{\fX}([G])$ is a dense subspace of $\cT_{\fX}([G])$.

Let $\fX_{\Delta}$ denote the cuspidal datum represented by $(M_P,\pi)$ such that $\pi$ satisfies \eqref{eq:intro_3}. We write $\cS_{\Delta}([G])$ (resp. $\cT_{\Delta}([G])$) for $\cS_{\fX_{\Delta}}([G])$ (resp. $\cT_{\fX_{\Delta}}([G])$). Then we have the following theorem:
\begin{theorem} \label{thm:intro_2}
    The period $\cP$, defined on $\cS_{\Delta}([G])$, admits a (necessarily unique) continuous extension to $\cT_{\Delta}([G])$.
\end{theorem}

Let $P$ be a standard parabolic subgroup of $G$ and let $\pi$ be a unitary cuspidal automorphic representation of $M_P$ such that $\pi_0$ satisfies \eqref{eq:intro_3}. Then the Eisenstein series $E(\varphi)$ lies in $\cT_{\Delta}([G])$. This explains the meaning of (1) in Theorem \ref{thm:intro}.

\subsubsection{Definition via truncation operator}

When $m=1$, there is an alternative definition of the period with potential applications, for example, in relative trace formulas. The period $\cP$ is taking a $\Sp_{2n}$ period of an automorphic form on $\GL_{2n+1}$. The work of Zydor \cite{Zydor19} suggests a regularization of the period $\cP$ via truncation. Let $f \in \cT([G])$ and let $T$ be a truncation parameter, in \emph{loc. cit.}, Zydor defines a truncated function $\Lambda^T f$ on $[H]$ which is rapidly decreasing. In \S \ref{sec:truncation}, we prove the following result:

\begin{proposition}
    For $f \in \cT_{\Delta}([G])$, the integral
    \begin{equation*}
        \int_{[H]} \Lambda^T f(h) \rd h
    \end{equation*}
    is independent of $T$. Moreover, this constant coincides with $\cP(f)$ as defined in Theorem \ref{thm:intro_2}.
\end{proposition}

\subsection{The strategy of the proof}

The proof of Theorem \ref{thm:intro} and Theorem \ref{thm:intro_2} proceed via an unfolding argument, analogous to the standard unfolding of period integrals into integrals of Whittaker functions. 

Let $\psi_n$ be the degenerate character on $N(\bA)$ defined by 
\begin{equation*}
    \psi_n(u) = \psi(u_{1,2}+\cdots+u_{n-1,n}+u_{n+1,n+2}+\cdots+u_{2n+m-1,2n+m}). 
\end{equation*}
For $f \in \cT([G])$, we define the associated degenerate Whittaker coefficient by:
\begin{equation*}
    V_f(g) = \int_{[N]} f(ug) \psi_n^{-1}(u) \rd u.
\end{equation*}

The key step is the following proposition:
\begin{proposition}[Proposition \ref{prop:n_m_unfolding}]
    For $f \in \cS_{\Delta}([G])$, then we have
    \begin{equation*}
        \cP(f) = \int_{N_H(\bA) \backslash H(\bA)} V_f(h) \rd h.
    \end{equation*}
\end{proposition}

The proof of this proposition involves performing a Fourier expansion along certain abelian unipotent subgroups, similar to the Fourier expansion of a cusp form. However, since $f$ is not necessarily cuspidal, extra terms appear in the unfolding process. Our assumption on the cuspidal support of $f$ ensures that these extra terms do not contribute to the period.

For $f \in \cT([G])$ and $\lambda \in \fa_{Q_n,\C}^*$, we define a global zeta integral by
\begin{equation*}
    Z(\lambda,f) = \int_{N_H(\bA) \backslash H(\bA)} V_f(h) e^{\langle \lambda, H_{Q_n}(h) \rangle} \rd h.
\end{equation*}
The global zeta integral $Z(\lambda,f)$ is absolutely convergent when $\opn{Re}(\lambda)$ lies in suitable half-plane. We then show that for $f \in \cT_{\Delta}([G])$, the zeta integral $Z(\lambda,f)$ is holomorphic at $\lambda=0$, and the map $f \mapsto Z(0,f)$ provides the continuous extension of $\cP$ to $\cT_{\Delta}([G])$.

Let $Q_n^H := Q_n \cap H$ be the Siegel parabolic of $H$. Let $K_H$ denote a maximal compact of $H(\bA)$ which is in good position relative to the upper triangular Borel at the non-Archimedean places. Using the Iwasawa decomposition $H(\bA) = Q_n^H(\bA) K_H$, the zeta integral $Z(\lambda,f)$ can be expressed as
\begin{equation} \label{eq:intro_5}
    Z(\lambda,f) = \int_{K_H} \widetilde{Z}^{\mathrm{RS}}(s_{\lambda}+n+1,R(k)f_{Q_n})  \rd k,
\end{equation}
where 
\begin{enumerate}
    \item $\widetilde{Z}^{\mathrm{RS}}$ denotes the (twisted) Rankin-Selberg zeta integral: for $f \in \cT([G_n \times G_{n+m}])$, we define
    \begin{equation*}
        \widetilde{Z}^{\mathrm{RS}}(s,f) = \int_{N_n(\bA) \backslash G_n(\bA)} W_f(J {}^t g^{-1}J, \begin{pmatrix}
            g & \\ & 1_m
        \end{pmatrix} ) \lvert \det g \rvert^s \rd g,
    \end{equation*}
    where $N_n$ denotes the upper triangular unipotent subgroup of $G_n$ and 
    \begin{equation*}
        W_f(g) = \int_{[N_n \times N_{n+m}]} f(ug) \psi^{-1}(u) \rd u
    \end{equation*} is the Whittaker coefficient of $f$.
    \item $\alpha \in \Delta_{Q_n}$ is the unique simple root and $\alpha^{\vee}$ denotes its coroot and $s_{\lambda} = -\langle \lambda, \alpha^{\vee} \rangle$.
\end{enumerate}

By \eqref{eq:intro_5}, the problem reduces to show that the Rankin-Selberg integral admits a continuous extension to uniform moderate growth functions with specific cuspidal support. When $m=1$, this is achieved in \cite[\S 7]{BPCZ}. We will show the case when $m>1$ in \S \ref{sec:higher_corank_Rankin_Selberg}. The proof involves another unfolding process and an application of the Phragmen-Lindel\"{o}f principle. 

Finally, when $f = E(\varphi)$ is a cuspidal Eisenstein series, we use the formula \eqref{eq:intro_4} to compute $\cP(E(\varphi)) = Z(0,f)$. By combining the constant term formula for Eisenstein series and the Euler decomposition of Rankin-Selberg integral, we will achieve \eqref{eq:intro_2}. The summation of $L$-values appearing in the formula results from the formula for the constant terms of Eisenstein series.

\subsection{The structure of this article}

After the preliminaries in \S \ref{sec:prelim}, we will review the result of canonical extension of Rankin-Selberg period of corank $1$ \cite[\S 7]{BPCZ} and equal rank \cite[\S 10.3]{BPCZ} to functions with certain cuspidal support in \S \ref{sec:canonical_extension_of_periods}. And we will do the higher corank case in \S \ref{sec:higher_corank_Rankin_Selberg}. Then we will study the period detecting $(n,n)$-Eisenstein series in \S \ref{sec:n_n} and detecting $(n,n+m)$-Eisenstein series in \S \ref{sec:higher_corank}.

\subsection{Acknowledgement}

We thank Chen Wan for introducing this problem. We thank Rapha\"{e}l Beuzart-Plessis, Paul Boisseau, Colin Loh, Omer Offen, Dihua Jiang, Zeyu Wang, Hang Xue, Lei Zhang and Wei Zhang for helpful suggestions and discussions.

\section{Preliminaries}
\label{sec:prelim}

\subsection{General notations}
    \begin{itemize} 
        \item Throughout this article, unless otherwise specified, we fix a number field $F$. Let $\bA := \bA_F$ be the ad\`{e}le ring of $F$ and let $\bA_f$ be the finite ad\`{e}les. Let $v$ be a place of $F$, we write $F_v$ for the completion. Let $\tS$ be a finite set of places of $F$, we write $F_{\tS} := \prod_{v \in \tS} F_v$.
        \item We write $G_n$ for the general linear group $\GL_n$ over $F$.  Let $\tS_{\infty}$ be the set of Archimedean places, we write $F_{\infty} := F_{\tS_{\infty}}$.
        \item Let $\Sp_{2n}$ be the symplectic group preserving the symplectic form $\begin{pmatrix}
            & J \\ -J &
        \end{pmatrix}$ with 
        $J = \begin{psmallmatrix}
            & & 1 \\ & \cdots & \\1 & &
        \end{psmallmatrix}$. 
        \item  For integer $m>1$, let $1_{m}$ denote the identity matrix of size $m$.
        \item Let $\cH_{>C} = \{z \in \C \mid \opn{Re}(z)>C\}$.
        \item For a ring $R$, we write $R^n$ the column vector with coefficient in $R$ of size $n$ and we write $R_n$ for the row vector of size $n$.
        \item Let $f,g$ be two positive functions on a set $X$, we write $f \ll g$ if there exists $C>0$ such that $f(x) \le Cg(x)$ for any $x \in X$. 
        \item For a set $X$ and a subset $A \subset X$, we write $A^c$ the complement of $A$ in $X$.
    \end{itemize}

\subsection{Groups}
\label{ssec:prelim_group}

Let $G$ be a connected linear algebraic group over a global field $F$. Let $[G] := G(F) \backslash G(\bA)$ the ad\`{e}lic quotient of $G$.

\subsubsection{Tamagawa measure} \label{sssec:prelim_Tamagawa}

We fix the Tamagawa measure $dg$ on $G(\bA)$, and thus on $[G]$ as described in \cite[section 2.3]{BPCZ}. We recall the definition here. Fix a non-trivial additive character $\psi:F \backslash \bA_F \to \C^\times$. We decompose $\psi$ as $\psi = \prod \psi_v$. For each place $v$ of $F$, $\psi_v$ determines the self-dual measure on $F_v$. Let $\omega$ be an $F$-rational $G$-invariant top differential form on $G$. For each place $v$, $\lvert \omega \rvert_v$ gives a measure $\rd g_v$ on $G(F_v)$. Moreover, according to the results of Gross \cite{Gross97}, there exists a global Artin-Tate $L$-function $L_G(s)$ such that 
\begin{equation*}
    \rd g_v(G(\cO_v)) = L_{G,v}(0)
\end{equation*}
for almost all places $v$. We denote by
\begin{equation*}
    \Delta_{G,v} := L_{G,v}(0)
\end{equation*}
and let $\Delta_G^*$ denote the leading coefficient of the Laurent expansion of $L_G(s)$ at $s=0$. The Tamagawa measure is defined by
\begin{equation*}
    \rd g = (\Delta_G^*)^{-1} \prod_v \rd g_v.
\end{equation*}
The measure is independent of the choice of $\omega$. For a finite set $\tS$ of places of $F$, let $\Delta_{G}^{\tS,*}$ denote the leading coefficient of the partial $L$-function $L^{\tS}_{G}(s)$ at $s=0$.

\subsubsection{Norms and heights}

Let $N$ be a positive integer. For $x \in \bA^N$, we define
\begin{equation*}
    \| x \| = \prod_{v} \max\{ \lvert x_{1,v} \rvert_v, \cdots, \lvert x_{n,v} \rvert_v,1 \},
\end{equation*}
where the product runs over the set of places of $F$. For a linear algebraic group $G$, we fix a closed embedding $\iota$ of $G$ into an affine space. Then for $g \in G(\bA)$ we define $\| g \|_{G(\bA)} = \| \iota(g) \|$. Let $\| \cdot \|'_{G(\bA)}$ be the norm defined by another embedding $\iota'$, then there exists $r>0$ such that $\| g \|_{G(\bA)} \ll \| g \|_{G(\bA)}^{\prime,r}$. We refer the reader to \cite[Appendix A]{BP21} for more properties of the norm $\| \cdot \|_{G(\bA)}$.

For the rest of \S \ref{ssec:prelim_group}, we assume that $G$ is a connected reductive group. We fix a maximal split torus $A_0$ of $G$ and fix a minimal parabolic subgroup $P_0$ of $G$ containing $A_0$. A parabolic subgroup of $G$ is called \emph{standard} if it contains $P_0$ and is called \emph{semi-standard} if it contains $A_0$.

Let $P$ be a semi-standard parabolic subgroup, we denote by $M_P$ the Levi subgroup of $P$ containing $A_0$ and denote by $N_P$ the unipotent radical of $P$. Since the natural map $M_P \times N_P \to P$ is an isomorphism of varieties. We see that $\| mn \|_{P(\bA)} \sim \|m\|_{M_P(\bA)} \|n\|_{N_P(\bA)}$. That is, there exists $c>1$ such that
\begin{equation*}
  \| mn \|_{P(\bA)}^{1/c} \ll \|m\|_{M(\bA)} \| n \|_{N(\bA)} \ll \|mn\|_{P(\bA)}^c
\end{equation*}
holds for all $m,n \in M_P(\bA) \times N(\bA)$. As a consequence
\begin{num} 
    \item \label{eq:Iwasawa_bound} There exists $C>0$ and $r>0$ such that for any $g \in G(\bA)$ and $(m,n,k) \in M_P(\bA)N_P(\bA)K$ such that $g=nmk$, we have $\|m\|_{M(\bA)} \le C \|g\|_{G(\bA)}^r$
\end{num}

 For a semi-standard parabolic subgroup $P$ of $G$, we put
        \[
            [G]_P := N_P(\bA)M_P(F) \backslash G(\bA).
        \]
    We define a norm on $[G]_P$ by 
        \[
            \| g \|_P := \inf_{\gamma \in N_P(\bA)M_P(F)} \|\gamma g\|.
        \]

\subsubsection{Weyl groups}
Let $W$ be the Weyl group of $(G, A_0)$, that is, the quotient by $M_0(F)$ of the normalizer of $A_0$ in $G(F)$. For a standard parabolic subgroup $P$, we write $W^P := W^{M_P}$, and we regarded it as a subgroup of $W$. For standard parabolic subgroups $P,Q$, we denote by
    \begin{equation*}
        {}_Q W {}_P := \{ w \in W \mid M_P \cap w^{-1}P_0w = M_P \cap P_0, \quad M_Q \cap wP_0w^{-1} = M_Q \cap P_0 \}.
    \end{equation*}
The set ${}_Q W_P$ forms a representative of the double coset $W^Q \backslash W/W^P$. For $w \in {}_Q W_P$, $M_P \cap w^{-1}M_Qw$ is the Levi factor of the standard parabolic subgroup $P_w = (M_P \cap w^{-1}Qw)N_P$. In the same way, $M_Q \cap wM_Pw^{-1}$ is the Levi factor of the standard parabolic subgroup $Q_w = (L \cap wPw^{-1})N_Q$. We have $P_w \subset P$, $Q_w \subset Q$ We also define
    \begin{equation*}
        W(P;Q) = \{ w \in {}_Q W_P \mid M_P \subset w^{-1}M_Qw \}.
    \end{equation*}
    and
    \begin{equation*}
        W(P,Q) = \{ w \in {}_Q W_P \mid M_P = w^{-1}M_Q w \}.
    \end{equation*}

 $P$ and $Q$ are called \emph{associate} if $W(P,Q) \ne \varnothing$. For example, for any $P,Q$ and $w \in {}_Q W_P$, the parabolics $P_w$ and $Q_w$ are associate.

    \subsubsection{}
    For a semi-standard parabolic subgroup $P$ of $G$, define
        \[
            \fa_P^* := X^*(P) \otimes_\Z \R ,\quad \fa_P := \Hom_\Z(X^*(P),\R).
        \]
    We endow $\fa_P$ with the Haar measure such that the lattice $\Hom(X^*(P),\Z)$ has covolume 1.
        
    Let $\fa_0 := \fa_{P_0}$ and $\fa_0^* := \fa_{P_0}^*$. 
        \[
            \epsilon_P := (-1)^{\dim \fa_P - \dim \fa_G}.
        \]

    Let $A_P$ denote the maximal central split torus of $M_P$. Then $\fa_P$ can also be identified with $X^*(A_P) \otimes_{\Z} \R$. When $P \subset Q$ are two semi-standard parabolic subgroups, then natural maps $P \hookrightarrow Q$ and $A_Q \hookrightarrow A_P$ induce a projection $\fa_P^* \to \fa_Q^*$ and an injectiion $\fa_Q^* \to \fa_P^*$.

    Let $P_0'$ be a minimal semi-standard parabolic subgroup, let $\Delta_{P_0'} \subset \fa_{P_0'}$ be the set of simple roots of the $A_{P_0}$ action on $\opn{Lie}(N_{P_0'})$. Let $P$ be a semi-standard parabolic subgroup, choose a minimal parabolic subgroup $P_0' \subset P$, then we denote by $\Delta_P$ the image of $\Delta_{P'_0}$ under the projection $\fa_{P_0'} \to \fa_P$. $\Delta_P$ can also be identified with the set of simple roots of $A_P$ action on $\opn{Lie}(N_P)$, in particular, $\Delta_P$ is independent of the choice of $P_0'$.

    \subsubsection{Iwasawa decomposition}
\label{sssec:Iwasawa_decomposition}

Let $K$ be a maximal compact subgroup of $G(\bA)$ which is in good position with respect to $P_0$. Then for any semi-standard parabolic subgroup $P$ of $G$, we have the \emph{Iwasawa decomposition} $G(\bA)=P(\bA)K$. 

When $G=G_n$, we denote by $K_n$ the usual maximal compact subgroup of $G_n(\bA)$. In the main text, we will sometimes use $H$ to denote the symplectic group $\Sp_{2n}$, and $K_H$ will denote the usual maximal compact subgroup of $\Sp_{2n}(\bA)$ accordingly. 

\begin{lemma}
    There exists measurable maps $G(\bA) \to P(\bA) \times K, g \mapsto (p(x),k(x))$ such that for any $g \in G(\bA)$, we have $g = p(g)k(g)$.
\end{lemma}

\begin{proof}
    Since $P(\bA) \times K$ is a Polish space, this follows from Kuratowski and Ryll-Nardzewski measurable selection theorem applied to the natural map $P(\bA) \times K \to G(\bA)$.
\end{proof}

We will sometimes refer to any function $p(g),k(g)$ as in the previous lemma a measurable (family of) Iwasawa decomposition.

For positive integers $k,n$, denote by $\opn{Mat}_{k \times n}(\bA)$ the set of matrices of size $k \times n$ with coefficients in $\bA$. For future use, we record the following estimate

\begin{lemma} \label{lem:Iwasawa_estimate}
    Let $n,k$ be positive integers. Fix $m$ such that $k \le m \le n+k$ , let $Q$ be the parabolic subgroup of $G_{n+k}$ with Levi factor $G_{n+k-m} \times (G_1)^m$. For any $x = \begin{pmatrix}
        x_1 \\ \cdots \\ x_k
    \end{pmatrix} \in \opn{Mat}_{k \times n}(\bA), x_i \in \bA_n$, assume that under the Iwasawa decomposition $G_{n+k}(\bA) = N_Q(\bA)M_Q(\bA)K_{n+k}$, we write
    \begin{equation} \label{eq:exchange_of_root_convergence_Iwasawa}
         \begin{pmatrix}
            1_n &  \\ x & 1_k
        \end{pmatrix} = u(x) \begin{pmatrix}
            g(x) &  \\ & t(x)
        \end{pmatrix} k(x),
    \end{equation}
    where $g(x) \in \GL_{n+k-m}(\bA)$ and $t(x) = \mathrm{diag}(t_1(x),\cdots,t_m(x))$. Then there exists $M>0$ such that
    \begin{num}
        \item \label{eq:lemma_Iwasawa_1} For $1 \le i \le k$, we have $ | t_{m+k-i}(x) \cdots t_m(x) | \gg \|x_i \|_{\bA_n}$,
        \item \label{eq:lemma_Iwasawa_2} $\|g(x)\|_{\GL_n(\bA)} \ll \| x \|^{M}_{\opn{Mat}_{k \times n}(\bA)}$, $\| t_i(x) \|_{\GL_1(\bA)} \ll \| x \|^M_{\opn{Mat}_{k \times n}(\bA)}$ .
    \end{num}
    holds for any $x \in \opn{Mat}_{k \times n}(\bA_n)$ and $1 \le i \le k$. 
\end{lemma}

\begin{remark}
    Since different choices of Iwasawa decomposition will yield right translation of $g(x)$ or $t(x)$ by elements of $K_{m+k}$, hence \eqref{eq:lemma_Iwasawa_1} is independent of the choice of Iwasawa decomposition and \eqref{eq:lemma_Iwasawa_2} holds for any choice of Iwasawa decomposition (after possibly enlarging constant).
\end{remark}

\begin{proof}
    \eqref{eq:lemma_Iwasawa_2} follows from \eqref{eq:Iwasawa_bound}. Now we prove \eqref{eq:lemma_Iwasawa_1}. Let $e_1,\cdots,e_{n+k}$ be the canonical basis for $F_{n+k}$.  The basis $e_i$ yield a canonical basis $\{e_I := \bigwedge_{i \in I} e_i \}_{I \subset \{1,\cdots,k\},\lvert I \rvert = i}$ for the exterior power $\bigwedge^i F_{n+k}$. For $\omega \in \bigwedge^i \bA_{n+k}$, write $\omega = \sum_I a_I e_I$, we define
    \begin{equation*}
        \lvert \omega \rvert = \prod_v \max_I\{ \lvert a_I \rvert_v \}.
    \end{equation*}
    For any $g \in G_{n+m}(\bA)$ and $\omega \in \bigwedge^i \bA_{n+k}$, we denote by $\omega \cdot g$ the natural action of $g$ on $\omega$. The absolute value $\lvert \cdot \rvert$ satisfies
    \begin{equation}
        \label{eq:exchange_of_root_convergence_lemma_3} 
      \lvert \omega \rvert \ll \lvert \omega \cdot k \rvert \ll \lvert \omega \rvert, \quad \forall \omega \in \bigwedge^i \bA_{n+k}, k \in K_{n+k}.
    \end{equation}
    Let $1 \le i \le k$. Consider $\omega_i := e_{n+i} \wedge \cdots \wedge e_{n+k} \in \bigwedge^{n-i} \bA_{n+k}$. Since $\omega_{k-i+1} \cdot u = \omega_i$ for any $u \in N_Q(\bA)$.  We can check that
    \begin{equation} \label{eq:exchange_of_root_convergence_lemma_1}
       \left \vert \omega_i \cdot \begin{pmatrix}
            1_n & \\ x & 1_k
        \end{pmatrix} \right \vert \gg \prod_v \max\{ \lvert x_{i,1} \rvert_v ,\cdots, \lvert x_{i,n} \rvert_v ,1 \} = \| x_i \|_{\bA_n}.
    \end{equation}
    By \eqref{eq:exchange_of_root_convergence_lemma_3}, applying right hand side of \eqref{eq:exchange_of_root_convergence_Iwasawa} to $\omega$ yields
    \begin{equation} \label{eq:exchange_of_root_convergence_lemma_2}
         \left \vert \omega_i \cdot \begin{pmatrix}
            1_n & \\ x & 1_k
        \end{pmatrix} \right \vert \ll \lvert t_{m+k-i}(x) \cdots t_m(x) \rvert.
    \end{equation}
    Combining \eqref{eq:exchange_of_root_convergence_lemma_1} and \eqref{eq:exchange_of_root_convergence_lemma_2} yields \eqref{eq:lemma_Iwasawa_1}.
\end{proof}

\subsubsection{\texorpdfstring{The map $H_P$}{The map H_P}}

 We denote by $A_G^\infty$ the neutral component of real points of the maximal split central torus of $\mathrm{Res}_{F/\Q} G$. For a semi-standard parabolic subgroup $P$ of $G$, let $A_P^\infty := A_{M_P}^\infty$. We also define $A_0^\infty := A_{P_0}^\infty = A_{M_0}^\infty$.

    The map
        \[
            H_P: P(\bA) \to \fa_P, \, p \mapsto \left( \chi \mapsto \log \lvert \chi(g) \rvert  \right), \quad \chi \in X^*(P),
        \]
    extends to $G(\bA)$, by requiring it trivial on $K$. The map $H_P$ induces an isomorphism $A_P^\infty \cong \fa_P$, we endow $A_P^\infty$ with the Haar measure such that this isomorphism is measure-preserving. 

\subsubsection{An estimate}

\begin{lemma} \label{lem:scaling_Theta_series}
    For every $k \ge n$, if $N$ is sufficiently large, we have
    \begin{equation*}
        \sum_{v \in F_n \setminus \{0\}} \| av \|_{\bA_n}^{-N} \ll \lvert a \rvert^{-k}, \quad a \in \bA^\times.
    \end{equation*}
\end{lemma}

\begin{proof}
    We write $a$ as $a^1t$, where $t \in \R_{>0}$ and $\lvert a^1 \rvert = 1$. Then
    \begin{equation*}
        \sum_{v \in F_n \setminus \{0\}} \| av \|_{\bA_n}^{-N} \ll  \sum_{v \in F_n \setminus \{0\}} \| tv \|_{\bA_n}^{-N} \| a^1 \|^N_{\bA^\times}
    \end{equation*}
    Since the LHS is invariant under $F^\times$, we have
     \begin{equation*}
        \sum_{v \in F_n \setminus \{0\}} \| av \|_{\bA_n}^{-N} \ll  \sum_{v \in F_n \setminus \{0\}} \| tv \|_{\bA_n}^{-N} \| a^1 \|^N_{\mathbb{G}_m}
    \end{equation*}
    Since $[\mathbb{G}_m]^1$ is compact, $\| a^1 \|_{\mathbb{G}_m}^N$ is bounded. Therefore we are reduced to the case when $a \in \R_{>0}$, in which case, it is proved in \cite[(2.6.2.6)]{BPCZ}.
\end{proof}

\begin{corollary} \label{cor:A_norm_integral}
    For any $c>1$, there exists $N_0$ such that for any $N \ge N_0$, the integral 
    \begin{equation*}
        \int_{\bA^\times} \| x \|_{\bA}^{-N} \lvert x \rvert^s \rd x 
    \end{equation*}
    is absolutely convergent for $1 < \opn{Re}(s) < c$.
\end{corollary}

\begin{proof}
    We write $[\GL_1]^{\le 1}$ (resp. $[\GL_1]^{\ge 1}$) for the elements $x \in [\GL_1]$ such that $\lvert x \rvert \le 1$ (resp. $\ge 1$). We write the integral as
    \begin{equation*}
        \int_{[\GL_1]} \sum_{v \ne 0} \lVert  vx \rVert_{\bA}^{-N} \lvert x \rvert^s \rd x.
    \end{equation*}
    By Lemma \ref{lem:scaling_Theta_series},  when $N$ is sufficiently large, it is essentially bounded by
    \begin{equation*}
        \int_{[\GL_1]^{\le 1}} \lvert x \rvert^{s-1} \rd x + \int_{[\GL_1]^{\ge 1}} \lvert x \rvert^{s-c} \rd x.
    \end{equation*}
    This is finite when $1<\opn{Re}(s)<c$.
\end{proof}

\subsection{Spaces of functions}

\subsubsection{}
\label{sssec:prelim_function_spaces}

There are various function spaces on $[G]_P$ which we briefly recall below. The reader may consult \cite[\S 2.5]{BPCZ} for more details.

    A function $f:G(\bA) \to \C$ is called \emph{smooth}, if it is right $J$-invariant for some open compact subgroup $J \subset G(\bA_f)$ and for any $g_f \in G(\bA_f)$, the function $g_\infty \mapsto f(g_f g_\infty)$ is $C^\infty$. A function on $[G]_P$ is called smooth if it pulls back to a smooth function on $G(\bA)$.

    Let $\cS([G]_P)$ be the space of \emph{Schwartz functions} on $[G]_P$. It is the union of $\cS([G]_P,J)$ for open compact subgroup $J \subset G(\bA_f)$. Where $\cS([G]_P,J)$ is the space of smooth functions on $[G]_P$ which are right $J$ invariant and 
    \[
        \| f \|_{X,N} := \sup_{x \in [G]_P} \lvert R(X)f(x) 
        \rvert \| x \|_P^N < \infty
    \]
    for any $X \in \cU(\fg_{\infty})$ and $N>0$. The vector space $\cS([G]_P,J)$ is naturally a Fr\'{e}chet space and $\cS([G]_P)$ is naturally a strict LF space.

    For $N>0$, let $\cS_N([G]_P)$ be the set of smooth functions $f$ on $[G]_P$ such that $\| f \|_{X,N} < \infty$ for all $X \in \cU(\fg_{\infty})$. It is also a natural LF space.

    Let $\cS^0([G]_P)$ be the space of measurable function $f$ on $[G]_P$ such that
    \begin{equation} \label{eq:norm_infty_N}
       \| f\|_{\infty,N} := \sup_{x \in [G]_P} \lvert f(x) \rvert \| x \|^{N}_P < \infty
    \end{equation}
    for any $N>0$. It is naturally a Fr\'{e}chet space.

Let $\cT([G]_P)$ be the function of \emph{uniform moderate growth} on $[G]_P$. It is the union of $\cT_{N}([G]_P,J)$, where $N>0$ and $J \subset G(\bA_f)$ is open compact subgroup. $\cT_{N}([G]_P,J)$ consists of smooth functions $f$ on $[G]_P$ which are right $J$-invariant and
\begin{equation*}
    \| f \|_{X,-N} := \sup_{x \in [G]_P} \lvert R(X)f(x) \rvert \| x \|_P^{-N} < \infty
\end{equation*}
for any $X \in \cU(\fg_\infty)$. The vector space $\cT_{N}([G]_P,J)$ is naturally a Fr\'{e}chet space and $\cT([G]_P)$ then carries the induced (non-strict) LF topology.

For a Hilbert representation $V$ of $G(\bA)$, we write $V^{\infty}$ for the set of smooth vectors, i.e. the set $v \in V$ that is fixed by a compact open subgroup of $G(\bA_f)$ and is a smooth vector as $G(F_{\infty})$ representation. For each compact open subgroup $J \subset G(\bA_f)$, the vector space $V^{\infty,J}$ carries the usual Fr\'{e}chet topology (for smooth vectors in a Lie group representation). We endow $V^{\infty} = \bigcup_{J} V^{\infty,J}$ the LF topology.

For an integer $N$, we write $L^2_N([G]_P)$ for the weighted $L^2$ space consisting of measurable functions $f$ on $[G]_P$ such that 
\begin{equation*}
    \int_{[G]_P} \lvert f(x) \rvert^2 \| x \|_P^N \rd x < \infty.
\end{equation*}
Let $L^2_N([G]_P)^{\infty}$ be the set of smooth vectors. By Sobolev lemma \cite[\S 3.4, Key Lemma]{Bernstein}, we have
\begin{num}
    \item \label{eq:L2N_Schwartz_enbeddimgs} For each $N>0$ there exists $N'>0$ such that we have closed embedding of topological vector spaces
    \begin{equation*}
        L^2_N([G]_P)^{\infty} \hookrightarrow S_{N'}([G]_P), \quad \cS_{N}([G]_P) \hookrightarrow L^2_{N'}([G]_P)^{\infty}.
    \end{equation*}
\end{num}

\subsubsection{Constant terms}
For $P \subset Q$, we have the following \emph{constant term} map
\begin{equation*}
    \cT([G]_Q) \ni f \mapsto f_P := \left( g \mapsto \int_{[N_P]} f(ng) \rd n \right) \in \cT([G]_P).
\end{equation*}

We recall the following useful estimate of constant term of a Schwartz function \cite[Lemma 2.5.13.1]{BPCZ}
\begin{lemma} \label{lemma: estimate constant term}
    Let $P$ be a parabolic subgroup of $G$. Then there is a constant $c > 0$ such that for every $N \geq 0$,
    \[
        f \mapsto \sup_{x \in [G]_P} \delta_P(x)^{cN}\|x\|_P^N |f_P(x)|
    \]
    is a continuous semi-norm on $\cS([G])$.
\end{lemma}
As a direct consequence, we obtain
\begin{num} 
    \item \label{eq: truncation of constant term}
    Let $P$ be a standard parabolic subgroup of $G$ and $\kappa \in C_c^\infty(\fa_P)$ be a compactly supported smooth function on $\fa_P$. Then for every $f \in \cS([G])$, we have
    \[
        (\kappa \circ H_P) \cdot f_P \in \cS([G]_P).
    \]
\end{num}

Also, combining Lemma \ref{lemma: estimate constant term} with \eqref{eq:L2N_Schwartz_enbeddimgs}, we obtain:
\begin{num}
    \item \label{eq:constant_term_L^2_N}
    For $N>0$. There exists $c_N>0$ such that for any $s \in \C$ with $\opn{Re}(s) > c_N$, we have 
    \begin{equation*}
        [M_P] \ni m \mapsto f_P(m) \delta_P(m)^s \in L^2_N([M_P])^{\infty}.
    \end{equation*}
\end{num}

The following two variable versions of the constant term estimates follows from the same proof of \cite[Lemma 2.5.13.1]{BPCZ}

\begin{lemma} \label{lem:constant_term_estimate_product}
    Let $G,H$ be connected reductive groups over $F$. Let $P \times Q$ be a parabolic subgroup of $G \times H$. Then there exists $c>0$ such that for any $M,N \ge 0$, 
    \begin{equation*}
        f \mapsto \sup_{(x,y) \in [G \times H]_{P \times Q}} \delta_{P}(x)^{cN} \delta_Q(y)^{cM } \| x \|_P^{N} \| y \|_Q^M \lvert f_{P \times Q}(x,y) \rvert
    \end{equation*}
    is a continuous semi-norm on $\cS([G \times H])$.
\end{lemma}

\subsubsection{Polarized $\Theta$-series}
Let $\Phi \in \cS(\bA_n)$. We associate the following \emph{$\Theta$-series} :
\begin{equation} \label{eq:Theta_series}
    \Theta(g, \Phi)= \sum_{v \in F_n} \Phi(vg) \lvert g \rvert^{\frac 12}, \quad g \in [\GL_n]
\end{equation}

The factor $\lvert g \rvert^{\frac 12}$ appears because the action of $\GL_n(\bA)$ on $\cS(\bA_n)$ given by $ (g \cdot \Phi)(v) = \Phi(vg) \lvert g \rvert^{\frac 12}$ is unitary.

The convergence and the growth of the $\Theta$-series are justified by the following lemma
\begin{lemma} \label{Theta moderate}
    There exists $M>0$ and $N_0 >0$, such that for every $N \geq N_0 $, we have
    \begin{equation} \label{eq:Theta_moderate}
        \sum_{v \in F_n} \|vh \|_{\bA_n}^{-N} \ll \|h\|_{G_n}^{M}.
    \end{equation}
    In particular, there exists $N_0>0$ such that for any $\Phi \in \cS(\bA_n)$, we have $\Theta( \cdot,\Phi) \in \cT_{M}([G_n])$.
\end{lemma}

\begin{proof}
    Note that the left-hand side of \eqref{eq:Theta_moderate} is decreasing in $N$, so it suffices to find $N=N_0$ such that \eqref{eq:Theta_moderate} holds. There exists $c>0$ such that 
    \begin{equation*}
        \| vh \|^{-N}_{\bA_n} \ll \| v \|_{\bA_n}^{-cN} \| h \|_{G_n(\bA)}^{cN}
    \end{equation*}
    holds for any $v \in \bA_n$ and $h \in G_n(\bA)$. It then suffices to pick $N_0>0$ such that $\sum_{v \in F_n} \| v \|_{\bA_n}^{-cN_0} < \infty$.
\end{proof}

\begin{corollary} \label{cor:Theta_series_integral}
    For any $C>0$, there exists $N_0>0$, such that for any $N,N'>N_0$, the integral
    \begin{equation*}
        \int_{\cP_n(F) \backslash G_n(\bA)} \lVert e_n h \rVert^{-N}_{\bA_n} \lvert \det h \rvert^s \|h\|_{G_n}^{-N'} \rd h
    \end{equation*}
    converges for $\lvert \opn{Re}(s) \rvert<C$.
\end{corollary}

\begin{proof}
    The integral can be written as
    \begin{equation*}
        \int_{[G_n]} \sum_{v \ne 0} \lVert  vh \rVert_{\bA_n}^{-N} \lvert \det h \rvert^s \| h \|_{G_n}^{-N'} \rd h. 
    \end{equation*}
    By Lemma \ref{Theta moderate} and the fact that
    \begin{equation*}
       \max\{ \lvert \det h \rvert, \lvert \det h \rvert^{-1} \} \ll  \lVert h \rVert_{G_n}^r 
    \end{equation*}
    for some $r>0$, we see that the integral is bounded by
    \begin{equation*}
        \int_{[G_n]}   \|h\|_{G_n}^{-N'+M + \opn{Re}(s)r} \rd h,
    \end{equation*}
    for some $M>0$, the result follows.
\end{proof}

We also remark that, by the Poisson summation formula, $\Theta$-series satisfies
\begin{equation} \label{eq:Theta_series_functional_equation}
    \Theta(g,\Phi) = \Theta({}^tg^{-1},\widehat{\Phi}).
\end{equation}

\subsubsection{Estimates on Fourier coefficients}
Let $P \subset G$ be a standard parabolic subgroup, $\psi : \bA /F \to \mathbb{C}^\times$ be a non-trivial character and $V_P$ be the vector space of additive algebraic characters $N_P \to \mathbb{G}_a$. Let $ l \in V_P(F)$ and set $\psi_l := \psi \circ l_\bA : [N_P] \to \mathbb{C}^\times$ where $l_\bA$ denotes the homomorphism between adelic points $N_P(\bA) \to \bA$. For $\varphi \in \cT([G])$, we set
\[
    \varphi_{N_P, \psi_l}(g) =\int_{[N_P]}\varphi(ug) \psi_l(u)^{-1} \rd u, \quad g \in G(\bA).
\]
The adjoint action of $M_P$ on $N_P$ induces one on $V_P$ that we denote by $\opn{Ad}^*$.
\begin{lemma} \label{lemma: estimate Fourier coeff} \cite[Lemma 2.6.1.1]{BPCZ}
    \begin{enumerate}
        \item \label{Schwartz Fourier coefficient} There exists $c > 0$ such that for every $N_1, N_2 \geq 0$, 
        \[
            \varphi \mapsto  \sup_{m \in M_P(\bA)} \sup_{k \in K} \| \opn{Ad}^*(m^{-1})l \|^{N_1}_{V_P(\bA)} \| m \|_{M_P}^{N_2} \delta_P(m)^{cN_2} \lvert \varphi_{N_P, \psi_l}(mk) \rvert 
        \]
        is a continuous semi-norm on $\cS([G])$.
        \item \label{T Fourier coeff} Let $N >0$. Then, for every $N_1 \geq 0$,
        \[
            \varphi \mapsto \sup_{m \in M_P(\bA)} \sup_{k \in K} \| \opn{Ad}^*(m^{-1})l \|^{N_1}_{V_{P}(\bA)} \| m \|_{M_P}^{-N} \lvert \varphi_{N_P, \psi_l}(mk) \rvert
        \]
        is a continuous semi-norm on $\cT_N([G])$.
    \end{enumerate}
\end{lemma}

\begin{proof}
    Without the term $\sup_{k \in K}$, this is exactly \cite[Lemma 2.6.1.1]{BPCZ}. Since for any continuous semi-norm $\| \cdot \|$ on $\cS([G])$ or $\cT_N(G)$, $f \mapsto \sup_{k \in K} \| R(k) f \|$ is still a continuous semi-norm. The result follows.
\end{proof}

\subsection{Automorphic forms and Eisenstein series}

\subsubsection{Automorphic forms}

Let $G$ be a connected reductive group over $F$ and let $P$ be a standard parabolic subgroup. An \emph{automorphic form} on $[G]_P$ is, by definition, is a $Z(\fg_{\infty})$-finite function in $\cT([G]_P)$. We denote by $\cA_P(G)$ the set of automorphic form on $[G]_P$.

Let $\cA_{P,\mathrm{cusp}}(G)$ denote the subspace of $\cA_P(G)$ consisting of cuspidal automorphic forms, that is, consisting of $\varphi \in \cA_P(G)$ such that $\varphi_Q=0$ for any standard parabolic subgroups $Q \subset P$.

Let $\cJ$ be a finite codimensional ideal of $\cZ(\fg_{\infty})$. Let $\cA_{P,\cJ}(G)$ denote the subspace of automorphic form $\varphi \in \cA_P(G)$ such that $R(z) \varphi=0$ for all $z \in \cJ$. Then there exists $N>0$ such that $\cA_{P,J}(G)$ is a closed subspace of $\cT_N([G])$. We endow $\cA_{P,J}$ with the topology induced from $\cT_N(G)$. This topology is independent of the choice of $N$. We then endow $\cA_{P}(G)$ with the inductive limit topology $\cA_P(G) = \bigcup_{\cJ} \cA_{P,\cJ}(G)$. For each $\cJ$, the inclusion $\cA_{P,\cJ}(G) \hookrightarrow \cA_P(G)$ is a closed embedding. We refer the reader to \cite[\S 2.7.1]{BPCZ} for the proof of these facts.

A \emph{cuspidal automorphic representation} of $G$ is defined to be a topologically irreducible subrepresentation $\pi$ of $G(\bA)$ on $\cA_{\mathrm{cusp}}(G)$. Note that a cuspidal automorphic representation $\pi$ is unitary if and only if any $\varphi \in \pi$ has a unitary central character, in the sense that for any $z \in G(\bA)$, $\varphi(zg)=\varphi(g) \omega(z)$ for some unitary character $\omega: Z_G(\bA) \to \C^{\times}$, where $Z_G$ denote the center of $G$.

\subsubsection{Eisenstein series}

Let $P=M_PN_P$ be a standard parabolic subgroup of $G$. Let $\pi$ be a cuspidal automorphic representation of $M_P$. Let $\cA_{\pi,\mathrm{cusp}}(M_P)$ denote the the sum of all cuspidal automorphic representations of $M_P(\bA)$ that are isomorphic to $\pi$.

We write $\opn{Ind}_{P(\bA)}^{G(\bA)} \pi$ (resp. $\cA_{P,\pi}$) for the subspace
\begin{equation*}
    \{ \varphi \in \cA_P(G) \mid \text{ for any }g \in G(\bA), m \mapsto \delta_P^{-\frac 12}(m) \varphi(mg) \in \pi \text{ (resp. $\cA_{\pi,\mathrm{cusp}}(M_P)$)} \}.
\end{equation*}
of $\cA_P(G)$. 

For $\varphi \in \cA_{P,\pi}$ and $\lambda \in \fa_{P,\C}^*$, we define the \emph{Eisenstein series} as
\begin{equation*}
    E(g,\varphi,\lambda) = \sum_{\gamma \in P(F) \backslash G(F)} e^{\langle \lambda, H_P(\gamma g) \rangle} \varphi(\gamma g).
\end{equation*}
The series is absolutely convergent when $\opn{Re}(\lambda)$ lies in some cones and by \cite{BL24}, \cite{Lapid08}, it has meromorphic continuation to $\fa_{P,\C}^*$. When $\pi$ is unitary, for $\varphi \in \cA_{P,\pi}$, the Eisenstein series $E(g,\varphi,\lambda)$ is regular when $\lambda \in i\fa_P^*$.

\subsubsection{Intertwining operators and normalizations}
\label{sssec:prelim_intertwining_operator}

Let $P$ and $Q$ be standard parabolic subgroups of $G$. For any $w \in W(P, Q)$ and $\lambda \in \fa_{P,\C}^*$, the \emph{intertwining operator}
\begin{equation*}
    M(w,\lambda): \cA_P(G) \to \cA_Q(G)
\end{equation*}
 is defined by the meromorphic continuation of the integral
\begin{align*}
    (M(w, \lambda)\varphi)(g)&= \opn{exp}(-\langle w\lambda, H_P(g) \rangle)\\
    & \times \int_{(N_Q\cap wN_P w^{-1})(\bA)\backslash N_Q(\bA)}\opn{exp}(\langle \lambda, H_P(w^{-1}ng) \rangle)\varphi(w^{-1}ng) \rd n.
\end{align*}
(see \cite{BL24} for the meromorphic continuation). 
Let $\pi$ be a cuspidal representation of $M_P$, we denote by $M_\pi(w, \lambda)$ the restriction of $M(w, \lambda)$ to the subspace $\cA_{P, \pi}(G) \subset \cA_{P}(G)$. It is known that if $\pi$ is unitary, then $M_{\pi}(w,\lambda)$ is regular on $i \fa_{P}^*$.

Now we assume $G$ is $G_n$ and write $M_P=G_{n_1} \times \cdots \times G_{n_k}$ and $\pi= \pi_1 \boxtimes \cdots \boxtimes \pi_k$. Let $\Sigma^+_P \subset X^*(A_P)$ denote the set of positive roots of $A_P$ action on $\fn_P$. Let $\beta \in \Sigma^+_P$ be the positive root of $P$ associated to the two blocks $G_{n_i}$ and $G_{n_j}$ with $1 \leq i < j \leq k$. Set
\begin{equation*}
    n_{\pi}(\beta, s)=\frac{L(s, \pi_i \times \pi_j^\vee)}{\epsilon(s, \pi_i \times \pi_j^\vee)L(1+s, \pi_i \times \pi_j^\vee)}= \frac{L(1-s, \pi_i^\vee \times \pi_j)}{L(1+s, \pi_i \times \pi_j^\vee)},
\end{equation*}
then we define
\begin{equation*}
     n_{\pi}(w, \lambda),= \prod_{\substack{
  \beta \in \Sigma_P \\
  w\beta < 0
}} n_\pi(\beta, \langle \lambda, \beta^\vee \rangle).
\end{equation*}

Following \cite{MW89}, we normalize $M(w, \lambda)$ as
\begin{equation} \label{eq: normalized intertwining operator}
    M_\pi(w, \lambda)=n_{\pi}(w, \lambda) N_\pi(w, \lambda).
\end{equation}
Let $\varphi \in \opn{Ind}_{P(\bA)}^{G(\bA)} \pi$. Assume that $\varphi = \otimes'_v \varphi_v$ is factorizable, where $\varphi_v \in \opn{Ind}_{P(F_v)}^{G(F_v)}\pi_v$. Let $\tS$ be a sufficiently large finite set of places of $F$, which we assume to contain Archimedean places as well as the places $\varphi_v$ is ramified. We have a factorization
\begin{equation} \label{eq: decompose normalized intertwine}
    N_\pi(w, \lambda)\varphi= \prod_{v \in \tS}N_{\pi_v}(w, \lambda)\varphi_v.
\end{equation}
Here $N_{\pi_v}(w, \lambda)$ is the meromorphic \emph{local normalized intertwining operator} $\opn{Ind}_{P(F_v)}^{G(F_v)}\pi_{v,\lambda} \to \opn{Ind}_{Q(F_v)}^{G(F_v)}(w\pi)_{v,\lambda}$, 
(see \cite{MW89}). The product notation of \eqref{eq: decompose normalized intertwine} means $N_\pi(w, \lambda)\varphi$ is factorizable and for $v \notin \tS$ the local component $N_{\pi_v}(w, \lambda)\varphi_v$ is the unique unramified vector in $\opn{Ind}_{Q(F_v)}^{G(F_v)}(w\pi)_v$ such that $N_{\pi_v}(w, \lambda)\varphi_v(1)$ corresponds to $\varphi_v(1)$ under the natural identification between $\pi_v$ and $(w\pi)_v$. 

The following result is taken from \cite[Page 607]{MW89}
\begin{lemma}
    Let $\pi_v$ be a smooth irreducible and unitary representation of $M_P(F_v)$. Then the operator $N_{\pi_v}(w, \lambda)$ is holomorphic and unitary if $\lambda \in i \fa_P^*$. It is an isomorphism.
\end{lemma}

From now on, we simply write $N_{\pi}(w)$ for $N_{\pi}(w,0)$. If we put
\begin{equation} \label{eq:L_function_on_n_P_-}
    L(s,\pi,\widehat{\fn}_P^-) := \prod_{1 \le i <  j \le k} L(s,\pi_i \times \pi_j^{\vee}),
\end{equation}
then we have
\[
    n_\pi(w,0)=\frac{L(1,w\pi,\widehat{\fn}_Q^-)}{L(1,\pi,\widehat{\fn}_P^-)}.
\]
For sufficiently large $\tS$ as above, we denote by
\[
    N_{\pi, \tS}(w)=\prod_{v \in \tS}N_{\pi_v}(w) : \opn{Ind}_{P(F_\tS)}^{G(F_\tS)}\pi_\tS \to \opn{Ind}_{Q(F_\tS)}^{G(F_\tS)}(w\pi)_\tS.
\]

We finally remark that the normalized intertwining operator naturally extends to the case when $G$ is a product of $G_{n_i}$.

\subsection{Langlands spectral decomposition}
\label{ssec:Langlands_spectral_decomposition}

\subsubsection{Cuspidal datum}
Let $G$ be a connected reductive group over $F$. Let $\underline{\fX}(G)$ denote the set of pairs $(M_P,\pi)$, where $M_P$ is the Levi component of a standard parabolic subgroup $P$ and $\pi$ is a cuspidal automorphic representation of $M_P(\bA)$ with central character trivial on $A_P^{\infty}$. Two elements $(M_P,\pi)$ and $(M_Q,\pi')$ of $\underline{\fX}(G)$ are called equivalent, if there exists $g \in G(F)$ such that $gM_Pg^{-1}=M_Q$ and $g \pi = \pi'$. Let $\fX(G)$ denote the equivalence class of $\underline{\fX}(G)$, an element of $\fX(G)$ will be called a \emph{cuspidal data}.

For a standard parabolic subgroup $P \subset G$, there exists a natural map $\underline{\fX}(M_P) \to \underline{\fX}(G)$, and it induces a map $\fX(M_P) \to \fX(G)$ which has finite fiber. For each subset $\fX \subset \fX(G)$, we will write $\fX^M$ for its preimage in $\fX(M_P)$.

\subsubsection{Langlands decomposition}

For $\chi \in \underline{\fX}(G)$, and $P$ be a standard parabolic subgroup, we write $\fO^P_{\chi} \subset \cS([G]_P)$ the set of \emph{pseudo-Eisenstein series} with respect to $\chi$ (See \cite[\S II.1]{MW95}, \cite[\S 2.9]{BPCZ}). Let $L^2_\chi([G]_P)$ denote the closure of $\fO^P_{\chi}$ in $L^2([G]_P)$, then we have the following coarse Langlands decomposition:
\begin{equation} \label{eq:Langlands_spectral_decomposition}
    L^2([G]_P) = \bigoplus_{\chi \in \fX(G)} L^2_{\chi}([G]_P).
\end{equation}

For a subset $\fX \subset \fX(G)$, we write $L^2_{\fX}([G]_P) := \widehat{\bigoplus}_{\chi \in \fX} L^2_{\chi}([G]_P)$. Then we define
\begin{equation} 
    \cS_{\fX}([G]_P) := L^2_{\fX}([G]_P) \cap \cS([G]_P).
\end{equation}
Note that $\cS_{\fX}([G]_P)$ is a closed subspace of $\cS([G]_P)$, since it is orthogonal complement of $\bigcup_{\chi \not \in \fX} \fO_\chi$ in $\cS([G]_P)$. 

We then define $\cT_{\fX}([G]_P)$ (resp. $L^2_N([G]_P)^{\infty}$) be the orthogonal complement of $\cS_{\fX^c}([G]_P)$ in $\cT([G]_P)$ (resp. $L^2_N([G]_P)^{\infty}$). 

We call element of $\cS_{\fX}([G]_P)$ the set of \emph{Schwartz function with cuspidal support in $\fX$} and $\cT_{\fX}([G]_P)$ the set of \emph{uniform moderate growth function with cuspidal support in $\fX$}. For any subset $\fX \subset \fX(G)$, the space $\cS_{\fX}([G]_P)$ is dense in $\cT_{\fX}([G]_P)$ (see \cite[\S 2.9.5]{BPCZ})

The following theorem \cite[Theorem 2.9.4.1]{BPCZ} describes the decomposition of a function according to cuspidal support:

\begin{theorem}[Beuzart-Plessis-Chaudouard-Zydor] \label{thm:decomposition_cuspidal_support}
    We have the following statements:
    \begin{enumerate}
        \item For $f \in \cS([G]_P)$, let $f_{\chi}$ denote the $\chi$-part of $f$ under the decomposition \eqref{eq:Langlands_spectral_decomposition}, then $f_{\chi} \in \cS([G]_P)$ and $f = \sum_{\chi} f_{\chi}$, where the sum is absolutely summable in $\cS([G]_P)$.
        \item The map $f \mapsto f_{\chi}: \cS([G]_P) \to \cT([G]_P)$ extends by continuity to a map $\cT([G]_P) \to \cT([G]_P)$, which we still denote by $f \mapsto f_{\chi}$. Then for any $f \in \cT([G]_P)$, $f_{\chi} \in \cT_{\chi}([G]_P)$ and the sum $f = \sum_{\chi} f_{\chi}$ is absolutely summable in $\cT([G]_P)$. 
    \end{enumerate}
\end{theorem}

\subsubsection{Some lemmas}

\begin{lemma} \label{lemma: pseudo dense}
    For each $\chi \in \fX(G)$, $\fO_{\chi}^P$ is dense in $\cS_{\chi}([G]_P)$ and $\cT_{\chi}([G]_P)$.
\end{lemma}

\begin{proof}
    See \cite[Lemma 5.5.1.2]{Boisseau25} for the density in $\cS_{\chi}([G]_P)$, the density in $\cT_{\chi}([G]_P)$ also follows, since $\cS_{\chi}([G]_P)$ is dense in $\cT_{\chi}([G]_P)$.
\end{proof}

\begin{lemma} \label{lemma: constant term}
Let $\chi \in \fX(G)$ be a cuspidal datum and $P$ be a standard parabolic subgroup of $G$. Then we have
\[
    E_P^G(\cS_\chi([G]_P)) \subset \cS_\chi([G]), \quad \cT_\chi([G])_P \subset \cT_\chi([G]_P).
\]
\end{lemma}

\begin{proof}
    See \cite[Lemma 2.9.3.1]{BPCZ}.
\end{proof}

\begin{lemma} \label{lemma: restriction to levi}
    Let $\chi \in \fX(G)$ be a cuspidal datum, $P$ be a standard parabolic subgroup of $G$ and $\chi_M$ be the inverse image of $\chi$ in $\fX(M_P)$. Then, for every $f \in \cS_\chi([G]_P)$(resp. $\cT_\chi([G]_P)$), its restriction $f|_{[M_P]}$ to $[M_P]$ belongs to $\cS_{\chi^M}([M_P])$(resp. $\cT_{\chi^M}([M_P])$).
\end{lemma}

\begin{proof}
    If $f \in \fO^P_{\chi}$, this follows from the definition, and the general cases follow by the density (Lemma \ref{lemma: pseudo dense}).
\end{proof}

\begin{lemma} \label{lemma: restriction to component}
    Assume that $G= H \times L$, where $H$ and $L$ are connected reductive groups over $F$. Then we have a natural identification $\fX(G)= \fX(H) \times \fX(L) $. For a subset $\fX \subset \fX(G)$, denote its projection to $\fX(H)$ by $\fX_H$. Then for every $\cS_{\fX}([G])$, its restriction to $[H]$ belongs to $\cS_{\fX_H}([H])$.
\end{lemma}

The proof is the same as the proof of Lemma \ref{lemma: restriction to levi}.

\begin{lemma} \label{lemma: stable under truncation}
    Let $\chi \in \fX(G)$ be a cuspidal datum, $P$ be a standard parabolic subgroup of $G$, and $\kappa \in C_c^\infty(\fa_P)$ be a compactly supported smooth function on $\fa_P$. Then $\cT_\chi([G]_P)$ is stable under the multiplication by $\kappa \circ H_P$.
\end{lemma}

\begin{proof}
    This can also be proved via the method in the proof of Lemma \ref{lemma: restriction to levi}. Alternatively, for $f \in \cT_{\chi}([G]_P)$, we need to show that $f \cdot (\kappa \circ H_P)$ is orthogonal to any $f' \in \cS_{\chi'}([G]_P)$ for $\chi' \ne \chi$. Then
    \begin{equation*}
        \langle (\kappa \circ H_P)f, f' \rangle_{[G]_P} = \langle (f, (\kappa \circ H_P)f' \rangle_{[G]_P}.
    \end{equation*}
    Therefore, it reduces to proving $\cS_{\chi}$ is stable under multiplication by $\kappa \circ H_P$. Since $\cS_{\chi}$ is orthogonal to $\fO_{\chi'}$ to all $\chi' \ne \chi$. By the same trick, it reduces to proving each $\fO_{\chi}$ is stable under multiplication by $(\kappa \circ H_P)$, which follows from the definition.
\end{proof}

\subsection{Whittaker model}
\label{ssec:Whittaker}

\subsubsection{Local Whittaker model}
\label{sssec:prelim_local_Whittaker}

We now assume that $F$ be a local field. Let $G$ be a quasi-split group over $F$. We fix a splitting $\opn{Spl} = (B, T,\{ X_{\alpha} \}_{\alpha})$ of $G$. This means $B=TU$ is an $F$-Borel subgroup, $T$ is a maximal torus and $\{X_{\alpha}\}_{\alpha}$ is a set of $\Gamma_F$ invariant root vector.

We fix a splitting $\opn{Spl}$ of $G$ and an additive character $\psi: F \to \C^{\times}$. They give rise to a Whittaker data $\fw = \fw_{(\opn{Spl},\psi)} = (B,\psi_U)$ of $G$. More generally, for any Levi subgroup $M$ containing $T$, they give rise to a Whittaker data $\fw_M = \fw_{M,\opn{Spl},\psi} = (B_M, \psi_{U_M})$ of $M$, where $B_M := B \cap M$ and $\psi_{U_M} : U_M := U \cap M \to \C^{\times}$ is the character induced by $\opn{Spl}$ and $\psi$.

Let $\pi$ be an irreducible representation of $G(F)$. Recall that $\pi$ is called \emph{generic}, if it satisfies $\Hom_{U(F)}(\pi,\psi_U) \ne 0$. When $\pi$ is generic, it can be identified with its \emph{Whittaker model} $\cW(\pi,\psi_U)$. Recall that
\begin{equation*}
    \cW(\pi,\psi_U) = \{ g \mapsto \lambda(\pi(g)v) \mid v \in \pi \} \subset C^{\infty}(U(F) \backslash G(F),\psi_U),
\end{equation*}
where $\lambda$ is any non-zero element of $\Hom_{U(F)}(\pi,\psi_U) \ne 0$.

When $G=G_{n_1} \times \cdots \times G_{n_k}$ is a product of general linear groups, we always fix the standard splitting.

\subsubsection{Jacquet integral}
\label{sssec:prelim_Jacquet_integral}

Let $P=M_PN_P$ be a parabolic subgroup and let $\pi$ be an irreducible generic representation of $M_P(F)$.

\begin{num}
    \item $\opn{Ind}_{P(F)}^{G(F)} \pi$ can be identified with the following space of functions on $G(F)$.
    \begin{equation*}
        \left\{ W^{M_P}: G(F) \to \C \mid \forall g \in G(F), m \in M(F) \mapsto \delta_P^{-\frac 12}(m) W^M(mg) \in \cW(\pi,\psi_{U_M}) \right \}.
    \end{equation*}
    We denote this space by $\opn{Ind}_{P(F)}^{G(F)}(\cW(\pi,\psi_{U_M}))$
\end{num}

Let $N_{\overline{P}}$ be the unipotent radical of the parabolic subgroup $\overline{P}$ of $G$ opposite to $P$. Let $w_0=w_{\ell} w_{\ell}^{P}$, where $w_{\ell}$ and $w_{\ell}^M$ are the longest elements in $W$ and $W^P$, respectively. Denote by $\widetilde{w_0} \in G(F)$ the \emph{Tits lifting} \cite[p. 228]{LS87} of $w_0$ and let $N'=\widetilde{w_0}N_{\overline{P}}\widetilde{w_0}^{-1}$.
The \emph{Jacquet functional} is given by the holomorphic continuation of the Jacquet integral
\begin{equation*}
    \Omega_{\lambda}(W^{M_P})(g) = \int_{N'(F)} W^{M_P}(\widetilde{w_0}^{-1}n'g) \psi_U(n')^{-1} \rd n'.
\end{equation*}
 It induces an isomorphism between $\opn{Ind}_{P(F)}^{G(F)}\cW(\pi_{\lambda},\psi_{U_M})$ and $\cW(\opn{Ind}_{P(F)}^{G(F)}\pi_{\lambda},\psi_U)$, where $\lambda \in \fa_{P,\C}^*$, and $\pi_\lambda$ denotes the unramified twist of $\pi$ by $\lambda$. When $\lambda=0$, we may simply denote $\Omega_{\lambda}$ by $\Omega$.

More generally, let $P \subset Q$ be standard parabolics. Let $w_0^Q := w_{\ell}^{Q} w_{\ell}^{P}$. Let $N' = \widetilde{w_0^Q} N_{\overline{P}} \widetilde{w_0^Q}^{-1} \cap M_Q$. Then we have a Jacquet functional $\Omega^Q$ from $\opn{Ind}_{P(F)}^{G(F)} \cW(\pi,\psi_{U_{M_P}}) \to \opn{Ind}_{Q(F)}^{G(F)} \cW(\opn{Ind}_{P}^Q \pi,\psi_{U_{M_Q}})$, defined by the meromorphic continuation of
\begin{equation*}
    \Omega_{\lambda}^Q(W^{M_P})(g) =  \int_{N'(F)} W^{M_P}(\widetilde{w_0}^{-1}n'g) \psi_{U_{M_Q}}(n')^{-1} \rd n'.
\end{equation*}

\subsubsection{}

Let $G$ be a product of $G_{n_i}$. Let $P,Q$ be standard parabolics $G$, and $w \in W(P,Q)$. Let $\pi$ be a generic representation of $M_P(F)$. The normalized intertwining operator $N(w,\lambda): \opn{Ind}_{P(F)}^{G(F)} \pi_{\lambda} \to \opn{Ind}_{Q(F)}^{G(F)} (w\pi)_{w\lambda}$ transports to a map $\opn{Ind}_{P(F)}^{G(F)} \cW(\pi_{\lambda},\psi) \to \opn{Ind}_{Q(F)}^{G(F)} \cW(w\pi_{w \lambda},\psi)$, which we will still denote it by $N(w,\lambda)$, and we write $N(w)$ for $N(w,0)$.

\subsubsection{}

Now let $F$ be a number field. Let $N_n$ be the unipotent radical of the Borel subgroup of $G_n$, we define a generic character $\psi_{N_n}$ of $[N_n]$ by
\[
    \psi_{N_n}(u)= \psi\left( \sum_{i=1}^{n-1}u_{i, i+1} \right).
\]
Assume $G=G_{n_1} \times \cdots \times G_{n_k}$. Let $N$ be the unipotent radical of the Borel subgroup of $G$ and $\psi_N= \psi_{N_{n_1}} \boxtimes \cdots \boxtimes \psi_{N_{n_k}}$ be the generic character on $[N]=[N_{n_1}] \times \cdots \times [N_{n_k}] $.
For every $f \in \cT([G])$, we set
\[
    W_f= \int_{[N]}f(ug)\psi_N(u)^{-1} \rd u.
\]
Let $\pi$ be a cuspidal representation of $G(\bA)$, then the map $f \mapsto W_f$ gives an isomorphism between $\pi$ and its \emph{Whittaker model}
\[
    \cW(\pi, \psi_N)= \left \{ W_f \mid f \in \pi \right \}.
\]
More generally, let $P$ be a standard parabolic of $G$ and $\varphi \in \cT([G]_P)$, we set
\[
    W^{M_P}_{\varphi}(g)=\int_{[M_P \cap N]}\varphi(ug)\psi_N(u)^{-1} \rd u,
\]
Let $\pi$ be a cuspidal unitary representation of $M_P(\bA)$, then the map $\varphi \in \cA_{P,\pi} \mapsto W^{M_P}_\varphi$ gives an isomorphism between $\Pi=\opn{Ind}^{G(\bA)}_{P(\bA)}\pi$ and the induction of the Whittaker model
\[
    \opn{Ind}_{P(\bA)}^{G(\bA)} \cW(\pi, \psi_N) = \{ W^{M_P}: G(\bA) \to \C, \forall g \in G(\bA), m \in M_P(\bA) \mapsto \delta_{P(\bA)}^{-\frac 12}(m)W^{M_P}(mg) \in \cW(\pi,\psi_N) \}
\]
For $f \in \cT([G])$ (resp. $\varphi \in \cT([G]_P)$) and for a finite set of places $\tS$ of $F$, let $W_{f,\tS}$ (resp. $W^{M_P}_{\varphi,\tS}$) be the restriction of $W_{f}$(resp. $W^{M_P}_{\varphi}$) to $G(F_\tS)$. 

For $\varphi \in \Pi$, write $E(\varphi)(g)=E(g, \varphi,0)$ for the Eisenstein series of $\varphi$. Let $\tS$ be a sufficiently large finite set of places of $F$.  Then it follows from \cite[\S 4]{Shahidi} that
\begin{equation} \label{eq: jacquet functional}
     W_{E(\varphi), \tS}= L(1,\pi,\widehat{\fn}_{P}^-)^{-1} \Omega_{\tS}(W^{M_P}_{\varphi,\tS}),
\end{equation}
when $L(1,\pi,\widehat{\fn}_P^-)$ has a pole at $s=1$, the right hand side is interpreted as 0.

More generally, let $R$ be a standard parabolic subgroup of $G$ containing $P$, we have that
\begin{equation} \label{eq: Induction Jacquet functional}
    W^{M_R}_{E^R(\varphi),\tS}= L(1,\pi,\widehat{\fn_P^R}^-)^{-1} \Omega^R_\tS(W^{M_P}_{\varphi, \tS}).
\end{equation}

\subsubsection{}

We still assume that $G$ is a product of $G_{n_i}$. Let $P=MN$ be a standard parabolic of $G$ and let $\varphi \in \opn{Ind}^{G(\bA)}_{P(\bA)} \pi= \cA_{P, \pi}(G)$ and $\tS$ be a sufficiently large finite set of places of $F$, which we assume to contain Archimedean places as well as places where $\varphi$ is ramified. Then we have a decomposition $W^M_\varphi= W^M_{\varphi, \tS} W^{M, \tS}_\varphi$ such that $W^{M, \tS}_\varphi(1)=1$ and is fixed by $K^\tS$. For $N_\pi(w) \varphi \in \opn{Ind}^{G(\bA)}_{Q(\bA)} w \pi= \cA_{Q, w\pi}(G)$, we also have a decomposition
\[
    W^M(N_\pi(w)\varphi)= W_\tS^M(N_\pi(w)\varphi) W^{M, \tS}(N_\pi(w)\varphi)
\]
such that $W^{M, \tS}(N_\pi(w)\varphi)(1)=1$ and is fixed by $K^\tS$. Then it follows from \eqref{eq: normalized intertwining operator} and \eqref{eq: decompose normalized intertwine} that
\begin{equation} \label{eq: intertwine Whittaker}
    W_\tS^M(N_\pi(w)\varphi)= N_{\pi, \tS}(w)(W^M_{\varphi, \tS}).
\end{equation}

\subsection{Topological vector spaces}

In this article, a LVTVS means a Hausdorff, locally convex topological vector space. We refer the readers to \cite[Appendix A]{BPCZ} for more details. Let $V,W$ be two LCTVS. We endow $\opn{Hom}(V,W)$ with the pointwise convergence topology. If $W$ is quasi-complete, then so is $\opn{Hom}(V,W)$.

Let $V,W,X$ be LCTVS. Let $\opn{Bil}_s(V,W;X)$ denote the set of separately bilinear map $V \times W \to X$. It consists of bilinear maps $f:V \times W \to X$ such that for any $v \in V$, the map $f(v,\cdot): W \to X$ is continuous and for any $w \in W$, the map $f(\cdot,w): V \to X$ is continuous.

The set $\opn{Bil}_s(V,W;X)$ is naturally identified with either $\Hom(V,\Hom(W,X))$ or $\Hom(W,\Hom(V,X))$. Using the weak topology between $\Hom$ between any LCTVS, both $\Hom(V,\Hom(W,X))$ and $\Hom(W,\Hom(V,X))$ carry a natural topology. They indeed induce the same topology on $\opn{Bil}_s(V,W;X)$, which is in fact the locally convex Hausdorff topology given by the semi-norms $f \mapsto p(f(v,w))$, where $(v,w)$ runs through $V \times W$ and $p$ runs through the continuous semi-norms on $X$.

The following fact is standard (see e.g. \cite[\S 3.2]{BL24}):
\begin{num}
    \item \label{eq:Holomorphic_map_to_Hom} A map $\C \to \opn{Hom}(V,W), s \mapsto T_s$ is holomorphic if and only if for any $v \in V$, then map $ \C \ni s \mapsto T_s(v) \in W$ is holomorphic.
\end{num}

\begin{lemma} \label{lemma: composition holomorphic}
    \begin{enumerate}
        \item \label{itm: homomorphism} Assume that $V$ is LF, $W$ is quasi-complete and let $X$ be a topological space. Let $s \in M \mapsto T_s \in \opn{Hom}(V, W)$ be holomorphic and $(s, x) \in M\times X \mapsto v_{s,x} \in V$ be a continuous map which is holomorphic in the first variable. Then, the map $(s, x) \in M\times X \mapsto T_s(v_{s,x}) \in W$ is continuous and holomorphic in the first variable.
        \item \label{item: bilinear form} Assume that $V$ and $W$ are LF. Let $s \in M \mapsto B_s \in \opn{Bil}_s(V, W)$ be holomorphic and $(s, k) \in M \times K \mapsto v_{s,x} \in V,\: (s, x) \in (M,X) \mapsto w_{s,x} \in W$ be continuous maps which are holomorphic in the first variable. Then, the function $(s,x) \in M \times X \mapsto B_s(v_{s, x}, w_{s,x})$ is continuous and holomorphic in the first variable.
    \end{enumerate}
\end{lemma}

\begin{proof}
    See \cite[p. 329]{BPCZ}.
\end{proof}

The following lemma is standard

\begin{lemma} \label{lemma: integrate compact}
    Let $K$ be a compact Hausdorff topological group, $X$ be a topological space, and let $f: \C \times K \times X \to \C$ be a continuous map which is holomorphic in the first variable. Then for any $x \in X$
    \begin{equation*}
         s \in \C \mapsto \int_{K}  f(s,k, x) \rd k
    \end{equation*}
    is holomorphic and the map
    \[
        x \in X \mapsto \int_K f(\cdot, k, x) \rd k \in \cO(\C)
    \]
    is continuous.
\end{lemma}

Let $M$ be a complex manifold and let $V$ be a topological vector space. A map $f:M \to V$ is said to be \emph{holomorphic}, if for any $\lambda \in V'$, the map $M \ni m \mapsto \langle \lambda, f(m) \rangle$ is holomorphic. 

Let $C \in \R \cup \{-\infty \}$ and $f:\cH_{>C} \to V$ be a holomorphic map. We say $f$ is \emph{of order at most $d$ in vertical strips} if for every $d'>d$, the function $z \mapsto e^{-\lvert z \rvert^{d']}} f(z)$ is bounded in vertical strips.

We also recall the following version Phragmen-Lindel\"{o}f principle \cite[Corollary A.0.11.2]{BPCZ}.

\begin{proposition} \label{prop:Phragmen-Lindelof}
    Let $W$ be a LF space, and $S \subset W$ be a dense subspace. Let $C>0$ and $Z_+,Z_-:\cH_{>C} \times W \to \C$ be two functions. Assume that
    \begin{enumerate}
        \item For every $s \in \cH_{>C}, Z_+(s,\cdot)$ and $Z_-(s,\dot)$ are continuous functional on $W$;
        \item There exists $d>0$ such that for every $w \in W$ and $\epsilon \in \{ \pm \}$, $\cH_{>C} \ni Z_{\varepsilon}(s,w)$ is a holomorphic function of order at most $d$ in vertical strips;
        \item For any $f \in S$, $s \mapsto Z_{\varepsilon}(s,f)$ extends to a holomorphic function on $\C$ of finite order in vertical strips satisfying
            \begin{equation*}
                Z_+(s,f) = Z_-(-s,f).
            \end{equation*}
    \end{enumerate}
    Then $Z_+$ and $Z_-$ extend to holomorphic functions $\C \to W'$ of finite order in vertical strips satisfying $Z_+(s,w)=Z_-(-s,w)$ for every $(s,w) \in \C \times W$.
\end{proposition}

\section{Canonical extensions of Rankin-Selberg periods -- corank 0 and 1}
\label{sec:canonical_extension_of_periods}

\subsection{Rankin-Selberg period on $\GL_n \times \GL_{n+1}$}
\label{ssec:Rankin_Selberg}

In \S \ref{ssec:Rankin_Selberg}, we discuss some results on canonical extension Rankin-Selberg period based on \cite[\S 7]{BPCZ}.

\subsubsection{Set up}

Throughout \S \ref{ssec:Rankin_Selberg}, let $G=\GL_{n} \times \GL_{n+1}$ and let $H=\GL_n$, regarded as the diagonal subgroup $(h,\begin{pmatrix}
    h & \\ & 1
\end{pmatrix})$ of $G$. For $f \in \cS([G])$, the \emph{Rankin-Selberg period} of $f$ is defined by the absolute convergent integral
\begin{equation*}
    \cP_{\mathrm{RS}}(f) := \int_{[H]} f(h) \rd h. 
\end{equation*}

\subsubsection{Rankin-Selberg regular cuspidal datum}

Let $\chi \in \fX(G)$ be a cuspidal datum of $G$. Assume that $\chi$ is represented by $(M_P,\pi)$, and we write 
\begin{equation} \label{eq:cuspidal_data_Rankin_Selberg_1}
    M_P = M_{P_n} \times M_{P_{n+1}}, \quad M_n = G_{n_1} \times \cdots \times G_{n_s}, \quad M_{n+1} = G_{m_1} \times \cdots \times G_{m_t}
\end{equation}
and 
\begin{equation} \label{eq:cuspidal_data_Rankin_Selberg_2}
    \pi = \pi_n \boxtimes \pi_{n+1}, \quad \pi_n = \pi_{n,1} \boxtimes \cdots \boxtimes \pi_{n,s}, \quad \pi_{n+1} = \pi_{n+1,1} \boxtimes \cdots \boxtimes \pi_{n+1,t}.
\end{equation}

We say $\chi$ is \emph{Rankin-Selberg regular}, if for any $1 \le i \le s, 1 \le j \le t$, we have $\pi_{n,i} \ne \pi_{n+1,j}^{\vee}$. We write $\fX_{\mathrm{RS}} \subset \fX([G])$ for the set of all the Rankin-Selberg regular cuspidal datum. We write $\cT_{\mathrm{RS}}([G])$ (resp. $\cS_{\mathrm{RS}}([G]))$ for $\cT_{\fX_{\mathrm{RS}}}([G])$ (resp. $\cS_{\fX_{\mathrm{RS}}}([G])$). 

\subsubsection{Zeta integral}

Recall that $N_n$ is the unipotent radical of the Borel subgroup of $G_n$ and we write $N= N_n \times N_{n+1}$. We define a character $\psi^\prime_N$ of $[N]$ by
\begin{equation*}
    \psi^\prime_N(u,u') = \psi \left(-\sum_{i=1}^{n-1}  u_{i,i+1} + \sum_{j=1}^n u'_{j,j+1}  \right).
\end{equation*}
We write $N_H$ for $N \cap H$. For $f \in \cT([G])$, we associate
\begin{equation*}
    W^\prime_f(g) = \int_{[N]} f(ug) \psi_N^\prime(u)^{-1} \rd u.
\end{equation*}

For $f \in \cT([G])$ and $s \in \C$, we define the \emph{zeta-integral} by
\begin{equation*}
    Z^{\mathrm{RS}}(s,f) = \int_{N_H(\bA) \backslash H(\bA)} W^\prime_f(h) \lvert \det h \rvert^s \rd h,
\end{equation*}
provided by the integral is absolutely convergent. For any $f \in \cT([G])$, the integral is convergent and holomorphic when $\mathrm{Re}(s) \gg 0$ , see \cite[Lemma 7.1.1.1]{BPCZ}.

\subsubsection{Main results}

The following theorem summarizes the main result of \cite[\S 7]{BPCZ}. 

\begin{theorem} \label{thm:corank_1_Rankin_Selberg}
    (\cite[Theorem 7.1.3.1]{BPCZ}) Let $\chi$ be a Rankin-Selberg regular cuspidal datum. Then
    \begin{enumerate}
        \item The linear functional $\cP_{\mathrm{RS}}$ on $\cS_{\chi}([G])$ extends (uniquely) by continuity to a linear functional $\cP_{\mathrm{RS}}^*$ on $\cT_{\chi}([G])$.
        \item For $f \in \cT_{\chi}([G])$, the zeta integral  $Z(\cdot,f)$ extends to an entire function of $s$. And we have
        \begin{equation*}
            \cP_{\mathrm{RS}}^*(f) = Z^{\mathrm{RS}}(0,f).
        \end{equation*}
        \item For any $s \in \C$, the functional $Z^{\mathrm{RS}}(s,\cdot)$ on $\cT_{\chi}([G])$ is continuous.
    \end{enumerate}
\end{theorem}

We provide a mild extension of Theorem \ref{thm:corank_1_Rankin_Selberg} to $\cT_{\mathrm{RS}}([G])$.

\begin{proposition} \label{prop:corank_1_Rankin_Selberg}
    We have the following statements:
    \begin{enumerate}
        \item The linear functional $\cP_{\mathrm{RS}}$ on $\cS_{\mathrm{RS}}([G])$ extends (uniquely) by continuity to a linear functional $\cP_{\mathrm{RS}}^*$ on $\cT_{\mathrm{RS}}([G])$.
        \item For $f \in \cT_{\mathrm{RS}}([G])$, the zeta integral $Z(\cdot,f)$ extends to an entire function of $s$. And we have
        \begin{equation*}
            \cP_{\mathrm{RS}}^*(f) = Z^{\mathrm{RS}}(0,f).
        \end{equation*}
        \item For any $s \in \C$, the functional $Z^{\mathrm{RS}}(s,\cdot)$ on $\cT_{{\mathrm{RS}}}([G])$ is continuous.
    \end{enumerate}
\end{proposition}

\begin{proof}
    The proof follows the same line of \cite[p. 300]{BPCZ}, we sketch the proof. For $f \in \cS([G])$, we put
    \begin{equation*}
        Z_n(s,f) = \int_{[H]} f(h) \lvert \det h \rvert^s \rd h.
    \end{equation*}
    It is an entire function in $s$ with the functional equation $Z(s,f) = Z(-s,\widetilde{f})$, where $\widetilde{f}(g)=f({}^tg^{-1})$. In order to apply Proposition \ref{prop:Phragmen-Lindelof}, hence prove the proposition, it suffices to prove that
    \begin{num}
        \item \label{eq:RS_extension_suffices} $Z_n(s,f) = Z^{\mathrm{RS}}(s,f)$ for any $f \in \cS_{\mathrm{RS}}([G])$.
    \end{num}
    For $\chi \in \fX_{\mathrm{RS}}$, let $f_{\chi}$ be the projection of $f$ in $\cS_{\chi}([G])$. Then by Theorem \ref{thm:decomposition_cuspidal_support}, the sum $\sum_{\chi \in \fX_{\mathrm{RS}}} f_\chi$ is absolutely summable in $\cS([G])$. By the main result of \cite[\S 7]{BPCZ}, $Z^{\mathrm{RS}}(s,f_{\chi}) = Z_n(s,f_{\chi})$ for any $\chi \in \fX_{\mathrm{RS}}$. By \cite[Lemma 7.1.1.1]{BPCZ}, for any $s \in \C$ such that $\mathrm{Re}(s) \gg 0$, we have 
    \begin{equation*}
        \sum_{\chi \in \fX_{\mathrm{RS}}} Z^{\mathrm{RS}}(s,f_{\chi}) = Z^{\mathrm{RS}}(s,f).
    \end{equation*}
    where the RHS is absolutely summable. It is clear that $Z_n(s,f)$ depends continuously on $f$ for any $s \in \C$, therefore
   \begin{equation*}
           \sum_{\chi \in \fX_{\mathrm{RS}}} Z_n(s,f_{\chi}) = Z_n(s,f),
   \end{equation*}
   therefore \eqref{eq:RS_extension_suffices} is proved.
\end{proof}

We endow the topological dual $\cT'_{\mathrm{RS}}([G])$ of $\cT_{\mathrm{RS}}([G])$ with the weak topology. Since the natural map $\cT_{\mathrm{RS}}([G]) \to (\cT'_{\mathrm{RS}}([G]))'$ is a bijection, we obtain:
\begin{num} 
    \item \label{eq:Rankin_Selberg_holomorphic}  The map $Z^{\mathrm{RS}}(\cdot,\cdot): \C \to \cT'_{\mathrm{RS}}([G]), s \mapsto (f \mapsto Z^{\mathrm{RS}}(s,f) )$ is holomorphic.
\end{num}

\subsection{Rankin-Selberg period on $\GL_n \times \GL_n$}
\label{ssec:Rankin_Selberg_equal_rank}

In \S \ref{ssec:Rankin_Selberg_equal_rank}, we discuss the canonical extension of equal rank Rankin-Selberg based on \cite[\S 10.3]{BLX}. The discussion is parallel to \S \ref{ssec:Rankin_Selberg}. 

\subsubsection{}

Let $G=\GL_{n} \times \GL_{n}$ and let $H=\GL_n$, regarded as the diagonal subgroup of $G$. For $f \in \cS([G])$ and $\Phi \in \cS(\bA_n)$, the (equal rank) \emph{Rankin-Selberg period} of $f$ and $\Phi$ is defined by the absolute convergent integral
\begin{equation*}
    \cP_{\mathrm{RS}}(f,\Phi) := \int_{[H]} f(h) \Theta(h,\Phi) \rd h. 
\end{equation*}

\subsubsection{Rankin-Selberg regular cuspidal datum}

Let $\chi \in \fX(G)$ be a cuspidal datum of $G$. Assume that $\chi$ is represented by $(P,\pi)$, and we write 
\begin{equation} \label{eq:cuspidal_datum_equal_rank_1}
    M_P = M_{P_1} \times M_{P_{2}}, \quad M_{P_1} = G_{n_1} \times \cdots \times G_{n_s}, \quad M_{P_2} = G_{m_1} \times \cdots \times G_{m_t}
\end{equation}
and 
\begin{equation} \label{eq:cuspidal_datum_equal_rank_2}
    \pi = \pi_1 \boxtimes \pi_{2}, \quad \pi_1 = \pi_{1,1} \boxtimes \cdots \boxtimes \pi_{1,s}, \quad \pi_{2} = \pi_{2,1} \boxtimes \cdots \boxtimes \pi_{2,t}.
\end{equation}

We say $\chi$ is \emph{Rankin-Selberg regular}, if for any $1 \le i \le s, 1 \le j \le t$, we have $\pi_{1,i} \ne \pi_{2,j}^{\vee}$. We write $\fX_{\mathrm{RS}} \subset \fX([G])$ for the set of all the Rankin-Selberg regular cuspidal datum. We write $\cT_{\mathrm{RS}}([G])$ (resp. $\cS_{\mathrm{RS}}([G])$) for $\cT_{\fX_{\mathrm{RS}}}([G])$ (resp. $\cS_{\fX_{\mathrm{RS}}}([G])$).

\subsubsection{Zeta integral}

We define a character $\psi'_N$ of $[N]$ by
\begin{equation*}
    \psi'_N(u,u') = \psi \left(-\sum_{i=1}^{n-1} u_{i,i+1} + \sum_{j=1}^{n-1} u'_{j,j+1}  \right).
\end{equation*}
We write $N_H$ for $N \cap H$. For $f \in \cT([G])$, we associate
\[
    W^\prime_f(g) = \int_{[N]} f(ug) \psi_N^\prime(u)^{-1} \rd u.
\]

For $f \in \cT([G])$, $\Phi \in \cS(\bA_n)$ and $s \in \C$, we define the \emph{zeta-integral} by
\begin{equation*}
    Z^{\mathrm{RS}}(s,f,\Phi) = \int_{N_H(\bA) \backslash H(\bA)} W'_f(h) \Phi(e_nh) \lvert \det h \rvert^{s+\frac12} \rd h,
\end{equation*}
provided by the integral is absolutely convergent. For any $f \in \cT([G])$, the integral is convergent when $\mathrm{Re}(s) \gg 0$ and holomorphic in $s$, see \cite[Lemma 10.2]{BLX}.

\subsubsection{Main results}

\begin{theorem} \label{thm:corank_0_Rankin_Selberg}
    (\cite[Theorem 10.4, Lemma 10.5]{BLX}) Let $\chi$ be a Rankin-Selberg regular cuspidal datum and $\Phi \in \cS(\bA_n)$. Then
    \begin{enumerate}
        \item The linear functional $\cP_{\mathrm{RS}}(\cdot,\Phi)$ on $\cS_{\chi}([G])$ extends (uniquely) by continuity to a linear functional $\cP_{\mathrm{RS}}^*(\cdot,\Phi)$ on $\cT_{\chi}([G])$.
        \item For $f \in \cT_{\chi}([G])$, the zeta integral  $Z(\cdot,f,\Phi)$ extends to an entire function of $s$. And we have
        \begin{equation*}
            \cP_{\mathrm{RS}}^*(f) = Z(0,f,\Phi).
        \end{equation*}
        \item For any $s \in \C$, the bilinear form $Z(s,\cdot,\cdot)$ on $\cT_{\chi}([G]) \times \cS(\bA_n)$ is continuous in the sense that there exists continuous semi-norms $\| \cdot \|$ and $\| \cdot \|'$ on $\cT_{\chi}([G])$ and $\cS(\bA_n)$ respectively, such that
        \begin{equation*}
            Z(s,f,\Phi) \ll \|f\| \|\Phi\|'.
        \end{equation*}
    \end{enumerate}
\end{theorem}

\begin{remark}
    In \emph{loc.cit}, the theorem is stated for $(G,H)$-regular cuspidal datum, but the proof indeed works for general Rankin-Selberg regular cuspidal data. See also the proof of Lemma \ref{lem:higher_corank_unfolding}.
\end{remark}

The following proposition is an analog of Proposition \ref{prop:corank_1_Rankin_Selberg} and we omit the proof.
\begin{proposition} \label{prop:corank_0_Rankin_Selberg}
    We have the following statements:
    \begin{enumerate}
        \item The linear functional $\cP_{\mathrm{RS}}$ on $\cS_{\mathrm{RS}}([G])$ extends (uniquely) by continuity to a linear functional $\cP_{\mathrm{RS}}^*$ on $\cT_{\mathrm{RS}}([G])$.
        \item For $f \in \cT_{\mathrm{RS}}([G])$ and $\Phi \in \cS(\bA_n)$, the zeta integral $Z(\cdot,f,\Phi)$ extends to an entire function of $s$.
        \item For any $s \in \C$, the bilinear map $Z(s,\cdot,\cdot)$ on $\cT_{\mathrm{RS}}([G]) \times \cS(\bA_n)$ is continuous.
    \end{enumerate}
\end{proposition}

By the same argument of \eqref{eq:Rankin_Selberg_holomorphic}, we obtain that for any $\Phi \in \cS(\bA_n)$, the map $Z^{\mathrm{RS}}(\cdot,\Phi,\cdot): \C \to \cT_{\mathrm{RS}}'([G]), s \mapsto (f \mapsto Z^{\mathrm{RS}}(s,\Phi,f) )$ is holomorphic.

Therefore, by \eqref{eq:Holomorphic_map_to_Hom}, we see that
\begin{num}
  \item \label{eq:Rankin_Selberg_corank_0_holomorphic} The map $Z^{\mathrm{RS}}(\cdot,\cdot,\cdot):\C \to \opn{Bil}_s(\cT_{\mathrm{RS}}([G]), \cS(\bA_n) ;\C), s \mapsto ((f,\Phi) \mapsto Z^{\mathrm{RS}}(s,f,\Phi))$ is holomorphic. 
\end{num}

\subsubsection{A twisted version}
\label{sssec:twisted_equal_rank}

Let $w_{\ell} = \begin{psmallmatrix}
     & & 1 \\ & \iddots & \\ 1 & &
\end{psmallmatrix} \in G_n$ be the longest Weyl group element. For $f \in \cS([G])$ and $\Phi \in \cS(\bA_n)$, we define
\begin{equation*}
    \widetilde{\cP}_{\mathrm{RS}}(f,\Phi) := \int_{[G_n]} f(w_{\ell} {}^tg^{-1} w_{\ell},g) \Theta(g,\Phi)  \rd g.
\end{equation*}
For $f \in \cT([G])$, $\Phi \in \cS(\bA_n)$ and $s \in \C$, we put the twisted Zeta integral
\begin{equation*}
    \widetilde{Z}^{\mathrm{RS}}(s,f,\Phi) = \int_{N_n(\bA) \backslash G_n(\bA)} W_f (w_{\ell} {}^tg^{-1}w_{\ell},g) \Phi(e_n g) \lvert \det g \rvert^{s+\frac 12} \rd g,
\end{equation*}
provided by the integral is absolutely convergent.

Let $\chi \in \fX(G)$. Assume that $\chi$ is represented by $(M,\pi)$ where $M$ and $\pi$ are as in \eqref{eq:cuspidal_datum_equal_rank_1} and \eqref{eq:cuspidal_datum_equal_rank_2}. We say that $\chi$ is \emph{twisted Rankin-Selberg regular}, if for any $1 \le i \le s,1 \le j \le t$, we have $\pi_{1,i} \ne \pi_{2,j}$. Let $\widetilde{\fX}_{\mathrm{RS}} \subset \fX(G)$ denote the set of twisted Rankin-Selberg regular cuspidal datum. We write $\cT_{\widetilde{\mathrm{RS}}}([G])$ for $\cT_{\widetilde{\fX}_{\mathrm{RS}}}([G])$.

\begin{corollary} \label{cor:twisted_equal_rank}
    We have the following statements:
    \begin{enumerate}
        \item For $f \in \cT([G])$ and $\Phi \in \cS(\bA_n)$, there exists $C>0$ such that the integral defining $\widetilde{Z}^{\mathrm{RS}}(s,f,\Phi)$ is convergent for $\mathrm{Re}(s) > C$ and defines a holomorphic function on $\cH_{>C}$.
        \item For any $\Phi \in \cS(\bA_n)$, the linear functional $\widetilde{\cP}_{\mathrm{RS}}(\cdot,\Phi)$ on $\cS_{\widetilde{\mathrm{RS}}}([G])$ extends (uniquely) by continuity to a continuous linear functional $\widetilde{\cP}_{\mathrm{RS}}(\cdot,\Phi)$ on $\cT_{\widetilde{\mathrm{RS}}}([G])$.
        \item For any $f \in \cT_{\widetilde{\mathrm{RS}}}([G])$ and $\Phi \in \cS(\bA_n)$, the zeta integral $\widetilde{Z}^{\mathrm{RS}}(\cdot,f,\Phi)$ extends to an entire function. And we have
        \begin{equation*}
          \widetilde{\cP}_{\mathrm{RS}}(f,\Phi) = \widetilde{Z}^{\mathrm{RS}}(0,f,\Phi) 
        \end{equation*}
        \item For any $s \in \C$, the bilinear map $\widetilde{Z}^{\mathrm{RS}}(s,\cdot,\cdot)$ on $\cT_{\widetilde{\mathrm{RS}}}([G]) \times \cS(\bA_n)$ is continuous.
    \end{enumerate}
\end{corollary}

\begin{proof}
    For a function $f$ on $[G]$, we put a new function $f'$ by $f'(g_1,g_2) = f(w_{\ell} {}^t g_1^{-1}w_{\ell},g_{2})$. Then $f \in \cS([G])$ (resp. $\cT([G])$) if and only if $f' \in \cS([G])$ (resp. $f' \in \cT([G])$). Moreover, $f \mapsto f'$ induces an isomorphism of $\cS([G])$ (resp. $\cT([G])$) to itself.

    Note that for $f \in \cS([G])$ and $\Phi \in \cS(\bA_n)$, we have $\cP_{\mathrm{RS}}(f,\Phi) = \widetilde{\cP}_{\mathrm{RS}}(f',\Phi)$. For $f \in \cT([G])$ and $\Phi \in \cS(\bA_n)$, we have $Z^{\mathrm{RS}}(s,f,\Phi) = \widetilde{Z}^{\mathrm{RS}}(s,f',\Phi)$. The corollary then easily follows from Proposition \ref{prop:corank_1_Rankin_Selberg}.
\end{proof}

By \eqref{eq:Rankin_Selberg_corank_0_holomorphic}, we see that
\begin{num}
  \item \label{eq:twisted_Rankin_Selberg_corank_0_holomorphic} The map $\widetilde{Z}^{\mathrm{RS}}(\cdot,\cdot,\cdot):\C \to \opn{Bil}_s(\cT_{\widetilde{\mathrm{RS}}}([G]), \cS(\bA_n) ;\C), s \mapsto ((f,\Phi) \mapsto \widetilde{Z}^{\mathrm{RS}}(s,f,\Phi))$ is holomorphic. 
\end{num}

\subsubsection{Euler decomposition}

Let $\tS$ be a finite set of places of $F$, let $\sigma = \sigma_n \boxtimes \sigma'_{n}$ be a generic irreducible representation of $G(F_{\tS})$. For $W \in \cW(\sigma,\psi_{N,\tS})$ and $\Phi \in \cS(F_{\tS})$, we define local (twisted) Rankin-Selberg integral \cite{JPSS} as
\begin{equation*}
    \widetilde{Z}^{\mathrm{RS}}_{\tS}(s,W,\Phi) := \int_{N_n(F_{\tS}) \backslash G_n(F_{\tS})}  W( w_{\ell}h^{-1}w_{\ell},h) \Phi(e_nh) \lvert \det h \rvert^{s+\frac 12} \rd h.
\end{equation*}
The integral defining $\widetilde{Z}^{\mathrm{RS}}_{\tS}(s,W,\Phi)$ is convergent when $\mathrm{Re}(s) \gg 0$ and has meromorphic continuation to $\C$. Moreover, by \cite{JPSS} and \cite{Jacquet09}, for any $W \in \cW(\sigma,\psi_{N,\tS})$ and $\Phi \in \cS(F_{\tS})$, the quotient
\begin{equation*}
    \frac{\widetilde{Z}^{\mathrm{RS}}_{\tS}(s,W,\Phi)}{L_{\tS}(s+\frac 12,\sigma^{\vee}_n \times \sigma'_{n})}
\end{equation*}
is entire.

Let $P \subset G$ be a standard parabolic subgroup and let $\pi$ be a cuspidal automorphic representation of $M_P$. Assume that $(M_P,\pi)$ gives a twisted Rankin-Selberg regular cuspidal data $\chi$. Let $\Pi = \opn{Ind}_{P(\bA)}^{G(\bA)} \pi=\Pi_n \boxtimes \Pi'_{n} $.  Let $\varphi \in \Pi$ and $\Phi \in \cS(\bA_n)$.

Let $\tS$ be a sufficiently large set of places of $F$, that we assume to contain Archimedean places as well as the places where $\Pi$, $\psi$, $\varphi$ or $\Phi$ is ramified. We then have a decomposition $W_{E(\varphi)} = W_{E(\varphi),\tS} W_{E(\varphi)}^{\tS}$ such that $W_{E(\varphi)}^{\tS}(1)=1$ and is fixed by $K^{\tS}$. We also write $\Phi$ as $\Phi = \Phi_{\tS} \Phi^{\tS}$, where $\Phi^{\tS}$ is the characteristic function of $\cO_F^{\tS}$ and $\Phi_{\tS} \in \cS(F_{\tS})$.

By the unramified computation for the Rankin-Selberg integral, we have
\begin{equation} \label{eq:Rankin_Selberg_Euler_equal_rank}
    \widetilde{Z}^{\mathrm{RS}}(s,E(\varphi),\Phi) = (\Delta_{G_n}^{\tS,*})^{-1} \widetilde{Z}^{\mathrm{RS}}_{\tS}(s,W_{E(\varphi),\tS},\Phi_{\tS}) L^{\tS}(s+\frac 12,\Pi^{\vee}_n \times \Pi'_{n}).
\end{equation}

\section{Canonical extensions of Rankin-Selberg periods -- higher corank}

\label{sec:higher_corank_Rankin_Selberg}

\subsection{Statements of main results}
\label{ssec:higher_corank_statements}

\subsubsection{Notations}
In \S \ref{sec:higher_corank_Rankin_Selberg}, $n \ge 0, m \ge 2$ be integers. Let $G = G_n \times G_{n+m}$. Let $H = G_n$ be the subgroup of $G$ consisting of matrices of the form $(g, \mathrm{diag}(g,1_m))$.

For integers $0 \le r \le k$, let $N_{r,k}$ be the unipotent radical of the standard parabolic subgroup of $G_k$ with Levi $G_r \times G_1^{k-r}$. Note that $N_{0,k}=N_{1,k}$ is the upper triangular unipotent subgroup of $G_k$ and $N_{k,k} = \{1\}$.

For $0 \le r \le n$, we then put
\begin{equation*}
    N_r^G := N_{r,n} \times N_{r,n+m}, \quad N_r^H := N_r^G \cap H \cong N_{r,n}.
\end{equation*}
In particular, $N := N_0^G$ is a maximal unipotent subgroup of $G$ and $N_H := N \cap H = N_0^H$ is a maximal unipotent subgroup of $H$.

We also put
\begin{equation*}
    N_{n+1}^G := 1 \times N_{n+1,n+m}.
\end{equation*}

We define a character $\psi_N$ of $[N]$ by
\begin{equation*}
    \psi'_N(u,u') = \psi \left( -\sum_{i=1}^{n-1} u_{i,i+1} + \sum_{j=1}^{n+m-1} u'_{j,j+1}  \right), \quad u \in [N_n], u' \in [N_{n+m}].
\end{equation*}
For $1 \le r \le n+1$, $\psi_N'$ restricts to a character on the subgroup $N_r^G$, we denote it by $\psi'_r$.

\subsubsection{Rankin-Selberg regular cuspidal datum}

Let $\chi \in \fX(G)$, assume that $\chi$ is represented by $(M,\pi)$, where we write
\begin{equation} \label{eq:higher_corank_cuspidal_data_1}
    M = M_n \times M_{n+m}, \quad M_n = G_{n_1} \times \cdots \times G_{n_s}, \quad M_{n+m} = G_{m_1} \times \cdots \times G_{m_t},
\end{equation}
and
\begin{equation} \label{eq:higher_corank_cuspidal_data_2}
    \pi = \pi_n \boxtimes \pi_{n+m}, \quad \pi_n = \pi_{n,1} \boxtimes \cdots \boxtimes \pi_{n,s}, \quad \pi_{n+m} = \pi_{n+m,1} \boxtimes \cdots \boxtimes \pi_{n+m,t}.
\end{equation}

we say that $\chi$ is \emph{Rankin-Selberg regular}, for any $1 \le i \le s$ and $1 \le j \le k$ we have $\pi_{n,i} \ne \pi_{n+m,j}^{\vee}$.

Let $\fX_{\mathrm{RS}}$ denote the set of Rankin-Selberg regular cuspidal datum. We write $\cS_{\mathrm{RS}}([G])$ (resp. $\cT_{\mathrm{RS}}([G])$) for $\cS_{\fX_{\mathrm{RS}}}([G])$ (resp. $\cT_{\mathrm{RS}}([G])$).

\subsubsection{Zeta integrals}

For $f \in \cT([G])$, let 
\begin{equation*}
    W'_f(g) = \int_{[N]} f(ug) \psi_N'(u)^{-1} \rd u
\end{equation*}
be its Whittaker model. For $s \in \C$, we put
\begin{equation*}
    Z^{\mathrm{RS}}(s,f) = \int_{N_H(\bA) \backslash H(\bA)} W'_f(h) \lvert \det h \rvert^s \rd h
\end{equation*}
provided by the integral is absolutely convergent.

\begin{lemma} \label{lem:higher_corank_convergence_1}
    For any $N>0$, then there exist $c_N > 0$ such that
    \begin{enumerate}
        \item For every $f \in \cT_N([G])$, the integral defining $Z^{\mathrm{RS}}(s,f)$ is absolutely convergent for $\opn{Re}(s)>c_N$, and $Z(\cdot,f)$ is holomorphic and bounded in vertical strips on $\cH_{>c_N}$.
        \item For every $s \in \cH_{>c_N}$, the functional $f \mapsto Z^{\mathrm{RS}}(s,f)$ is continuous.
    \end{enumerate}
\end{lemma}

The proof the Lemma \ref{lem:higher_corank_convergence_1} will be given in \S \ref{sssec:higher_corank_convergence_1}.

\subsubsection{}

\begin{theorem} \label{thm:higher_corank_Rankin_Selberg}
    We have the following assertions:
    \begin{enumerate}
        \item For any $f \in \cS([G])$, the \emph{Rankin-Selberg period}
        \begin{equation*}
            \cP_{\mathrm{RS}}(f) := \int_{[H]} f_{N_{n+1}^G,\psi'_{n+1}}(h) \rd h
        \end{equation*}
        is absolutely convergent.
        \item The restriction of $\cP_{\mathrm{RS}}$ to $\cS_{\mathrm{RS}}([G])$ extends by continuity to a linear functional $\cP^*_{\mathrm{RS}}$ on $\cT_{\mathrm{RS}}([G])$.
        \item For $f \in \cT_{\mathrm{RS}}([G])$, the zeta integral $Z(\cdot,f)$ extends to an entire function of $s$. And we have
        \begin{equation*}
            \cP^*_{\mathrm{RS}}(f) = Z^{\mathrm{RS}}(0,f)
        \end{equation*}
        \item For any $s \in \C$, the functional $Z^{\mathrm{RS}}(s,\cdot)$ is continuous.
    \end{enumerate}
\end{theorem}

The rest of \S \ref{sec:higher_corank_Rankin_Selberg} is devoted to the proof of Theorem \ref{thm:higher_corank_Rankin_Selberg}.

\subsection{Proof of Theorem \ref{thm:higher_corank_Rankin_Selberg}}

\subsubsection{An unfolding identity}

For $f \in \cS([G])$ and $s \in \C$, we put
\begin{equation} \label{eq:higher_corank_Z_n+1}
    Z_{n+1}(s,f) = \int_{[H]} f_{N_{n+1}^G,\psi_{n+1}'}(h) \lvert \det h \rvert^s \rd h.
\end{equation}

\begin{lemma} \label{lem:higher_corank_convergence_2}
    For any $f \in \cS([G])$, the integral defining $Z_{n+1}(s,f)$ is absolutely convergent for any $s \in \C$, and $s \mapsto Z_{n+1}(s,f)$ is entire. Moreover, for any $s \in \C$, the map $f \mapsto Z_{n+1}(s,f)$ is continuous on $\cS([G])$.
\end{lemma}

The proof of the Lemma \ref{lem:higher_corank_convergence_2} will be given in \S \ref{sssec:higher_corank_convergence_2}.

\begin{proposition} \label{prop:higher_corank_unfolding}
    Let $\chi$ be an Rankin-Selberg regular cuspidal data. Then for any $f \in \cS_{\chi}([G])$, we have
    \begin{equation*}
        Z_{n+1}(s,f) = Z^{\mathrm{RS}}(s,f)
    \end{equation*}
    holds for $\opn{Re}(s)$ sufficiently large.
\end{proposition}

The proof of Proposition \ref{prop:higher_corank_unfolding} will be given in \S \ref{ssec:higher_corank_Rankin_Selberg_unfolding}.

\begin{corollary} \label{cor:higher_corank_unfolding}
    Let $f \in \cS_{\mathrm{RS}}([G])$. For $\opn{Re}(s) \gg 0$, we have
    \begin{equation*}
        Z_{n+1}(s,f) = Z^{\mathrm{RS}}(s,f)
    \end{equation*}
    holds for $\opn{Re}(s)$ sufficiently large.
\end{corollary}

\begin{proof}
    By Theorem \ref{thm:decomposition_cuspidal_support}, we have a decomposition
    \begin{equation*}
        f = \sum_{\chi \in \fX_{\mathrm{RS}}} f_{\chi},
    \end{equation*}
    where $f_\chi \in \cS_{\chi}([G])$ and the sum is absolutely summable in $\cS([G])$. By Lemma \ref{lem:higher_corank_convergence_1} and Lemma \ref{lem:higher_corank_convergence_2}, both $Z_{n+1}(s,\cdot)$ and $Z^{\mathrm{RS}}(s,\cdot)$ is continuous when $\opn{Re}(s)$ is large enough. The result follows. 
\end{proof}

\subsubsection{Another zeta integral}

For $f \in \cT([G])$, we put
\begin{equation*}
    W''_f(g) := \int_{[N]} f(ug) \psi_N'(u) \rd u, \quad g \in G_{n+m}(\bA).
\end{equation*}

Then we define
\begin{equation} \label{eq:higher_corank_Z_1'}
    Z_1'(s,f) = \int_{N_H(\bA) \backslash H(\bA)} \int_{\opn{Mat}_{(m-1) \times n}(\bA)} W''_f \left( h, \begin{pmatrix}
    1_n & & \\ x & 1_{m-1} & \\ & & 1 \end{pmatrix} h \right)   \lvert \det h \rvert^s \rd x \rd h.
\end{equation}
provided by the integral is absolutely convergent.

\begin{lemma} \label{lem:higher_corank_convergence_3}
    For any $N>0$, there exists $c_N>0$ such that
    \begin{enumerate}
        \item For any $f \in \cT_N([G])$, the double integral defining $Z_1'(s,f)$ is absolutely convergent for $\opn{Re}(s)>c_N$, and $Z(\cdot,f)$ is holomorphic and bounded in vertical strips on $\cH_{>c_N}$.
        \item For every $s \in \cH_{>c_N}$, the functional $f \mapsto Z_1'(s,f)$ is continuous.
    \end{enumerate}
\end{lemma}

The proof of the Lemma \ref{lem:higher_corank_convergence_3} will be given in \S \ref{sssec:higher_corank_convergence_3}.

\subsubsection{Another unfolding identity}

Let $w_{\ell,m}$ denote the matrix $\begin{pmatrix}
     & & 1 \\ & \iddots & \\ 1 & &
\end{pmatrix}$ of size $m$. Let $w_{n,m}$ denote the matrix $\begin{pmatrix}
    1_n & \\ & w_{\ell,m}
\end{pmatrix}$. For a function $f$ on $[G]$, we put $\widetilde{f}(g) := f({}^t g^{-1})$.

\begin{proposition} \label{prop:higher_corank_unfolding_2}
     Let $\chi$ be an Rankin-Selberg regular cuspidal data. Then for any $f \in \cS_{\chi}([G])$, we have
     \begin{equation*}
         Z_{n+1}(s,f) = Z_1'(-s,R(w_{n,m})\widetilde{f}),
     \end{equation*}
     when $\opn{Re}(s)$ is sufficiently large.
\end{proposition}
The proof of Proposition \ref{prop:higher_corank_unfolding_2} will be given in \S \ref{ssec:higher_corank_Rankin_Selberg_unfolding}.

By the same argument of Corollary \ref{cor:higher_corank_unfolding}, 
\begin{num} \label{eq:higher_corank_unfolding_2}
    \item $Z_{n+1}(s,f) =  Z_1'(-s,R(w_{n,m})\widetilde{f})$, holds for any $f \in \cS_{\mathrm{RS}}([G])$.
\end{num}

\subsubsection{Proof of Theorem \ref{thm:higher_corank_Rankin_Selberg}}

Assertion (1) is a special case Lemma \ref{lem:higher_corank_convergence_2}. Fix $N>0$, we apply Proposition \ref{prop:Phragmen-Lindelof} to
\begin{equation*}
    W = L^{2}_{N,\opn{RS}}([G])^{\infty}, \quad S = \cS_{\mathrm{RS}}([G]), \quad Z_+(s,f) = Z^{\mathrm{RS}}(s,f), \quad Z_-(s,f) = Z_1'(s,R(w_{n,m})\widetilde{f}).
\end{equation*}
The conditions of Proposition \ref{prop:Phragmen-Lindelof} are satisifed by Lemma \ref{lem:higher_corank_convergence_1}, Lemma \ref{lem:higher_corank_convergence_3}, Lemma \ref{lem:higher_corank_convergence_2}, Corollary \ref{cor:higher_corank_unfolding} and \eqref{eq:higher_corank_unfolding_2}.

As a consequence, for any $f \in L^2_{N,\opn{RS}}([G])^{\infty}$, $Z^{\mathrm{RS}}(s,f)$ is entire and for any $s \in \C$, the map $f \mapsto Z^{\mathrm{RS}}(s,f)$. As $N$ varies, Assertion (4) is proved.

For $f \in L^{2,\infty}_{N}([G])^{\infty}$, we put
\begin{equation*}
    \cP^*_{\mathrm{RS}}(f) := Z(0,f),
\end{equation*}
by Corollary \eqref{cor:higher_corank_unfolding}, $\cP^*_{\mathrm{RS}}$ defines a continuous extension of $\cP_{\mathrm{RS}}$ to $L^{2}_N([G])^{\infty}$. As $N$ varies, assertions (2) and (3) are proved.

\subsubsection{} We endow the topological dual $\cT'_{\opn{RS}}([G])$ with the weak topology, from Theorem \ref{thm:higher_corank_Rankin_Selberg}, we see that 
\begin{num}
    \item \label{eq:higher_corank_holomorphic} The map $Z^{\mathrm{RS}}(\cdot,\cdot): \C \to \cT'_{\opn{RS}}([G]), s \mapsto (f \mapsto Z^{\mathrm{RS}}(s,f))$ is holomorphic. 
\end{num}

\subsection{Exchange of root identity} \label{ssec:exchange_of_root}

We prove an exchange of root identity in the style of \cite[Appendix A]{MS11}, \cite[\S 4]{IT13}. The main result is Corollary \ref{cor:exchange_of_root}.

\subsubsection{Settings} \label{sssec:exchange_of_root_setting}

For $0 \le r \le m-1$, let $\cU_r$ denote the unipotent subgroup of $G_{n+m}$ of the shape in the left of figure \ref{fig:unipotent_U_r}.

\begin{figure}[!ht]
\centering
\resizebox{0.8\textwidth}{!}{%
\begin{circuitikz}
\tikzstyle{every node}=[font=\LARGE]
\draw [short] (5,11.75) -- (16.25,11.75);
\draw [short] (5,11.75) -- (5,0.5);
\draw [short] (5,0.5) -- (16.25,0.5);
\draw [short] (16.25,11.75) -- (16.25,0.5);
\draw [short] (8.75,8) -- (8.75,6.75);
\draw [short] (8.75,6.75) -- (10,6.75);
\draw [short] (10,6.75) -- (10,5.5);
\draw [short] (10,5.5) -- (11.25,5.5);
\draw [short] (11.25,5.5) -- (11.25,4.25);
\draw [short] (11.25,4.25) -- (12.5,4.25);
\draw [short] (12.5,4.25) -- (12.5,3);
\draw [short] (12.5,3) -- (13.75,3);
\draw [short] (13.75,3) -- (13.75,1.75);
\draw [short] (13.75,1.75) -- (15,1.75);
\draw [short] (15,1.75) -- (15,0.5);
\draw [short] (8.75,8) -- (16.25,8);
\draw [short] (8.75,8) -- (5,8);
\draw [short] (8.75,6.75) -- (8.75,5.5);
\draw [short] (8.75,5.5) -- (5,5.5);
\draw [ color={rgb,255:red,150; green,150; blue,150} , fill={rgb,255:red,150; green,150; blue,150}] (5,8) rectangle (8.75,5.5);
\draw [ color={rgb,255:red,150; green,150; blue,150} , fill={rgb,255:red,150; green,150; blue,150}] (10,8) rectangle (16.25,6.75);
\draw [ color={rgb,255:red,150; green,150; blue,150} , fill={rgb,255:red,150; green,150; blue,150}] (11.25,6.75) rectangle (16.25,5.5);
\draw [ color={rgb,255:red,150; green,150; blue,150} , fill={rgb,255:red,150; green,150; blue,150}] (12.5,5.5) rectangle (16.25,4.25);
\draw [ color={rgb,255:red,150; green,150; blue,150} , fill={rgb,255:red,150; green,150; blue,150}] (13.75,4.25) rectangle (16.25,3);

\draw [ color={rgb,255:red,150; green,150; blue,150} , fill={rgb,255:red,150; green,150; blue,150}] (15,3) rectangle (16.25,1.75);
\draw  (5,11.75) rectangle (7,11.75);
\draw  (5,11.75) rectangle  node {\LARGE 1} (6.25,10.5);
\draw  (6.25,10.5) rectangle  node {\LARGE 1} (7.5,9.25);
\draw  (7.5,9.25) rectangle  node {\LARGE 1} (8.75,8);
\draw  (8.75,8) rectangle (8.75,7.75);
\draw  (8.75,8) rectangle  node {\LARGE 1} (10,6.75);
\draw  (10,6.75) rectangle  node {\LARGE 1} (11.25,5.5);
\draw  (11.25,5.5) rectangle  node {\LARGE 1} (12.5,4.25);
\draw  (12.5,4.25) rectangle  node {\LARGE 1} (13.75,3);
\draw  (13.75,3) rectangle  node {\LARGE 1} (15,1.75);
\draw  (15,1.75) rectangle  node {\LARGE 1} (16.25,0.5);
\draw [<->, >=Stealth] (4,8) -- (4,5.5)node[pos=0.5, fill=white]{$r=2$};
\draw [<->, >=Stealth] (5,12.5) -- (8.75,12.5)node[pos=0.5, fill=white]{$n=3$};
\draw [<->, >=Stealth] (8.75,12.5) -- (16.25,12.5)node[pos=0.5, fill=white]{$m=6$};
\draw [ color={rgb,255:red,150; green,150; blue,150} , fill={rgb,255:red,150; green,150; blue,150}] (12.5,11.75) rectangle (16.25,8);
\end{circuitikz}

\qquad \qquad \qquad
\begin{circuitikz}
\tikzstyle{every node}=[font=\LARGE]
\draw [short] (5,11.75) -- (16.25,11.75);
\draw [short] (5,11.75) -- (5,0.5);
\draw [short] (5,0.5) -- (16.25,0.5);
\draw [short] (16.25,11.75) -- (16.25,0.5);
\draw [short] (8.75,8) -- (8.75,6.75);
\draw [short] (8.75,6.75) -- (10,6.75);
\draw [short] (10,6.75) -- (10,5.5);
\draw [short] (10,5.5) -- (11.25,5.5);
\draw [short] (11.25,5.5) -- (11.25,4.25);
\draw [short] (11.25,4.25) -- (12.5,4.25);
\draw [short] (12.5,4.25) -- (12.5,3);
\draw [short] (12.5,3) -- (13.75,3);
\draw [short] (13.75,3) -- (13.75,1.75);
\draw [short] (13.75,1.75) -- (15,1.75);
\draw [short] (15,1.75) -- (15,0.5);
\draw [short] (8.75,8) -- (16.25,8);
\draw [short] (8.75,8) -- (5,8);
\draw [ color={rgb,255:red,150; green,150; blue,150} , fill={rgb,255:red,150; green,150; blue,150}] (10,8) rectangle (16.25,6.75);
\draw [ color={rgb,255:red,150; green,150; blue,150} , fill={rgb,255:red,150; green,150; blue,150}] (11.25,6.75) rectangle (16.25,5.5);
\draw [ color={rgb,255:red,150; green,150; blue,150} , fill={rgb,255:red,150; green,150; blue,150}] (12.5,5.5) rectangle (16.25,4.25);
\draw [ color={rgb,255:red,150; green,150; blue,150} , fill={rgb,255:red,150; green,150; blue,150}] (13.75,4.25) rectangle (16.25,3);

\draw [ color={rgb,255:red,150; green,150; blue,150} , fill={rgb,255:red,150; green,150; blue,150}] (15,3) rectangle (16.25,1.75);
\draw  (5,11.75) rectangle (7,11.75);
\draw  (5,11.75) rectangle  node {\LARGE 1} (6.25,10.5);
\draw  (6.25,10.5) rectangle  node {\LARGE 1} (7.5,9.25);
\draw  (7.5,9.25) rectangle  node {\LARGE 1} (8.75,8);
\draw  (8.75,8) rectangle (8.75,7.75);
\draw  (8.75,8) rectangle  node {\LARGE 1} (10,6.75);
\draw  (10,6.75) rectangle  node {\LARGE 1} (11.25,5.5);
\draw  (11.25,5.5) rectangle  node {\LARGE 1} (12.5,4.25);
\draw  (12.5,4.25) rectangle  node {\LARGE 1} (13.75,3);
\draw  (13.75,3) rectangle  node {\LARGE 1} (15,1.75);
\draw  (15,1.75) rectangle  node {\LARGE 1} (16.25,0.5);
\draw [<->, >=Stealth] (4,8) -- (4,5.5)node[pos=0.5, fill=white]{$r=2$};
\draw [<->, >=Stealth] (5,12.5) -- (8.75,12.5)node[pos=0.5, fill=white]{$n=3$};
\draw [<->, >=Stealth] (8.75,12.5) -- (16.25,12.5)node[pos=0.5, fill=white]{$m=6$};
\draw [ color={rgb,255:red,150; green,150; blue,150} , fill={rgb,255:red,150; green,150; blue,150}] (12.5,11.75) rectangle (16.25,8);
\draw [ color={rgb,255:red,150; green,150; blue,150} , fill={rgb,255:red,150; green,150; blue,150}] (5,8) rectangle (8.75,6.75);
\draw  (5,6.75) rectangle  node {\LARGE $R_r$} (8.75,5.5);
\draw  (12.5,8) rectangle  node {\LARGE $C_r$} (11.25,11.75);
\end{circuitikz}

}

\label{fig:unipotent_U_r}
\caption{The unipotent subgroups $\cU_r$, $\cU_r'$, $R_r$ and $C_r$}
\end{figure}
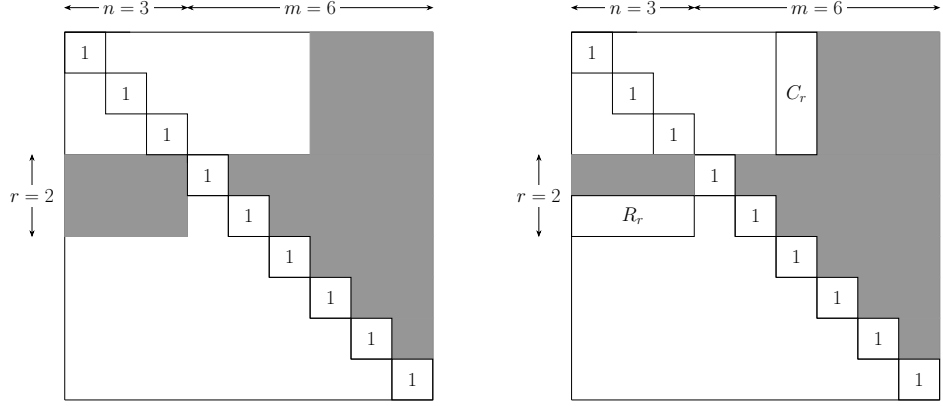

It consists of matrices $(u_{ij})$ with $1$ on the diagonal and $u_{ij} \ne 0$ only when $j>i>n$ or $1 \le i \le n, j \ge n+r+2$ or $1 \le j \le n$ and $n + 1 \le i \le n+ r$. Note that $\cU_0 = N_{n+1,n+m}$

Let $\psi_r$ denote the character $(u_{ij}) \mapsto \psi(u_{n+1,n+2} + \cdots + u_{n+m-1,n+m})$ on $\cU_r(\bA)$.

For $r \ge 1$ and $x \in \bA_n$, let $R_r(x)$ denote the matrix $\begin{pmatrix}
    1_n & & & \\ & 1_{r-1} & & \\ x &  & 1 & \\ & & & 1_{m-r}
\end{pmatrix}$. We write $R_r$ for the algebraic subgroup of $G_{n+m}$ formed by $R_r(x)$.

\subsubsection{}
\begin{lemma} \label{lem:exchange_of_root_convergence}
    Let $1 \le r \le m-1$ and $1 \le k \le r$. Let $f \in \cT([G_{n+m}])$. The integral
    \begin{equation} \label{eq:exchange_of_root_convergence}
        \int_{\opn{Mat}_{k \times n}(\bA)} f_{\cU_{r-k},\psi_{r-k}} \left( \begin{pmatrix}
            1_n &  &  &  \\  & 1_{r-k} &  & \\ x &  & 1_k & \\ & & & 1_{m-r} 
        \end{pmatrix} g  \right) \rd x
    \end{equation}
    is absolutely convergent for any $g \in G_{n+m}(\bA)$.
\end{lemma}

The proof of Lemma \ref{lem:exchange_of_root_convergence} will be given in \S \ref{sssec:exchange_of_root_convergence}.

\subsubsection{}

\begin{lemma} \label{lem:exchange_of_root}
    Let $f \in \cT([G_{n+m}])$ and $1 \le r \le m-1$, then
    \begin{equation} \label{eq:exchange_of_root}
        f_{\cU_r,\psi_{r}}(g) = \int_{\bA_n} f_{\cU_{r-1},\psi_{r-1}}(R_r(x)g) \rd x
    \end{equation}
    holds for any $g \in G_{n+m}(\bA)$, where the integral on the right-hand side is convergent.
\end{lemma}

\begin{proof}
    The convergence of the integral follows from Lemma \ref{lem:exchange_of_root_convergence}. For $1 \le r \le m-1$, let $\cU'_r := \cU_r \cap \cU_{r-1}$ denote the subgroup of $\cU_r$, see the shaded region of right hand side of figure \ref{fig:unipotent_U_r}. Let $\psi_r'$ denote the restriction of $\psi_r$ (equivalently $\psi_{r-1}$) on $\cU_r'$.

    For $y \in \bA^n$, let $C_r(y)$ denote the element $\begin{pmatrix}
    1_n & & y & \\ & 1_{r-1} & & \\  &  & 1 & \\ & & & 1_{m-r}
\end{pmatrix}$ of $G_{n+m}(\bA)$. Let $C_r$ denote the algebraic subgroup of $G_{n+m}$ formed by $C_r(y)$.

    The following statements can be checked directly:
    \begin{enumerate}
        \item \label{item:exchange_of_root_1} $\cU_r = \cU'_r \rtimes R_r$, $\cU_{r-1} = \cU'_r \rtimes C_r$, $R_r$ normalizes $\cU_{r-1}$ and $C_r$ normalizes $\cU_r$.
        \item \label{item:exchange_of_root_2}  By \eqref{item:exchange_of_root_1} above, we can write an element of $\cU_{r-1}(\bA)$ by $u'C_r(y)$, where $u' \in \cU'_r(\bA)$ and $y \in \bA^n$. For $a \in F_n$, the map $u'C_r(y) \mapsto \psi_r(u')\psi(ay)$ defines a character of $\cU_{r-1}(\bA)$ trivial on $\cU_{r-1}(F)$. We denote this character by $\psi_{r-1,a}$. Note that $\psi_{r-1,0}=\psi_{r-1}$.
        \item \label{item:exchange_of_root_3} The equality
        \begin{equation} \label{eq:exchange_of_root_key}
            \psi_{r-1}(R_r(-a)uR_r(a)) = \psi_{r-1,-a}(u)
        \end{equation}
        holds for any $a \in F_n, u \in \cU_{r-1}(\bA)$.
        \item \label{item:exchange_of_root_4} As a consequence of \eqref{item:exchange_of_root_3} above, the equation
        \begin{equation} \label{eq:psi_r_row_invariant}
            \psi_r'(R_r(-a)uR_r(a)) = \psi'_r(u)
        \end{equation}
        holds for any $a \in F_n$ and $u \in \cU_r'(\bA)$. Similarly, one can check
        \begin{equation} \label{eq:psi_r_column_invariant}
            \psi_r'(C_r(-b)uC_r(b)) = \psi'_r(u)
        \end{equation}
        holds for any $b \in F^n$ and $u \in \cU_r'(\bA)$.
    \end{enumerate}

    By \eqref{eq:psi_r_row_invariant} above, for any $g \in G_{n+1}(\bA)$, we have
    \begin{equation} \label{eq:exchange_of_root_1}
        f_{\cU_r,\psi_r}(g) = \int_{F_n \backslash \bA_n} f_{\cU_r',\psi_r'}(R_r(x)g) \rd x.
    \end{equation}
    By \eqref{eq:psi_r_column_invariant}, for any $g \in G_{n+m}(\bA)$, the map $y \mapsto f_{\cU_r',\psi_r'}(C_r(y)g)$ defines a function on $F^n \backslash \bA^n$. Therefore, by Fourier expansion, we can write
    \begin{equation} \label{eq:exchange_of_root_2}
        f_{\cU'_r,\psi_r'}(g) = \sum_{a \in F_n} \int_{ F^n \backslash \bA^n} f_{\cU'_r,\psi_r'}(C_r(y)g) \psi^{-1}(ay) \rd y = \sum_{a \in F_n} f_{\cU_{r-1},\psi_{r-1,-a}}(g)
    \end{equation}
    By \eqref{eq:exchange_of_root_key}, we have
    \begin{equation} \label{eq:exchange_of_root_3}
        f_{\cU_{r-1},\psi_{r-1,-a}}(g) = f_{\cU_{r-1},\psi_{r-1}}(R_r(a)g).
    \end{equation}
    
    Combining \eqref{eq:exchange_of_root_1}, \eqref{eq:exchange_of_root_2} and \eqref{eq:exchange_of_root_3} and Lemma \ref{lem:exchange_of_root_convergence}, \eqref{eq:exchange_of_root} is proved.
\end{proof}

\subsubsection{}
\begin{corollary} \label{cor:exchange_of_root}
    For any $f \in \cT([G_{n+m}])$, we have
    \begin{equation*}
        f_{\cU_{m-1},\psi_{m-1}}(g) = \int_{\opn{Mat}_{(m-1) \times n}(\bA)} f_{\cU_0,\psi_0} \left( \begin{pmatrix}
            1_n &  &  \\ x & 1_{m-1} & \\ & & 1
        \end{pmatrix}  g \right) \rd x,
    \end{equation*}
    where the integral of the right-hand side is absolutely convergent.
\end{corollary}

\begin{proof}
    The convergence of the right-hand side follows from Lemma \ref{lem:exchange_of_root_convergence}. 

    The equality follows from successively using Lemma \ref{lem:exchange_of_root}. The convergence of each step also follows from Lemma \ref{lem:exchange_of_root_convergence}.
\end{proof}

\subsubsection{Convergence} \label{sssec:exchange_of_root_convergence}

\begin{proof}[Proof of Lemma \ref{lem:exchange_of_root_convergence}]
    We temporarily denote by $P$ the standard parabolic subgroup of $G_{n+m}$ whose Levi factor is $G_1^{r-k} \times G_n \times G_1^{m-r+k}$. Let $\psi_{N_P}$ denote the character
    \begin{equation*}
        (u_{ij}) \mapsto \psi(u_{12}+\cdots+u_{r-k-1,r-k}+u_{r-k,n+r-k+1}+u_{n+r-k+1,n+r-k+2}+\cdots+u_{n+m-1,n+m})
    \end{equation*}
    on $N_P(\bA)$. When $r-k=0$, this is understood as $(u_{ij}) \mapsto\psi( u_{n,n+1}+\cdots+u_{n+m-1,n+m})$. When $r-k=1$, this is understood as $(u_{ij}) \mapsto \psi(u_{1,n+2}+u_{n+2,n+3}+\cdots+u_{n+m-1,n+m})$.

    Let $w \in G_{n+m}(F)$ be the permutation matrix associated to the permutation sending $1,2,\cdots,n+m$ to $n+1,\cdots,n+r-k,1,\cdots,n,n+r-k+1,\cdots,n+m$ respectively. Then the right hand side of \eqref{eq:exchange_of_root_convergence} can be written as
    \begin{equation} \label{eq:exchange_of_root_convergence_equivalence}
        \int_{\opn{Mat}_{k \times n}(\bA)} f_{N_P,\psi_{N_P}}  \left( \begin{pmatrix}
            1_{r-k} & & & \\ & 1_n & & \\ & x & 1_k & \\ & & & 1_{m-r}
        \end{pmatrix} gw \right) \rd x.
    \end{equation}
     Therefore we are reduced to show the convergence of \eqref{eq:exchange_of_root_convergence_equivalence}. Let $Q$ be the parabolic subgroup of $G_{n+k}$ whose Levi factor is $G_n \times (G_1)^k$. Assume that $\begin{pmatrix}
         1_{n} & \\ x & 1_k
     \end{pmatrix}$ is written as \eqref{eq:exchange_of_root_convergence_Iwasawa}, then we have the Iwasawa decomposition
     \begin{equation*}
          \begin{pmatrix}
            1_{r-k} & & & \\ & 1_n & & \\ & x & 1_k & \\ & & & 1_{m-r}
        \end{pmatrix} = u'(x) \begin{pmatrix}
            1_{r-k} & & & \\ & g(x) & & \\ & & t(x) & \\ & & & 1_{m-r}
        \end{pmatrix} k'(x)
     \end{equation*}
     for some $(u'(x),k'(x)) \in N_P(\bA) \times K_{n+m}$. Write $t(x)$ as $\opn{diag}(t_1(x),\cdots,t_k(x))$, assume that  $f \in \cT_N([G_{n+m}])$, by Lemma \ref{lemma: estimate Fourier coeff}, we see that the integral \eqref{eq:exchange_of_root_convergence_equivalence} is essentially bounded by
     \begin{equation} \label{eq:exchange_of_root_convergence_bound}
         \int_{\opn{Mat}_{k \times n}(\bA)} \prod_{i=1}^{k-1} \| t_i(x) t_{i+1}(x)^{-1}\|^{-N_1}_{\bA}  \|t_k(x)\|^{-N_1}_{\bA} \prod_{i=1}^k \| t_i(x) \|_{G_1}^{N} \| g \|_{G_n}^{N} \rd x
     \end{equation}
     for any $N_1 > 0$. Note that for any $N_2>0$, there exists $N_1>0$, such that
     \begin{equation} \label{eq:exchange_of_root_convergence_6}
        \prod_{i=1}^{k-1}  \| t_i(x) t_{i+1}(x)^{-1}\|^{-N_1}_{\bA}  \|t_k(x)\|^{-N_1}_{\bA} \ll \prod_{i=1}^{k} \| t_i(x) \cdots t_n(x) \|^{-N_2}.
     \end{equation}
    Combining \eqref{eq:exchange_of_root_convergence_6}, \eqref{eq:lemma_Iwasawa_1} and \eqref{eq:lemma_Iwasawa_2}, we see that the integral \eqref{eq:exchange_of_root_convergence_bound} is bounded by
    \begin{equation*}
        \int_{\opn{Mat}_{k \times n}(\bA)} \| x \|_{\opn{Mat}_{k \times n}(\bA)}^{-N_2} \rd x
    \end{equation*}
    for any $N_2 \gg 0$. The convergence hence follows.
\end{proof}

\subsection{Convergence of zeta integrals} \label{ssec:higher_corank_convergence}

\subsubsection{More zeta integrals}

The goal of \S \ref{ssec:higher_corank_convergence} is to prove convergence of various zeta integrals.

For later use in \S \ref{ssec:higher_corank_Rankin_Selberg_unfolding}, we introduce more zeta integrals. Let $f \in \cS([G])$. For $1 \le r \le n+1$, we define
\begin{equation*}
    Z_r(s,f) = \int_{ \cP_r(F)N_r^H(\bA) \backslash H(\bA)} f_{N_r^G,\psi'_r}(h) \lvert \det h \rvert^s \rd h.
\end{equation*}
Note that when $r=n+1$, this coincides with the definition in \eqref{eq:higher_corank_Z_n+1}, and $Z_1(s,f)=Z^{\mathrm{RS}}(s,f)$

For $1 \le r \le n$, we also introduce
\begin{equation*}
    Z_r'(s,f) = \int_{\cP_r(F)N_{r,n}(\bA) \backslash G_n(\bA)} \int_{\opn{Mat}_{(m-1) \times n}(\bA)} f_{N_r^G,\psi^{\prime,-1}_r}  \left( h, \begin{pmatrix}
        1_n &  &  \\ x & 1_{m-1} & \\ & & 1
    \end{pmatrix} h \right) \lvert \det h \rvert^s \rd h.
\end{equation*}
Note that when $r=1$, this coincides with the definition in  \eqref{eq:higher_corank_Z_1'}. And we put
\begin{equation*}
    Z_{n+1}'(s,f) = \int_{[G_n]} \int_{\opn{Mat}_{(m-1) \times n}(\bA)} f_{N_{n+1}^G,\psi^{\prime,-1}_{n+1}}  \left( h,\begin{pmatrix}
        1_n &  &  \\ x & 1_{m-1} & \\ & & 1
    \end{pmatrix} h \right) \lvert \det h \rvert^s \rd h.
\end{equation*}

\begin{lemma} \label{lem:higher_corank_convergence_4}
    For $f \in \cS([G])$ and $1 \le r \le n+1$, the integral defining $Z_r(s,f)$ and $Z'_r(s,f)$ are absolutely convergent when $\opn{Re}(s)$ is sufficiently large.
\end{lemma}

\subsubsection{Proof of Lemma \ref{lem:higher_corank_convergence_2}} \label{sssec:higher_corank_convergence_2}

By Lemma \ref{lem:constant_term_estimate_product}, for any $N,M>0$, there exists a continuous semi-norm $\| \cdot \|$ on $\cS([G])$, such that the integral defining $Z_{n+1}(s,f)$ is bounded by
\begin{equation*}
    \|f\| \cdot \int_{[H]} \|h\|^{-N}_H \|h\|_H^{-M} \delta_{P_{n+1,n+m}}(h)^{-cM} \lvert \det h \rvert^{\opn{Re}(s)} \rd h
\end{equation*}
for some constants $c>0$ and any $N,M>0$. The result follows.

\subsubsection{Proof of Lemma \ref{lem:higher_corank_convergence_4}}
\label{sssec:higher_corank_convergence_4}

Let $Q_r$ denote the parabolic subgroup of $H$ with Levi component $G_r \times G_1^{n-r}$.  Convergence of $Z_{n+1}(s,f)$ is covered in Lemma \ref{lem:higher_corank_convergence_2}. For $1 \le r \le n$, let $P_r$ denote the parabolic subgroup of $G$ with Levi component $(G_r \times G_1^{n-r}) \times (G_r \times G_1^{m+n-r})$. Using Iwasawa decomposition $H(\bA)=Q_r(\bA)K_n$. The integral defining $Z_r(s,f)$ is bounded by
\begin{equation}  \label{eq:higher_corank_convergence_1}
    \int_{\cP_r(F) \backslash G_r(\bA)} \int_{[G_1]^{n-r}} \int_{K_n} \left \vert R(k)f_{N_r^G,\psi_r'} \begin{pmatrix}
        h & \\ & t
    \end{pmatrix}  \right \vert \lvert \det h \rvert^s \lvert  \det t \rvert^s \delta_{Q_r} \begin{pmatrix}
        h & \\ & t
    \end{pmatrix}^{-1} \rd h \rd t \rd k.
\end{equation}
By Lemma \ref{lemma: estimate Fourier coeff}, there exists $c>0$ such that for any $N>0$ and $N_1>0$, we have
\begin{equation*}
    \begin{split}
    &\left \vert R(k)f_{N_r^G,\psi_r'} \begin{pmatrix}
        h & \\ & t
    \end{pmatrix}  \right \vert \ll \\
    &\| t_1^{-1} e_r h_r  \|^{-2N_1}_{\bA_r} \|t_1 t_2^{-1} \|_{\bA}^{-2N_1} \cdots \|t_{n-r-1}t_{n-r}^{-1}\|_{\bA}^{-2N_1} \|t_{n-r}\|_{\bA}^{-N_1} \delta_{P_r} \begin{pmatrix}
        h & \\ & t
    \end{pmatrix}^{-cN} \|h\|_{G_r}^{-N} \|t\|_{G_1^r}^{-N},
    \end{split}
\end{equation*}
where $t = \opn{diag}(t_1,\cdots,t_{n-r})$. Since for any $N_2>0$ there exists $N_1>0$ such that
\begin{equation*}
    \| t_1^{-1} e_r h_r  \|^{-2N_1}_{\bA_r} \|t_1 t_2^{-1} \|_{\bA}^{-2N_1} \cdots \|t_{n-r-1}t_{n-r}^{-1}\|_{\bA}^{-2N_1} \|t_{n-r}\|_{\bA}^{-N_1}  \ll \|e_rh\|_{\bA}^{-N_2} \|t_1\|_{\bA}^{-N_2} \cdots \|t_{n-r}\|_{\bA}^{-N_2}.
\end{equation*}
We then see that the integral \eqref{eq:higher_corank_convergence_1} is essentially bounded by
\begin{equation} \label{eq:higher_corank_convergence_4}
    \int_{\cP_r(F) \backslash G_r(\bA)} \int_{[G_1]^{n-r}} \|e_r h\|_{\bA}^{-N_2} \prod_{i=1}^{n-r} \|t_i\|_{\bA}^{-N_2} \lvert \det h \rvert^{s-\alpha(N)} \prod_{i=1}^{t-r} \lvert t_i \rvert^{s-\alpha_i(N)} \| h \|^{-N}_{G_n} \rd h \rd t,
\end{equation}
for some $c>0$ and any $N>0$ and $N_2>0$, where
\begin{equation} \label{eq:higher_corank_convergence_3}
    \alpha(N) = 2c(n+m-r)N+n-r, \quad \alpha_i(N) = c(2n+m-4r+2-4i)N + (n-2r+1-2i).
\end{equation}
We have $\alpha(N)>\alpha_1(N)>\cdots>\alpha_{n-r}(N)$.
Let $C=2$ in Corollary \ref{cor:Theta_series_integral}, together with Corollary \ref{cor:A_norm_integral}, we see that there exists $N_0>0$ such that for any $C_1>0$ and $N > N_0$, the integral is convergent for
\begin{equation*}
    -2 + \alpha(N) < \opn{Re}(s) < 2 + \alpha(N) \text{ and } 1 + \alpha_{1}(N) < \opn{Re}(s) < C_1 + \alpha_{n-r}(N).
\end{equation*}
As $N$ and $C_1$ vary, we see that the integral is convergent when $\opn{Re}(s) \gg 1$. This shows Lemma \ref{lem:higher_corank_convergence_4} for $Z_r(s,f)$. 

To show the convergence of $Z_r'(s,f)$. We prove the case when $1 \le r \le n$, the case when $r=n+1$ follows from a similar (and easier) argument. For simplicity, for $x \in \opn{Mat}_{(m-1) \times n}$ we write $A(x)$ for the matrix $\begin{pmatrix}
    1_n & & \\ x & 1_{m-1} & \\ & & 1
\end{pmatrix}$. The absolute convergence of $Z_r'(s,f)$ is equivalent to the convergence of 
\begin{equation*}
    \int_{\cP_r(F)N_{r,n}(\bA) \backslash G_n(\bA)} \int_{\opn{Mat}_{(m-1) \times n}(\bA)} \left \vert f_{N_r^G,\psi^{\prime,-1}_r}(h,hA(x)) \right \vert \lvert \det h \rvert^{\opn{Re}(s)-m} \rd h.
\end{equation*}

Using Iwasawa decomposition, note that for $k \in K_n$, $kA(x)k^{-1}=A(xk)$, the integral can be written as
\begin{equation} \label{eq:higher_corank_convergence_6}
    \begin{split}
    \int_{\cP_r(F) \backslash G_r(\bA)} \int_{[G_1]^r} \int_{K_n} \int_{\opn{Mat}_{(m-1) \times n}(\bA)} &\left \vert f_{N_r^G,\psi^{\prime,-1}_r}\left( \begin{pmatrix}
        h & \\ & t
    \end{pmatrix}k , \begin{pmatrix}
         h & \\ & t
    \end{pmatrix} A(x)k \right) \right \vert \\
    &\lvert \det h \rvert^{\opn{Re}(s)-m} \lvert \det t \rvert^{\opn{Re}(s)-m} \delta_{Q_r} \begin{pmatrix}
        h & \\ & t
    \end{pmatrix}^{-1}
    \rd x \rd k \rd t \rd h.
    \end{split}
\end{equation}
Let $R_r$ (resp. $R_r'$) denote the parabolic subgroup of $G_{n+m}$ (resp. $G_{n+m-1}$) with Levi component $G_r \times G_1^{n+m-r}$ (resp. $G_r \times G_1^{n+m-r-1}$). Let $\begin{pmatrix}
    1_n & \\ x & 1_m
\end{pmatrix} =n'(x)m'(x)k'(x)$ be a measurable decomposition of $\begin{pmatrix}
    1_n & \\ x & 1_m
\end{pmatrix}$ under the Iwasawa decomposition $N_{R'_r}(\bA) M_{R'_r}(\bA)K_{n+m-1}$ (see \S \ref{sssec:Iwasawa_decomposition}). Then we can write $A(x)$ as $\begin{pmatrix}
    n'(x) & \\ & 1
\end{pmatrix} \begin{pmatrix}
    m'(x) & \\ & 1
\end{pmatrix} \begin{pmatrix}
    k'(x) & \\ & 1
\end{pmatrix} =: n(x)m(x)k(x)$, which is an Iwasawa decomposition of $A(x)$ under $G_{n+m}(\bA) = N_{R_r}(\bA) M_{R_r}(\bA) K_{n+m}$. We also write $m(x)$ as $m(x) = \opn{diag}(h(x),t(x),t'(x),1)$, where $h(x) \in \GL_r(\bA),t(x) = \opn{diag}(t_1(x),\cdots,t_{n-r}(x))$ and $t'(x)=\opn{diag}(t_1'(x),\cdots,t_{m-1}'(x))$. The integral \eqref{eq:higher_corank_convergence_6} then can be written as
\begin{equation} \label{eq:higher_corank_convergence_2}
    \begin{split}
    \int_{\cP_r(F) \backslash G_r(\bA)} \int_{[G_1]^r} \int_{K_n} \int_{\opn{Mat}_{(m-1) \times n}(\bA)} &\left \vert (R(k,kk(x))f)_{N_r^G,\psi^{\prime,-1}_r}\left( \begin{pmatrix}
        h & \\ & t
    \end{pmatrix} , \begin{pmatrix}
         h & \\ & t
    \end{pmatrix} m(x) \right) \right \vert \\
    &\lvert \det h \rvert^{\opn{Re}(s)-m} \lvert \det t \rvert^{\opn{Re}(s)-m} \delta_{Q_r} \begin{pmatrix}
        h & \\ & t
    \end{pmatrix}^{-1}
    \rd x \rd k \rd t \rd h.
    \end{split}
\end{equation}
We will use the notation from Lemma \ref{lemma: estimate Fourier coeff}. Let $l$ denote the map 
\begin{equation*}
    (u,u') \in N_{P_r} \mapsto \sum_{i=r}^{n-1} u_{i,i+1} - \sum_{j=r}^{n+m-1} u_{j,j+1} \in \bG_a.
\end{equation*}
One readily checks that there exists $N_0>0$ such that
\begin{equation*}
   \left \Vert \opn{Ad}^* \left( \begin{pmatrix}
        h & \\ & t
    \end{pmatrix}, \begin{pmatrix}
        h & \\ & t
    \end{pmatrix} m(x) \right)^{-1} l \right \Vert_{V_{P_r},\bA} \gg \| e_r hh(x) \|^{N_0} \prod_{i=1}^{n-r} \|t_it_i(x)\|^{N_0} \prod_{i=1}^{m-1} \|t'(x)\|^{N_0}.
\end{equation*}
Therefore, by Lemma \ref{lemma: estimate Fourier coeff}, we have
\begin{equation}  \label{eq:higher_corank_convergence_5}
    \begin{split}
   &\left \vert R(k,kk(x))f_{N_r^G,\psi_r^{\prime,-1}} \left( \begin{pmatrix}
        h & \\ & t
    \end{pmatrix}, \begin{pmatrix}
        h & \\ & t
    \end{pmatrix} m(x) \right) \right \vert \ll  \|e_r hh(x) \|^{-N_1}_{\bA_r} \prod_{i=1}^{n-r} \| t_i t_i(x) \|^{-N_1}_{\bA} \prod_{i=1}^{m-1} \|t_i'(x)\|^{-N_2}_{\bA} \\
    &  \|h\|_{G_r}^{-N} \|t\|_{G_1^r}^{-N} \|tt(x)\|_{G_1^r}^{-N} \|t'(x)\|^{-N}_{G_1^{m-1}} \delta_{P_r} \left(\begin{pmatrix}
        h & \\ & t
    \end{pmatrix} , \begin{pmatrix}
         h & \\ & t
    \end{pmatrix} m(x) \right)^{-cN}.
    \end{split} 
\end{equation}
for some $c>0$ and any $N_1,N_2>0$ and $N>0$.
By Lemma \ref{lem:Iwasawa_estimate}, for any $N_3>0$, we can find $N_2>0$ such that right hand side of \eqref{eq:higher_corank_convergence_5} is essentially bounded by
\begin{equation*}
    \|e_rh\|^{-N_1} \prod_{i=1}^{n-r} \|t_i\|^{-N_1} \|x\|_{\opn{Mat}_{(m-1) \times n}(\bA)}^{-N_3} \|h\|_{G_r}^{-N} \|t\|_{G_1^r}^{-N} \delta_{P_r} \left( \begin{pmatrix}
        h & \\ & t
    \end{pmatrix},\begin{pmatrix}
        h &  & \\ & t & \\ & & 1_{m}
    \end{pmatrix} \right)^{-cN}
\end{equation*}

Therefore the integral \eqref{eq:higher_corank_convergence_2} is essentially bounded by
\begin{equation} \label{eq:higher_corank_convergence_7}
    \begin{split}
  \int_{\cP_r(F) \backslash G_r(\bA)} \int_{[G_1]^r} &\int_{\opn{Mat}_{(m-1) \times n}(\bA)}  \|e_rh\|^{-N_1}_{\bA_r} \prod_{i=1}^{n-r} \|t_i\|_{\bA}^{-N_1} \| x \|_{\opn{Mat}_{(m-1) \times n}(\bA)}^{-N_3} \| h \|_{G_r}^{-N} \| t \|_{G_1^r}^{-N}  \\
  & \lvert \det h \rvert^{\opn{Re}(s)-\alpha(N)-m} \prod_{i=1}^{n-r} \lvert t_i \rvert^{\opn{Re}(s)-\alpha_i(N)-m} \rd x \rd h \rd t
  \end{split}
\end{equation}
for any $N,N_1,N_3>0$, where $\alpha(N)$ and $\alpha_i(N)$ is as in \eqref{eq:higher_corank_convergence_3}. The convergence follows from the convergence of \eqref{eq:higher_corank_convergence_4} when $\opn{Re}(s) \gg 0$.

\subsubsection{Proof of Lemma \ref{lem:higher_corank_convergence_1}}
\label{sssec:higher_corank_convergence_1}

Let $B_H$ be the upper triangular Borel subgroup of $H$. By Iwasawa decomposition $H(\bA)=B_H(\bA)K_n$, the integral defining $Z^{\mathrm{RS}}(s,f)$ is bounded by
\begin{equation*}
    \int_{[G_1^r]} \int_{K_n} \lvert W'_{R(k)f}(t) \rvert \lvert \det t \rvert^s \delta_{B_H}(t)^{-1} \rd t \rd k.
\end{equation*}
Similar to the derivation of \eqref{eq:higher_corank_convergence_4}, for any $N_2>0$, there exists a continuous semi-norm $\| \cdot \|$ on $\cT_N([G])$, such that the integral is essentially bounded by
\begin{equation*}
   \|f \| \cdot \int_{[G_1^r]} \prod_{i=1}^n \|t_i\|^{-N_2} \lvert t_i \rvert^{\opn{Re}(s)-(n+1-2i)} \|t\|_{G_1}^N \rd t.
\end{equation*}
Note that there exists $M>0$ such that
\begin{equation} \label{eq:higher_corank_convergence_8}
    \| t \|_{G_1} \ll \max\{ \lvert t \rvert^M , \lvert t \rvert^{-M} \},
\end{equation}
therefore the result follows from Corollary \ref{cor:A_norm_integral}.

\subsubsection{Proof of Lemma \ref{lem:higher_corank_convergence_3}}
\label{sssec:higher_corank_convergence_3}

Let $\begin{pmatrix}
    1 & \\ x & 1_m
\end{pmatrix} = n'(x)t'(x)k'(x)$ be a measurable decomposition under $G_{n+m-1}(\bA) = N_{n+m-1}(\bA)T_{n+m-1}(\bA)K_{n+m-1}$. Then $A(x) = \begin{pmatrix}
    n'(x) & \\ & 1
\end{pmatrix}\begin{pmatrix}
    t'(x) & \\ & 1
\end{pmatrix}\begin{pmatrix}
    k'(x) & \\ & 1
\end{pmatrix} := n(x)t(x)k(x)$, where $t(x) = \opn{diag}(t_1(x),\cdots,t_n(x),t_1'(x),\cdots,t_{m-1}'(x),1)$. Same as the derivation of \eqref{eq:higher_corank_convergence_2}, the integral defining $Z_1'(s,f)$ is essentially bounded by
\begin{equation*}
    \int_{[G_1]^r} \int_{K_n} \int_{\opn{Mat}_{(m-1) \times n}(\bA)} \left \vert W''_{R(k,kk(x))f}(t,tt(x)) \right \vert \lvert \det t \rvert^{\opn{Re}(s)-m} \delta_{B_H}(t)^{-1} \rd x \rd k \rd t.
\end{equation*}
The using the same argument for \eqref{eq:higher_corank_convergence_7}, for any $N_1,N_3>0$, there exists a continuous semi-norm $\| \cdot \|$ on $\cT_N([G])$, such that this integral is essentially bounded by
\begin{equation*}
    \|f \| \cdot \int_{[G_1]^r} \int_{\opn{Mat}_{(m-1 \times n)}(\bA)} \prod_{i=1}^n \|t_i\|^{-N_1} \|x\|_{\opn{Mat}_{(m-1) \times n}(\bA)}^{-N_3} \|t\|_{G_1}^N \prod_{i=1}^n \lvert t_i \rvert^{\opn{Re}(s)-m-n-1+2i} \rd x \rd t,
\end{equation*}
Using \eqref{eq:higher_corank_convergence_8} and Corollary \ref{cor:A_norm_integral}, the result follows.

\subsection{Unfolding} \label{ssec:higher_corank_Rankin_Selberg_unfolding}

The goal of \S \ref{ssec:higher_corank_Rankin_Selberg_unfolding} is to prove Proposition \ref{prop:higher_corank_unfolding} and Proposition \ref{prop:higher_corank_unfolding_2}.

Recall the subgroup $\cU_0$ and the character $\psi_0$ on $\cU_0(\bA)$ defined in \S \ref{sssec:exchange_of_root_setting}. By the change of variable $h \mapsto {}^t h^{-1}$, for any $f \in \cS([G])$ we have
\begin{equation*}
    Z_{n+1}(s,f) = \int_{[H]} (R(w_{n,m}) \widetilde{f})_{\{1\} \times \cU_0,\psi_0^{-1}}(h) \lvert \det h \rvert^{-s} \rd h.
\end{equation*}
Therefore by Corollary \ref{cor:exchange_of_root} and the absolute convergence of $Z_{n+1}'$, we see that
\begin{equation*}
    Z_{n+1}(s,f) = Z'_{n+1}(-s, R(w_{n,m}) \widetilde{f}).
\end{equation*}

Therefore we are left to show:
\begin{lemma} \label{lem:higher_corank_unfolding}
    Let $\chi \in \fX(G)$ be an Rankin-Selberg regular cuspidal data. Then for any $1 \le r \le n$ and $f \in \cS_{\chi}([G])$,
    \begin{equation*}
        Z_r(s,f) = Z_{r+1}(s,f), \quad Z_r'(s,f) = Z_{r+1}'(s,f)
    \end{equation*}
    holds when $\opn{Re}(s)$ is sufficiently large.
\end{lemma}

\begin{proof}
    For $1 \le r \le n$, recall from the introduction, we regard $U_r$ as a subgroup of $G_n$. We put $U_r^G:=U_r \times U_r \subset G$. We also define $U_{n+1}^G := \{1\} \times U_{n+1}$. For $1 \le r \le n+1$ Let $U_{r}^H := U_r^G \cap H$.

    For $1 \le r \le n$, using Fourier expansion on the compact abelian group $U_r^H(\bA)U_r^G(F) \backslash U_r^G(\bA)$, we can write
    \begin{equation*}
       \int_{[U_r^H]} f_{N_{r+1}^G,\psi_r^{\prime,\pm}}(ug) \rd u =  (f_{N_{r+1}^G,\psi_r^{\prime,\pm}})_{U_{r+1}^G}(g) + \sum_{\gamma \in \cP_r(F) \backslash G_r(F)} (f_{N_{r+1}^G,\psi_r^{\prime,\pm}})_{U_{r+1}^G,\psi^{\pm}}((\gamma,\gamma)g),
    \end{equation*}
    where
    \begin{equation*}
        (f_{N_{r+1}^G,\psi_r^{\prime,\pm}})_{U_{r+1}^G}(g) = \int_{[U_{r+1}^G]} f_{N_{r+1}^G,\psi^{\prime,\pm}}(ug) \rd g,
    \end{equation*}
    and
    \begin{equation*}
        \sum_{\gamma \in \cP_r(F) \backslash G_r(F)} (f_{N_{r+1}^G,\psi_r^{\prime,\pm}})_{U_{r+1}^G,\psi^{\pm}}(g) = \int_{[U_{r+1}^G]} f_{N_{r+1}^G,\psi^{\prime,\pm}}(ug) \psi_N^{\prime,\pm}(u) \rd u = f_{N_r^G,\psi_r^{\prime,\pm}}(g).
    \end{equation*}
    We then formally have that $1 \le r \le n$
    \begin{equation} \label{eq:higher_corank_unfolding_5}
        Z_{r+1}(s,f) = Z_r(s,f) + F_r(s,f), \quad Z_{r+1}'(s,f) = Z_r'(s,f) + F_r'(s,f),
    \end{equation}
    where
    \begin{equation*}
        F_r(s,f) = \int_{G_r(F)N_r^H(\bA) \backslash H(\bA)}  (f_{N_{r+1}^G,\psi_r^{\prime}})_{U_{r+1}^G}(h) \lvert \det h \rvert^s \rd h,
    \end{equation*}
    and
    \begin{equation*}
        \begin{split}
        &F_r'(s,f) = \int_{G_r(F)N_{r,n}(\bA) \backslash G_n(\bA)} \int_{\opn{Mat}_{(m-1) \times n}(\bA)} (f_{N_{r+1}^G,\psi_r^{\prime,-1}})_{U_{r+1}^G} \left( h, \begin{pmatrix}
            1_n & & \\ x & 1_{m-1} & \\ & & 1
        \end{pmatrix} h \right) \lvert \det h \rvert^s \rd h \\
         &=\int_{G_r(F)N_{r,n}(\bA) \backslash G_n(\bA)} \int_{\opn{Mat}_{(m-1) \times n}(\bA)} (f_{N_{r+1}^G,\psi_r^{\prime,-1}})_{U_{r+1}^G} \left( h, h\begin{pmatrix}
            1_n & & \\ x & 1_{m-1} & \\ & & 1
        \end{pmatrix}  \right) \lvert \det h \rvert^{s-m} \rd h
        \end{split}
    \end{equation*}
    To verify \eqref{eq:higher_corank_unfolding_5}, we need to show that:
    \begin{lemma} \label{lem:higher_corank_F_r_convergence}
        The integral defining $F_r(s,f)$ and $F_r'(s,f)$ are absolutely convergent when $\opn{Re}(s) \gg 0$.
    \end{lemma}
    Assume Lemma \ref{lem:higher_corank_F_r_convergence} for now, it remains to show $F_r(s,f) = F_r'(s,f)=0$ for $\opn{Re}(s) \gg 0$ and $1 \le r \le n$. Note that $F_r(s,f) = 0$ (for any $f$ and $\psi$) implies $F'_r(s,f)=0$. We now prove that $F_r(s,f)=0$ for $\opn{Re}(s) \gg 0$. We temporarily denote by $P_r$ the parabolic subgroup of $G$ with Levi component $(G_r \times G_{n-r}) \times (G_r \times G_{n+m-r})$. 
    \begin{equation*}
        \begin{split}
        &(f_{N_{r+1}^G,\psi_r^{\prime}})_{U_{r+1}^G}(g) \\ =&\int_{[N_{n-r}]} \int_{[N_{m+n-r}]} f_{P_r} \left( \left( \begin{pmatrix}
            1_r & \\ & u
        \end{pmatrix}, \begin{pmatrix}
            1_r & \\ & u'
        \end{pmatrix}  \right) g \right) \psi_{N_{n-r}}(u)^{-1} \psi_{N_{m+n-r}}(u') \rd u \rd u'.
        \end{split}
    \end{equation*}
    Let $Q_r$ be the parabolic subgroup of $G_n$ with Levi component $G_r \times G_{n-r}$. Using the Iwasawa decomposition $G_n(\bA) = Q_r(\bA)K_n$, we can write $F_r(s,f)$ as
    \begin{equation} \label{eq:higher_corank_unfolding_3}
        \begin{split}
            &\int_{[G_r]} \int_{N_{n-r}(\bA) \backslash G_{n-r}(\A)} \int_{K_n} \int_{[N_{n-r}]} \int_{[N_{n+m-r}]} \\
            & (R(k)f)_{P_r} \left( \begin{pmatrix}
                h_r & \\ & u h_{n-r}
            \end{pmatrix},  \begin{pmatrix}
                h_r & \\ & u' \begin{pmatrix}
                    h_{n-r} & \\ & 1_m
                \end{pmatrix}
            \end{pmatrix} \right) \delta_{P_r} \begin{pmatrix}
                h_r & \\  & h_{n-r}
            \end{pmatrix}^{\frac{s-n+r}{2n-2r+m}} \\
            &\lvert \det h_{n-r} \rvert^{\frac{(2n+m)s+rm}{2n-2r+m}} \rd u' \rd u \rd k \rd h_{n-r} \rd h_r.
        \end{split}
    \end{equation}
    Fix $N$ sufficiently large. By \eqref{eq:constant_term_L^2_N}, for $\opn{Re}(s) \gg 0$, we have $[M_{P_r}] \ni m \mapsto f_{P_r}(m) \delta_{P_r}(m)^s \in L_{N,\chi^{M_{P_r}}}^2([M_{P_r}])^{\infty}$. We write an element of $[M_{P_r}]$ as $(h_r,x,h'_r,y) \in [G_r] \times [G_{n-r}] \times [G_r] \times [G_{n+m-r}]$. By Lemma \ref{lemma: constant term}, Lemma \ref{lemma: restriction to levi} and the definition of Rankin-Selberg regular, for any $(x,y) \in [G_{n-r}] \times [G_{n+m-r}]$, $(h_r,h_r') \mapsto f_{P_r}(h_r,x,h_r',y) \delta_{P_r}(h_r,x,h_r',y)^s$ lies in sum of $L^2_{N,\chi}([G_r \times G_r])^{\infty}$, where $\chi=(\chi_r,\chi_r')$ with $\chi_r \ne \chi_r^{\vee}$. Therefore, the integration of \eqref{eq:higher_corank_unfolding_3} over $[G_r]$ already vanishes. This finishes the proof.

    It remains to prove Lemma \ref{lem:higher_corank_F_r_convergence}. We use the notation from Lemma \ref{lemma: estimate Fourier coeff}. Let $l$ denote the map
    \begin{equation*}
        (u,u') \in N_{P_r} \mapsto -\sum_{i=r+1}^{n-1} u_{i,i+1} + \sum_{i=r+1}^{m+n-1} u'_{i,i+1},
    \end{equation*}
    when $r=n-1$ or $n$, the first term is understood as $0$. Then $(f_{N_{n+1}^G,\psi_{r}'})_{U_{r+1}^G} = f_{N_{P_r},\psi_l}$. Using Iwasawa decomposition, the integral defining $F_r(s,f)$ is bounded by
    \begin{equation} \label{eq:higher_corank_unfolding_4}
        \int_{[G_r]} \int_{[G_1]^{n-r}} \int_{K_n} \left \vert f_{N_{P_r},\psi_l} \left( \begin{pmatrix}
            h & \\ & t
        \end{pmatrix}k
       \right) \right \vert \delta_{Q_r} \begin{pmatrix}
            h & \\ & t
        \end{pmatrix}^{-1} \rd h \rd t \rd k.
    \end{equation}
    Note that there exists $N_0>0$ such that
    \begin{equation*}
        \left \Vert \opn{Ad}^* \begin{pmatrix}
            h & \\ & t
        \end{pmatrix}^{-1} l \right \Vert \gg \prod_{i=1}^{n-r} \|t_i\|_{\bA}^{N_0}.
    \end{equation*}
    By Lemma \ref{lemma: estimate Fourier coeff}, the integral
    \eqref{eq:higher_corank_unfolding_4} is essentially bounded by
    \begin{equation*}
        \int_{[G_r]} \int_{[G_1]^{n-r}} \prod_{i=1}^{n-r} \| t_i \|^{-N_1} \|h\|_{G_r}^{-N} \| t \|^{-N}_{G_1^r} \delta_{P_r} \begin{pmatrix}
            h & \\ & t
        \end{pmatrix}^{-cN} \delta_{Q_r} \begin{pmatrix}
            h & \\ & t
        \end{pmatrix}^{-1} \lvert \det h \rvert^s \lvert \det t \rvert^s \rd t \rd h,
    \end{equation*}
    whose convergence follows from the same argument of \S \ref{sssec:higher_corank_convergence_4}. The convergence of $F_r'(s,f)$ is also similar to the argument in \S \ref{sssec:higher_corank_convergence_4} and is left to the reader.
\end{proof}

\subsection{A twisted version} \label{ssec:twisted_higher_corank}

In \S \ref{ssec:twisted_higher_corank}, we discuss a twisted version of results in \S \ref{ssec:Rankin_Selberg} and \S \ref{ssec:higher_corank_statements}. We fix integers $n \ge 0$ and $m \ge 1$. Note that, in contrast with earlier part of \S \ref{sec:higher_corank_Rankin_Selberg}, we allow $m=1$.

\subsubsection{A twisted version}
\label{sssec:twisted_higher_corank}

Let $w_{\ell} = \begin{psmallmatrix}
     & & 1 \\ & \iddots & \\ 1 & &
\end{psmallmatrix} \in G_n$ be the longest Weyl group element. For $f \in \cS([G])$, we define the \emph{twisted Rankin-Selberg period} as
\begin{equation*}
    \widetilde{\cP}_{\mathrm{RS}}(f,\Phi) := \int_{[G_n]} f_{N_{n+1}^G,\psi_{n+1}'}\left(w_{\ell} {}^tg^{-1} w_{\ell}, \begin{pmatrix}
        g & \\ & 1_m
    \end{pmatrix}  \right) \rd g.
\end{equation*}
Let $\psi_N$ denote the character
\begin{equation*}
    \psi_{N}(u,u') = \psi \left( \sum_{i=1}^{n-1} u_{i,i+1} + \sum_{j=1}^{m+n-1} u'_{j,j+1}  \right)
\end{equation*}
and for $f \in \cT([G])$, its Whittaker function is defined by
\begin{equation*}
    W_f(g) = \int_{[N]} f(ug) \psi_N(u)^{-1} \rd u.
\end{equation*}

For $f \in \cT([G])$ and $s \in \C$, we put the twisted Zeta integral
\begin{equation*}
    \widetilde{Z}^{\mathrm{RS}}(s,f) = \int_{N_n(\bA) \backslash G_n(\bA)} W_f \left(w_{\ell} {}^tg^{-1}w_{\ell},\begin{pmatrix}
        g & \\ & 1_m
    \end{pmatrix} \right) \lvert \det g \rvert^s \rd g,
\end{equation*}
provided by the integral is absolutely convergent.

Let $\chi \in \fX(G)$. Assume that $\chi$ is represented by $(M,\pi)$ where $M$ and $\pi$ are as in \eqref{eq:higher_corank_cuspidal_data_1} and \eqref{eq:higher_corank_cuspidal_data_2}. We say that $\chi$ is \emph{twisted Rankin-Selberg regular}, if for any $1 \le i \le s,1 \le j \le t$, we have $\pi_{n,i} \ne \pi_{n+m,j}$. Let $\widetilde{\fX}_{\mathrm{RS}} \subset \fX(G)$ denote the set of twisted Rankin-Selberg regular cuspidal datum. We write $\cT_{\widetilde{\mathrm{RS}}}([G])$ for $\cT_{\widetilde{\fX}_{\mathrm{RS}}}([G])$.

The proof of the following corollary is parallel to the proof of Corollary \ref{cor:twisted_equal_rank}, and we omit the proof.
\begin{corollary} \label{cor:twisted_higher_corank}
    We have the following statements:
    \begin{enumerate}
        \item For $f \in \cT([G])$, there exists $C>0$ such that the integral defining $\widetilde{Z}^{\mathrm{RS}}(s,f)$ is convergent for $\mathrm{Re}(s) > C$ and defines a holomorphic function on $\cH_{>C}$.
        \item The linear functional $\widetilde{\cP}_{\mathrm{RS}}$ on $\cS_{\widetilde{\mathrm{RS}}}([G])$ extends (uniquely) by continuity to a continuous linear functional $\widetilde{\cP}_{\mathrm{RS}}$ on $\cT_{\widetilde{\mathrm{RS}}}([G])$.
        \item For any $f \in \cT_{\widetilde{\mathrm{RS}}}([G])$, the zeta integral $Z(\cdot,f)$ extends to an entire function. And we have
        \begin{equation*}
          \widetilde{\cP}_{\mathrm{RS}}(f) = \widetilde{Z}^{\mathrm{RS}}(0,f) 
        \end{equation*}
        \item For any $s \in \C$, the functional $\widetilde{Z}(s,\cdot)$ on $\cT_{\widetilde{\mathrm{RS}}}([G])$ is continuous.
    \end{enumerate}
\end{corollary}

By \eqref{eq:higher_corank_holomorphic}, we see that
\begin{num} 
    \item \label{eq:twsited_higher_corank_holomorphic}  The map $\widetilde{Z}^{\mathrm{RS}}(\cdot,\cdot): \C \to \cT'_{\widetilde{\mathrm{RS}}}([G]), s \mapsto (f \mapsto \widetilde{Z}^{\mathrm{RS}}(s,f) )$ is holomorphic.
\end{num}

\subsubsection{Euler decomposition}

Let $\tS$ be a finite set of places of $F$, let $\sigma = \sigma_n \boxtimes \sigma_{n+m}$ be a generic irreducible representation of $G(F_{\tS})$. For $W \in \cW(\sigma,\psi_{N,\tS})$, we define local (twisted) Rankin-Selberg integral of $W$ \cite{JPSS} as
\begin{equation*}
    \widetilde{Z}^{\mathrm{RS}}_{\tS}(s,W) := \int_{N_n(F_{\tS}) \backslash G_n(F_{\tS})}  W \left( w_{\ell}h^{-1}w_{\ell},\begin{pmatrix}
        h &  \\ & 1_m
    \end{pmatrix} \right) \lvert \det h \rvert^{s} \rd h.
\end{equation*}
The integral defining $\widetilde{Z}^{\mathrm{RS}}_{\tS}(s,W)$ is convergent when $\mathrm{Re}(s) \gg 0$ and has meromorphic continuation to $\C$. Moreover, by \cite{JPSS} and \cite{Jacquet09}, the quotient
\begin{equation*}
    \frac{\widetilde{Z}^{\mathrm{RS}}_{\tS}(s,W)}{L_{\tS}(s+\frac m2,\sigma^{\vee}_n \times \sigma_{n+m})}
\end{equation*}
is entire.

Let $P = P_n \times P_{n+m} \subset G$ be a standard parabolic subgroup and let $\pi = \pi_n \boxtimes \pi_{n+m}$ be a cuspidal automorphic representation of $M_P$. Assume that $(M_P,\pi)$ gives a twisted Rankin-Selberg regular cuspidal data $\chi$. Let $\Pi = \opn{Ind}_{P(\bA)}^{G(\bA)} \pi=\Pi_n \boxtimes \Pi_{n+m} $. 

For future use in \S \ref{ssec:n_m_Eisenstein}, we consider a section $\varphi \in \cA_{P,\pi_n \lvert \cdot \rvert^{\frac{n+m}{2}} \boxtimes \pi_{n+m} \lvert \cdot \rvert^{-\frac{n}{2}}}$. We write $E(\varphi)(g) = E(g,\varphi,0)$ for the Eisenstein series of $\varphi$. Note that $E(\varphi) \in \cT_{\widetilde{\mathrm{RS}}}([G])$.

Let $\tS$ be a sufficiently large set of places of $F$, that we assume to contain Archimedean places as well as the places where $\Pi$, $\psi$ or $\varphi$ is ramified. We then have a decomposition $W_{E(\varphi)} = W_{E(\varphi),\tS} W_{E(\varphi)}^{\tS}$ such that $W_{E(\varphi)}^{\tS}(1)=1$ and is fixed by $K^{\tS}$.

Note that $W_{E(\varphi),\tS} \in \cW(\Pi_{\tS},\psi_{N,\tS})$. By the unramified computation for the Rankin-Selberg integral, we have
\begin{equation} \label{eq:Rankin_Selberg_Euler}
    \widetilde{Z}^{\mathrm{RS}}(s,E(\varphi)) = (\Delta_{G_n}^{\tS,*})^{-1} \widetilde{Z}^{\mathrm{RS}}_{\tS}(s,W_{E(\varphi),\tS}) L^{\tS}(s-n,\Pi^{\vee}_n \times \Pi_{n+m})
\end{equation}

\section{Periods detecting $(n,n)$-Eisenstein series}
\label{sec:n_n}

\subsection{Statements of the main results}

\subsubsection{Notations}
\label{sssec:n_n_notation}

In \S \ref{sec:n_n}, we will use the following notations. Let $n \ge 1$ be a fixed integer, and let $G = G_{2n}$. Let $H=\Sp_{2n}$, regarded as a subgroup of $G$. Let $N$ denote the upper triangular unipotent subgroup of $G$ and let $N_H := N \cap H$.

Let $Q_n$ be the standard parabolic subgroup of $G$ with Levi component $ G_n\times G_n$. Note that $Q_{n}^H:=Q_n \cap H$ is the Siegel parabolic subgroup of $H$. The Levi component of $Q_n^H$ consists of elements of the form
\[
    \begin{pmatrix}
J{}^tg^{-1}J &  \\
 & g
\end{pmatrix}, \quad g\in G_n.
\]

For $0 \le r \le n$, we write $P_r$ for the standard parabolic subgroup whose Levi component is $G_1^{n-r} \times G_{2r} \times G_1^{n-r}$. Let $N_r$ denote the unipotent radical of $P_r$. We denote by $P^H_r = P_r \cap H$. Note that $P_0$ is the Borel subgroup of $G$ and $P_0^H$ is the Borel subgroup of $H$.

Let $\cP_{2r}$ denote the mirabolic subgroup of $\GL_{2r}$, it consists of elements of $\GL_{2r}$ with last row $(0,\cdots,0,1)$.  Let $\cP^H_{2r} := \cP_{2r} \cap \Sp_{2r}$. We regard $\Sp_{2r}$ as the subgroup $\begin{pmatrix}
    1_{n-r} & & \\ & h & \\ & & 1_{n-r}
\end{pmatrix}$ of $H$, where $h \in \Sp_{2r}$. We hence regard $\cP^H_{2r}$ as a subgroup of $H$ via the embedding $\cP^H_{2r} \subset \Sp_{2r} \subset H$.

Let $\Delta$ denote the BZSV quadruple \cite{MWZ1} $(G,H,\opn{std} \oplus \opn{std}^{\vee},1)$. Let $\psi_n$ denote the degenerate character on $N(\bA)$ defined by
\[
    \psi_n(u) = \psi \left( \sum_{ \substack{1 \le i \le 2n-1 \\ i \ne n}} u_{i,i+1} \right).
\]

For $1 \le r \le n$, we write $\psi_{N_r}$ for the restriction of $\psi_n$ to $N_r(\bA)$. For $f \in \cT([G])$, we put
\begin{equation*}
    f_{N_r,\psi}(g) := \int_{[N_r]} f(ug) \psi_{N_r}^{-1}(u) \rd u.
\end{equation*}

We write $K_H$ for the standard maximal compact subgroup of $H(\bA)$. For any semi-standard parabolic subgroup $Q$ of $H$, we have the Iwasawa decomposition $H(\bA) = Q(\bA)K_H$.

\subsubsection{The period}

For $f \in \cS([G])$ and $\Phi \in \cS(\bA_{2n})$, we define a bilinear map $\cS([G]) \times \cS(\bA_{2n}) \to \C$ by
\begin{equation*}
    \cP(f,\Phi) = \int_{[H]} f(h) \Theta(h,\Phi) \rd h.
\end{equation*}

By Lemma \ref{Theta moderate}, the integral defining $\cP$ is absolutely convergent and defines a continuous bilinear map on $\cS([G]) \times \cS(\bA_{2n})$.

\subsubsection{Zeta integral}

For every $f \in \cT([G])$, we associate the following degenerate Whittaker coefficient
\[
    V_f(g)= \int_{[N]}f(ug) \psi_n(u)^{-1} \rd u.
\]

For $f \in \cT([G])$ and $\Phi \in \cS(\bA_{2n})$ and $\lambda \in \fa^*_{Q_n,\C}$, we set
\[
    Z(\lambda, f, \Phi)= \int_{N_H(\bA) \backslash H(\bA)} V_f(h) \Phi(e_{2n}h) e^{\langle \lambda, H_{Q_n}(h) \rangle} \rd h
\]
provided by the expression converges absolutely. Denote the unique element in $\Delta_{Q_n}$ by $\alpha$. We define $s_\lambda := -\langle \lambda, \alpha^{\vee} \rangle$. Therefore $s_{\lambda}$ has the property
\[
     \exp \left( \langle \lambda, H_{Q_n}\begin{pmatrix}
J{}^tg^{-1}J &  \\
 & g
\end{pmatrix} \rangle    \right)= \lvert \det g \rvert ^{s_\lambda},
\]
thus inducing a linear map $\fa^*_{Q_n,\C} \to \C, \lambda \mapsto s_\lambda$.

The following two lemmas provide the convergence of zeta integral
\begin{lemma} \label{lemma: T convergence}
    Let $N \geq 0$. There exists $c_N >0$ such that
    \begin{enumerate}
        \item For every $f \in \cT_N([G])$ and $\Phi \in \cS(\bA_{2n})$, the expression defining $Z(\lambda, f, \Phi)$ converges absolutely when $\opn{Re}(s_{\lambda}) > c_N$ and defines a holomorphic function of $\lambda$ on the region $\opn{Re}(s_{\lambda})>c_N$;
        
        \item For every $\Phi \in \cS(\bA_{2n})$ and $\lambda \in \fa^*_{Q_n,\C}$ with $\opn{Re}(s_\lambda)>c_N$, the functional $f \in \cT_N([G]) \mapsto Z(\lambda, f, \Phi)$ is continuous.
    \end{enumerate}
\end{lemma}

\begin{lemma} \label{lemma: every lambda convergence}
    We have the following statements
    \begin{enumerate}
        \item For every $f \in \cS([G])$, $\Phi \in \cS(\bA_{2n})$ and $\lambda \in \fa^*_{Q_n,\C}$, the expression defining $Z(\lambda, f, \Phi)$ converges absolutely and defines an entire function in $\lambda$;
        
        \item For every $\lambda \in \fa^*_{Q_n,\C}$ and $\Phi \in \cS(\bA_{2n})$, the functional $f \in \cS([G]) \mapsto Z(\lambda,f,\Phi)$ is continuous;
    \end{enumerate}
\end{lemma}

Lemma \ref{lemma: T convergence} and Lemma \ref{lemma: every lambda convergence} will be proved in \S \ref{sssec:n_n_convergence_T} and \S \ref{sssec:n_n_convergence_S} respectively.

\subsubsection{$\Delta$-regular cuspidal datum}

Let $\chi \in \fX(G)$ be a cuspidal data, let $\chi^{M_{Q_n}}$ be the preimage of $\chi$ in $\fX(M_{Q_n}) = \fX(\GL_n \times \GL_{n})$. We say that $\chi$ is \emph{$\Delta$-regular}, if for any $\chi' \in \chi^{M_{Q_n}}$ is twisted Rankin-Selberg regular in the sense of \S \ref{sssec:twisted_equal_rank}. The reader can check that this definition is the same as the one given in \eqref{eq:intro_3}. We remark that $\Delta$ here stands for the quadruple defined in \S \ref{sssec:n_n_notation}. Note that any regular cuspidal data is $\Delta$-regular.

Let $\fX_{\Delta} \subset \fX(G)$ denote the set of $\Delta$-regular cuspidal data. We write $\cS_{\Delta}([G])$ (resp. $\cT_{\Delta}([G])$) for $\cS_{\fX_{\Delta}}([G])$ (resp. $\cT_{\fX_{\Delta}}([G])$).

\subsubsection{Main results}

For $\Phi \in \cS(\bA_{2n})$, we denote by $\Phi^{\flat} \in \cS(\bA_n)$ the restriction of $\Phi$ to $\{0\} \times \bA_n$.
\begin{theorem} \label{thm:n_n}
    We have the following statements
    \begin{enumerate}
        \item For any $\Phi \in \cS(\bA_{2n})$, the restriction of $\cP(\cdot,\Phi)$ to $\cS_{\Delta}([G])$ extends (uniquely) by continuity to a functional $\cP^{*}$ on $\cT_{\Delta}([G])$.
        \item For any $f \in \cT_{\Delta}([G])$ and $\Phi \in \cS(\bA_{2n})$, the map $\lambda \mapsto Z(\lambda,f,\Phi)$ extends to an entire function in $\lambda \in \fa^*_{{Q_n},\C}$. Indeed, for any $k \in K_H$, $(R(k)f)_{Q_n}|_{[G_{n} \times G_{n}]} \in \cT_{\widetilde{\mathrm{RS}}}([G_n \times G_{n}])$, and we have
        \begin{equation}
            Z(\lambda,f,\Phi) = \int_{K_H} \widetilde{Z}^{\mathrm{RS}}(s_{\lambda}+n+\frac 12,(R(k)f)_{Q_n}, (R(k)\Phi)^{\flat}) \rd k,
        \end{equation}
        here $(R(k)f)_{Q_n}$ means $(R(k)f)_{Q_n}|_{[G_n \times G_{n}]}$. 
        \item We have
        \begin{equation*}
            \cP^{*}(f,\Phi) = Z(0,f,\Phi).
        \end{equation*}
        \item The bilinear map $\cT_{\Delta}([G]) \times \cS(\bA_{2n}) \to \C, (f,\Phi) \mapsto \cP^*(f,\Phi)$ is continuous.
    \end{enumerate}
\end{theorem}

\subsection{Convergence of Zeta integrals}
\label{ssec:n_n_convergence}

\subsubsection{More zeta integrals}

For $f \in \cS([G])$ and $\Phi \in \cS(\bA_{2n})$ and $0 \le r \le n$. We define
\begin{equation*}
    Z_r(f,\Phi) = \int_{ N_r(\bA) \cP_{2r}^H(F) \backslash H(\bA)} f_{N_r,\psi}(h) \rd h.
\end{equation*}

Note that when $r=0$, $Z_r(f,\Phi) = Z(0,f,\Phi)$.

\begin{lemma} \label{lemma: Z_r convergence}
    For every $0 \leq r \leq n$ and $f \in \cS([G])$ and $\Phi \in \cS(\bA_{2n})$, the integral defining $Z_r(f,\Phi)$ converges absolutely.
\end{lemma}

\subsubsection{Proof of Lemma \ref{lemma: T convergence}}
\label{sssec:n_n_convergence_T}

\begin{proof} 
    By the Iwasawa decomposition $H(\bA)= P_0^H(\bA) K_{H}$, we need to show the existence of $c_N>0$ such that  
    \begin{equation} \label{equation: T convergence}
        \begin{split}
        \int_{K_H} \int_{(\bA^\times)^n}  \left \vert V_{f} (D(a_1, \cdots, a_n)k) \right \vert &\left \vert \Phi(a_1^{-1}e_{2n}k) \right \vert \delta_{P_0^H}(D(a_1, \cdots, a_n))^{-1}  \\
        & \prod_{i=1}^n \lvert a_i \rvert^{-\opn{Re}(s)}  \rd a_1 \cdots \rd a_n \rd k 
        \end{split}
    \end{equation}
    when $\opn{Re}(s) > c_N$. Where 
    \begin{equation*}
        D(a_1,\cdots,a_n) = \opn{diag}(a_1,\cdots,a_n,a_n^{-1},\cdots,a_1^{-1}).
    \end{equation*}
    The modular function is given by
    \[
       \delta_{P_0^H}(D(a_1, \cdots, a_n))= \prod_{i=1}^n \lvert a_i \rvert ^{2n-2i+2}.
    \]
    We apply Lemma \ref{lemma: estimate Fourier coeff} \eqref{T Fourier coeff}, then for every $N_1 >0$, we have
    \[
        \lvert V_{f} (D(a_1, \cdots, a_n)k) \rvert \ll \prod_{i=1}^{n-1} \| a_i a_{i+1}^{-1} \|_{\bA}^{-N_1} \prod_{i=1}^n \| a_i \|_{G_1}^{2N}
    \]
    for $(k, a_1, \cdots, a_n) \in K_H \times (\bA^\times)^n$. Note for every $N_1 >0$, we have $\lvert \Phi(a_1^{-1}e_{2n}k) \rvert \ll \| a_1^{-1} \|_\bA^{-N_1}$ for $(k, a_1) \in K_H \times \bA^\times$. Note for every $N_2>0$, there exists $N_1>0$ such that
    \[
       \prod_{i=1}^{n-1} \| a_i a_{i+1}^{-1} \|_{\bA}^{-N_1}  \| a_1^{-1} \|_\bA^{-N_1} \ll  \prod_{i=1}^n \| a_i^{-1} \|_{\bA}^{-N_2}.
    \]
    Then for every $N_2>0$, \eqref{equation: T convergence} is essentially bounded by
    \begin{equation} \label{eq:n_n_convergence_1}
        \prod_{i=1}^n \int_{\bA^\times} \| a_i \|_{G_1}^{2N} \| a_i^{-1} \|_\bA^{-N_2} \lvert a_i \rvert^{-\opn{Re}(s)-(2n-2i+2)} \rd a_i
    \end{equation}
    Since there exists $M>0$ such that $\|a_i\|_{G_1} \ll \max\{ \lvert a_i \rvert^M, \lvert a_i \rvert^{-M}  \}$, the convergence of \eqref{eq:n_n_convergence_1} follows from Corollary \ref{cor:A_norm_integral}.
\end{proof}

\subsubsection{Proof of Lemma \ref{lemma: Z_r convergence}}
\label{sssec:n_n_convergence_zeta_integra}

\begin{proof} 
     We assume that $r >0$, the case $r=0$ will be covered in Lemma \ref{lemma: every lambda convergence}. By the Iwasawa decomposition $H(\bA)=P_r^H(\bA) K_{H}$, we need to show the convergence of
    \begin{equation}
        \begin{split}
        \int_{K_H} \int_{(\bA^\times)^{n-r}} \int_{\cP^H_{2r}(F) \backslash \Sp_{2r}(\bA)}  &\left \vert f_{N_r,\psi}(D(a_1, \cdots, a_{n-r}, h)k) \right \vert \lvert \Phi(a_1^{-1}e_{2n}k) \rvert \\
        &\delta_{P_r^H}(D(a_1, \cdots, a_{n-r}, h))^{-1} \rd h \rd a_1 \cdots \rd a_{n-r} \rd k, \label{equation: convergence}
        \end{split}
    \end{equation}
    where
    \[
        D(a_1, \cdots, a_{n-r}, h)=\opn{diag}(a_1,\cdots,a_{n-r},h,a_{n-r}^{-1},\cdots,a_1^{-1}).
    \]
    and the modular function $\delta_{P_r^H}$ is given by
    \[
        \delta_{P_r^H}(D(a_1, \cdots, a_{n-r}, h))= \prod_{i=1}^{n-r} \lvert a_i \rvert ^{2n+2-2i}.
    \]
    We now apply Lemma \ref{lemma: estimate Fourier coeff} \eqref{Schwartz Fourier coefficient}. For this, we note $\psi_{N_r}= \psi \circ l$, where $l : N_{r} \to \mathbb{G}_a$ sends $u \in N_{r}$ to $u_{1,2}+ \cdots u_{n-r,n-r+1}+ u_{n+r, n+r+1}+ \cdots + u_{2n-1, 2n}$. One can check that
    \[
         \prod_{i=1}^{n-r-1} \| a_i a_{i+1}^{-1} \|_\bA \| a_{n-r} e_{2r}h \|_{\bA_{2r}}  \ll \| \opn{Ad}^*(D(a_1, \cdots, a_{n-r}, h)^{-1})l \|_{V_{P_r}, \bA}.
    \]
    Therefore, by \ref{lemma: estimate Fourier coeff} \eqref{Schwartz Fourier coefficient}, we can find $c > 0$ such that for every $N_1, N_2 >0$ we have
    \begin{align*}
        &\lvert f_{N_r,\psi}(D(a_1,\cdots,a_{n-r},h)k) \rvert \\
        & \ll \prod_{i=1}^{n-r-1} \| a_i a_{i+1}^{-1} \|_\bA^{-N_1}\| a_{n-r} e_{2r}h \|_{\bA_{2r}}^{-N_1} \prod_{i=1}^{n-r} \| a_i \|_{G_1}^{-2N_2} \| h \|_{G_{2r}}^{-N_2} \delta_{P_r}(D(a_1, \cdots, a_{n-r}, h))^{-cN_2}
    \end{align*}
    for $(k, a_1, \cdots, a_{n-r}, h) \in K_{H} \times (\bA^\times)^r \times \Sp_{2r}(\bA)$. The modular function is
    \[
        \delta_{P_r}(D(a_1, \cdots, a_{n-r}, h))= \prod_{i=1}^r \lvert a_i \rvert^{4n-4i+2}.
    \]
    On the other hand, for every $N_1 >0$, we have
    \[
        \lvert (R(k)\Phi)(a_1^{-1}e_{2n}) \rvert \ll \| a_1^{-1} \|_\bA^{-N_1}, \quad (k, a_1) \in K_{H} \times \bA^\times.
    \]
    One can check for every $N_3 >0$, there exists $N_1 >0$ such that
    \[
        \prod_{i=1}^{n-r-1} \| a_i a_{i+1}^{-1} \|_\bA^{-N_1} \| a_{n-r} e_{2r}h \|_{\bA_{2r}}^{-N_1}\| a_1^{-1} \|_\bA^{-N_1} \ll \prod_{i=1}^{n-r} \| a_i^{-1} \|_\bA^{-N_3} \| e_{2r} h \|_{\bA_{2r}}^{-N_3}.
    \]
    Then we deduce the existence of $c >0$ such that for every $N_3, N_2 >0$, \eqref{equation: convergence} is essentially bounded by the product of 
    \begin{equation} \label{eq: integral on H}
        \int_{\cP_{2r}^H(F) \backslash \Sp_{2r}(\bA)} \| h \|_{G_{2r}}^{-N_2} \| e_{2r}h \|_{\bA_{2r}}^{-N_3} \rd h
    \end{equation}
    and 
    \begin{equation} \label{eq: integral on G_1}
        \prod_{i=1}^{n-r} \int_{\bA^\times} \| a_i^{-1} \|^{-N_3}_\bA \lvert a_i \rvert^{-(4n-4i+2)cN_2-(2n-2i+2)} \rd a_i
    \end{equation}
    By Lemma \ref{Theta moderate}, there exists $N_0 >0$, such that for every $N_3 \geq N_0$, we have
    \begin{align*}
        &\int_{\cP_{2r}^H(F) \backslash \Sp_{2r}(\bA)} \| h \|_{G_{2r}}^{-N_2} \| e_{2r}h \|_{\bA_{2r}}^{-N_3} \rd h \\
        &= \int_{[\Sp_{2r}]}\| h \|_{G_{2r}}^{-N_2} \left( \sum_{v \in F_{2r} \backslash \{0\} } \| v h \|^{-N_3}_{\bA_{2r}} \right) \rd h \ll \int_{[\Sp_{2r}]}\| h \|_{\Sp_{2r}}^{N_0-N_2} \rd h
    \end{align*}
    Therefore, by Corollary \ref{cor:A_norm_integral}, the integral \eqref{eq: integral on G_1} and \eqref{eq: integral on H} are absolutely convergent when $N_3 \gg N_2 \gg 0$.
\end{proof}

\subsubsection{Proof of Lemma \ref{lemma: every lambda convergence}}
\label{sssec:n_n_convergence_S}

\begin{proof} 
    By the same argument of the proof of Lemma \ref{lemma: Z_r convergence}, the integral defining $Z(\lambda,f,\Phi)$ is bounded by
    \begin{equation*}
        \prod_{i=1}^{n} \int_{\bA^\times} \| a_i^{-1} \|_{\bA}^{-N_3} \lvert a_i \rvert^{-(4n-4i+2)cN_2-(2n-2i+2+\opn{Re}(s_{\lambda}))},
    \end{equation*}
    which is absolutely convergent when $N_3 \gg N_2 \gg  \max\{0, \opn{Re}(s_{\lambda}) \}$ by Corollary \ref{cor:A_norm_integral}.
\end{proof}

\subsection{Unfolding}
\label{ssec:n_n_unfolding}

\subsubsection{Main result}

In  \S \ref{ssec:n_n_unfolding}, we prove the following proposition.

\begin{proposition} \label{prop:n_n_unfolding}
    For any $f \in \cS_{\Delta}([G])$ and $\Phi \in \cS(\bA_{2n})$, we have 
    \begin{equation*}
        \cP(f,\Phi) = Z(0,f,\Phi)
    \end{equation*}
\end{proposition}

\subsubsection{A result of Offen}

We say a cuspidal data $\chi \in \fX(G)$ is \emph{even}, if $\chi$ can be represented by $(M,\pi)$, where
\begin{equation*}
    M = \GL_{n_1} \times \GL_{n_1} \times \cdots \times \GL_{n_k} \times \GL_{n_k}
\end{equation*}
and 
\begin{equation*}
    \pi = \pi_1 \boxtimes \pi_1 \boxtimes \cdots \boxtimes \pi_k \boxtimes \pi_k
\end{equation*}
according to this decomposition. We denote by $\fX_{\mathrm{even}}$ the set of even cuspidal datum, and denote its complement by $\fX_{\mathrm{even}}^c$.

\begin{theorem}[Offen] \label{thm:symplectic_period}
    The symplectic period is vanishing on $\cS_{\fX_{\mathrm{even}}^c}([G])$. That is, for any $f \in \cS_{\fX_{\mathrm{even}}^c}([G])$,
    \begin{equation*}
        \int_{[H]} f(h) \rd h = 0.
    \end{equation*}
\end{theorem}

\begin{proof}
    It is proved in \cite[Proposition 6.2, Theorem 6.3]{Offen06} (see also \cite[\S 7.1]{LO18}) that if $\chi \in \fX(G)$ is not even and $f \in \fO_\chi$ is a pseudo-Eisenstein series, then $\int_{[H]} f(h) \rd h = 0$. By Lemma \ref{lemma: pseudo dense}, for any $f \in \cS_{\chi}([G])$, we have $\int_{[H]} f(h) \rd h = 0$.

    Finally, for any $f \in \cS_{\fX^c_{\mathrm{even}}}([G])$, by Theorem \ref{thm:decomposition_cuspidal_support}, $f$ can be written as $\sum_{\chi \in \fX_{\mathrm{even}}^c} f_{\chi}$, where $f_{\chi} \in \cS([G])$ and the sum is absolutely convergent in $\cS([G])$. The theorem follows.
\end{proof}

\begin{corollary} \label{cor: constant term vanish}
    Let $a \geq 2b$ be integers and $\chi \in \fX(G_a)$. Let $P = MN$ be a standard parabolic subgroup of $G_a$ such that $G_{2b}$ is a factor of its Levi component $M$. For $\chi' \in \fX(M_P)$, denote by $\chi'_{2b} \in \fX(G_{2b})$ the component of $\chi'$ at $G_{2b}$. Suppose that for any $\chi' \in \chi^M$, $\chi'_{2b}$ is not even. Regard $\Sp_{2b} \subset G_{2b}$ as a subgroup of $M$, then for any $f \in \cS_\chi([G_a])$, we have
    \[
        \int_{[\Sp_{2b}]}f_P(h) \rd h=0.
    \]
\end{corollary}

\begin{proof}
    Note that $\delta_P$ is trivial on $\Sp_{2b}$, it follows from Lemma \ref{lemma: estimate constant term} that the restriction of $f_P$ to $[\Sp_{2b}]$ belongs to $\cS([\Sp_{2b}])$. The integral hence converges absolutely. By Lemma \ref{lemma: constant term}, we have $f_{P} \in \cT_{\chi}([G_{a}]_{P})$. Let $\kappa \in C_c^\infty(\fa_{P})$ be a compactly supported smooth function on $\fa_{P}$ with $\kappa(0)=1$. By \eqref{eq: truncation of constant term} and Lemma \ref{lemma: stable under truncation}, we conclude that
    \[
        (\kappa \circ H_{P}) \cdot f_{P} \in \cT_{\chi}([G_{a}]_{P}) \cap \cS([G_{a}]_{P})= \cS_\chi([G_{a}]_{P}).
    \]
    By Lemma \ref{lemma: restriction to levi}
    and Lemma \ref{lemma: restriction to component}, the restriction of $(\kappa \circ H_{P}) \cdot f_{P}$ to $[G_{2b}]$ belongs to $\sum_{\chi' \in \chi^M} \cS_{\chi'_{2b}}([G_{2b}])$. It follows from Theorem \ref{thm:symplectic_period} that
    \[
       \int_{[\Sp_{2b}]} f_P(h) \rd h = \int_{[\Sp_{2b}]}(\kappa \circ H_{P})(h)f_{P}(h) \rd h =0.
    \]
\end{proof}

\subsubsection{Proof of Proposition \ref{prop:n_n_unfolding}}

Proposition \ref{prop:n_n_unfolding} is implied by the following Lemma
\begin{lemma}
    For any $f \in \cS_{\Delta}([G])$ and $\Phi \in \cS(\bA_{2n})$, we have
    \[
        Z_0(f,\Phi) = Z_1(f,\Phi) = \cdots = Z_n(f,\Phi) = \cP(f,\Phi)
    \]
\end{lemma}

\begin{proof}
    We show that $Z_r(f,\Phi)=Z_{r+1}(f,\Phi)$ for $0 \le r \le n-1$, the proof of $Z_n(f,\Phi)=\cP(f,\Phi)$ is similar and is left to the reader.

    Let $r \ge 1$, we denote by $U_r$ the unipotent radical of the parabolic subgroup of $\GL_{2r}$ with Levi component $G_1 \times G_{2r-2} \times G_1$, which we regard $U_r$ as the subgroup $\begin{pmatrix}
        1_{n-r} & & \\ & u & \\ & & 1_{n-r}
    \end{pmatrix}, u \in U_r$ of $G$. Let $U_r^H := U_r \cap H$. Note that $U_r^H$ is a normal subgroup of $U_r$. By an abuse of notation, we write $\psi$ for the character $u \mapsto \psi(u_{12}+u_{2r-1,2r})$ of $U_r(\bA)$. 
    
    By Fourier inversion on the compact abelian group $U_{r+1}(\bA)/U_{r+1}(F)U_{r+1}^H(\bA)$, we have
    \[
        \int_{[U_{r+1}^H]} f_{N_{r+1},\psi}(uh) \rd h =(f_{N_{r+1},\psi})_{U_{r+1}} + \sum_{\gamma \in \cP^H_{2r}(F) \backslash \Sp_{2r}(F)} (f_{N_{r+1},\psi})_{U_{r+1},\psi}
    \]
    for all $h \in H(\bA)$, where we have set
    \begin{equation*}
    \begin{split}
    	&(f_{N_{r+1},\psi})_{U_r}(g) = \int_{[U_r]} f_{N_{r+1},\psi}(ug) \rd u, \\
    	&(f_{N_{r+1},\psi})_{U_r,\psi}(g) = \int_{[U_r]} f_{N_{r+1},\psi}(ug) \psi(u) \rd u = f_{N_r,\psi}(g).
    \end{split}
    \end{equation*}
    Therefore, we formally have
    \begin{equation} \label{equation: induction}
        Z_r(f,\Phi)= Z_{r+1}(f,\Phi)+ F_r(f,\Phi)
    \end{equation}
    where we have set
    \begin{equation*}
        F_r(f,\Phi) = \int_{\Sp_{2r}(F)N^H_r(\bA) \backslash H(\bA)} (f_{N_{r+1},\psi})_{U_r}(h) \Phi(e_{2n}h) \rd h.
    \end{equation*}
    To verify \eqref{equation: induction}, we need to show
    \begin{lemma} \label{lemma: F_r converge}
    For every $0 \leq r \leq n-1$, $f \in \cS([G])$ and $\Phi \in \cS(\bA_{2n})$, the integral defining $F_r(f,\Phi)$ converges absolutely.
    \end{lemma}
    
    \begin{proof} [Proof of Lemma \ref{lemma: F_r converge}] 
    By the same arguments as the proof of Lemma \ref{lemma: Z_r convergence}, there exists $c>0$ such that for every $N, N_2 >0$, the integral defining $F_r(f,\Phi)$ is essentially bounded by the product of
    \[
     \int_{[\Sp_{2r}]} \| h \|_{G_{2r}}^{-N_2} \rd h
    \]
    and 
    \[
       \prod_{i=1}^{n-r}  \int_{\bA^\times} \| a_i \|_{G_1}^{-2N_2} \| a_i^{-1} \|^{-N}_\bA \lvert a_i \rvert^{-(4n-4i+2)cN_2-(2n-2i+2)} \rd a_i.
    \]
    We can take $N \gg N_2 \gg 0$ such that these integrals converge.
\end{proof}
     
 Let $R_r$ be the standard parabolic subgroup of $G$ with Levi component $G_{n-r} \times G_{2r} \times G_{n-r}$. Let $V_k$ denote the upper triangular unipotent subgroup of $G_k$. Then 
 \begin{equation*}
     (f_{N_{r+1},\psi})_{U_r}(g) = \int_{[V_{n-r}]} \int_{[V_{n-r}]} f_{R_r} \left( \begin{pmatrix}
        u_1 & & \\ & 1_{2r} & \\ & & u_2
     \end{pmatrix}g \right) \psi^{-1}(u_1) \psi^{-1}(u_2) \rd u_1 \rd u_2.
 \end{equation*}
    Let $R_r^H := R_r \cap H$. Using the Iwasawa decomposition $H(\bA) = R_r^H(\bA) K_H$, we can write $F_r(f,\Phi)$ as
     \begin{equation*}
         \begin{split}
             F_r(f,\Phi) &= \int_{V_{n-r}(\bA) \backslash G_{n-r}(\bA)} \int_{[\Sp_{2r}]} \int_K  \int_{[V_{n-r}]} \int_{[V_{n-r}]} f_{R_r} \left( \begin{pmatrix}
                 u_1 g & & \\ & h & \\ & & u_2 J_{n-r} {}^t g^{-1} J_{n-r}
             \end{pmatrix} k \right) \\
             & \delta_{R_r^H} \begin{pmatrix}
                 g &  & \\ & h & \\ & & J_{n-r} {}^t g^{-1} J_{n-r}
             \end{pmatrix}^{-1} \rd u_1 \rd u_2 \rd k \rd h \rd g.
         \end{split}
     \end{equation*}
     Therefore the vanishing is implied by
     \begin{equation} \label{eq:n_n_unfolding_1}
         \int_{[\Sp_{2r}]} f_{R_r} \begin{pmatrix}
             1_{n-r} & & \\ & h & \\ & & 1_{n-r}
         \end{pmatrix} \rd h
     \end{equation}
     vanishes for any $f \in \cS_{\Delta}([G])$. It suffices to prove that for any $\Delta$-regular cuspidal data and any $f \in \cS_{\chi}([G])$, the integral \eqref{eq:n_n_unfolding_1} vanishes. However, one can check directly that for any $\chi’ = (\chi_1,\chi_2,\chi_3) \in \chi^{M_{P_r}} \subset \fX(G_{n-r} \times G_{2r} \times G_{n-r})$, $\chi_2$ is not even. Therefore, the integral vanishes by Corollary \ref{cor: constant term vanish}.
\end{proof}

\subsection{Proof of Theorem \ref{thm:n_n} }
\label{ssec:n_n_proof}

Let $\fX_{\Delta}^{M_{Q_n}}$ denote the preimage of $\fX_{\Delta}$ in $\fX(M_{Q_n})$. By the definition of $\Delta$-regularity, we have $\fX_{\Delta}^{M_{Q_n}} \subset \fX_{\widetilde{\mathrm{RS}}}(M_{Q_n})$ (indeed it is easy to see that this is an equality).

 Let $f \in \cT_{\Delta}([G])$ and $\Phi \in \cS(\bA_{2n})$, By the Iwasawa decomposition $H(\bA)= Q_n^H(\bA)K_{H}$, when $Z(\lambda,f,\Phi)$ is absolutely convergent, we have
\begin{equation} \label{eq:n_n_proof_1}
    Z(\lambda, f, \Phi)= \int_{K_{H}} \int_{N_n(\bA) \backslash G_n(\bA)} V_{R(k)f}\begin{pmatrix}
J{}^tg^{-1}J &  \\
 & g
\end{pmatrix} \Phi(e_{2n}g) \lvert \det g \rvert^{s_\lambda+n+1} \rd g \rd k .
\end{equation}
From the definition of the degenerate Whittaker coefficient, we have
\[
    V_{f}\begin{pmatrix}
g^\prime &  \\
 & g
\end{pmatrix}= W^{M_{Q_n}}_{f_{Q_n}}\begin{pmatrix}
g^\prime &  \\
 & g
\end{pmatrix}, \quad g,g' \in G_n(\bA)
\]
Then we can write
\[
    Z(\lambda, f, \Phi)=  \int_{K_{H}} \widetilde{Z}^{\mathrm{RS}}(s_\lambda+n+\frac12, (R(k)\Phi)^\flat , {(R(k)f)_{Q_n}}) \rd k, \quad \opn{Re}(s_\lambda) \gg 0.
\]
Then it follows from Corollary \ref{cor:twisted_equal_rank} that $\widetilde{Z}^{\mathrm{RS}}(s_\lambda+n+\frac12, (R(k)\Phi)^\flat, {(R(k)f)_{Q_n}})$ extends to an entire function of $s_\lambda$. Applying Lemma \ref{lemma: composition holomorphic} \eqref{item: bilinear form} with
\[
    W= \cS(\bA_n), \quad V= \cT_{\widetilde{\mathrm{RS}}}([G_n \times G_n]), \quad X=K_H \times \cS(\bA_{2n}) \times \cT_{\Delta}([G]),
\]
the holomorphic map
\[
    s \in \C \mapsto \widetilde{Z}^{\mathrm{RS}}(s+n+\frac12, \cdot, \cdot ) \in \opn{Bil}(\cS(\bA_n), \cT_{\widetilde{{\mathrm{RS}}}}([G_n \times G_n])),
\]
and continuous maps
\begin{align*}
    &(s, k, f, \Phi) \in \C \times K_{H} \times \cS(\bA_{2n}) \times \cT_{\Delta}([G])  \mapsto (R(k)\Phi)^\flat \in \cS(\bA_n), \\
    &(s, k, f, \Phi) \in \C \times K_{H} \times \cS(\bA_{2n}) \times \cT_{\Delta}([G])  \mapsto (R(k)f)_{Q_n} \in \cT_{\widetilde{\mathrm{RS}}}([G_n \times G_n])
\end{align*}
we deduce that the map
\[
    (s, k, f, \Phi)  \in \C \times K_{H} \times \cS(\bA_{2n}) \times \cT_{\Delta}([G]) \mapsto \widetilde{Z}^{\mathrm{RS}}(s+n+\frac12, (R(k)\Phi)^\flat,  {(R(k)f)_{Q_n}}) \in \C
\]
is continuous and holomorphic in the first variable. Then it follows from Lemma \ref{lemma: integrate compact} that 
the integral
\begin{equation*}
    \int_{K_{H}} \widetilde{Z}^{\mathrm{RS}}(s_\lambda+n+\frac12, (R(k)\Phi)^\flat,  {(R(k)f)_{Q_n}}) \rd k
\end{equation*}
is holomorphic in $s \in \C$. Therefore $Z(\lambda,f,\Phi)$ extends to an entire function. This proves (2). Lemma \ref{lemma: integrate compact} also implies $Z(0, f, \Phi)$ is continuous in $(f,\Phi) \in \cT_{\Delta}([G]) \times \cS(\bA_{2n})$. Moreover by Lemma \ref{lemma: every lambda convergence} and Proposition \ref{prop:n_n_unfolding},
\[
    Z(0, f, \Phi) = \cP(f,\Phi).
\]
Therefore for $\Phi \in \cS(\bA_{2n})$, the $f \mapsto Z(0,f,\Phi)$ provides a continuous extension of $\cP(\cdot,\Phi)$ to $\cT_{\Delta}([G])$, this proves (1),(3) and (4).

\subsection{Periods of $\Delta$-regular Eisenstein series}

\subsubsection{Local zeta integral}

Fix a place $v$ of $F$, let $\Pi_M = \Pi \boxtimes \Pi'$ be an irreducible generic representation of $M_{Q_n}(F_v)$. Recall the space $\opn{Ind}_{Q_n(F_v)}^{G(F_v)} \cW(\Pi_M ,\psi_v)$ defined in \S \ref{sssec:prelim_Jacquet_integral}. For $W^M \in \opn{Ind}_{Q_n(F_v)}^{G(F_v)} \cW(\Pi_{M},\psi_v)$ and $\Phi \in \cS(F_{v,2n})$ and $\lambda \in \fa_{Q_n,\C}^*$, we define
\[
    Z_v(\lambda, W^M,\Phi)= \int_{N_H(F_v)\backslash H(F_v)} W^M(h) \Phi(e_{2n}h)e^{\langle \lambda, H_{Q_n}(h) \rangle} \rd h,
\]
provided by the integral is absolutely convergent.

\begin{lemma} \label{lem:n_n_zeta_integarl}
    \begin{enumerate}
        \item For any $W^M \in \opn{Ind}_{Q_n(F_v)}^{G(F_v)}$ and $\Phi \in \cS(F_{v,2n})$, the integral defining $Z_v(\lambda,W^M,\Phi)$ is absolutely convergent when $\opn{Re}(s_{\lambda}) \gg 0$ and has a meromorphic continuation to $\fa_{Q_n,\C}^*$.
    \end{enumerate}
\end{lemma}

\begin{proof}
    Using the Iwasawa decomposition, we can formally write $Z_v(\lambda,W^M,\Phi)$ as
    \begin{equation} \label{eq:n_n_zeta_integral_1}
    \begin{split}
    Z_v(\lambda, W^M,\Phi)&= \int_{K_{H,v}} \int_{N_n(F_v) \backslash G_n(F_v)}R(k)W\begin{pmatrix}
    J{}^tg^{-1}J &  \\
     & g
    \end{pmatrix}(R(k)\Phi)^\flat(e_n g) \lvert \det g \rvert^{s_\lambda+n+1} \rd g \rd k \\
    &=\int_{K_{H,v}}\widetilde{Z}_v^{\mathrm{RS}}(s_\lambda+n+\frac12, (R(k)\Phi)^\flat, R(k)W^M|_{M_{Q_n}(F_v)}).
    \end{split}
    \end{equation}

    The convergence of the zeta integral hence follows from the convergence of the usual Rankin-Selberg integral \cite{JPSS}, \cite{Jacquet09}.

    When $v$ is non-Archimedean, by \eqref{eq:n_n_zeta_integral_1}, $Z_v(\lambda,W^M,\Phi)$ is essentially a finite sum of twisted Rankin-Selberg integral, hence has meromorphic continuation. We now assume $v$ is Archimedean.  Let $\cO(\C)$ denote the entire function on $\C$ with the usual compact-open topology. By \cite[Theorem 2.3]{Jacquet09}, the map    
    \begin{equation} \label{eq: zeta integral continuous}
        \cW(\lvert \cdot \rvert^{\frac{n}{2}}\Pi \boxtimes \lvert \cdot \rvert^{-\frac{n}{2}}\Pi^\prime,\psi) \times \cS(F_{v,n}) \to \cO(\C), (W, \Phi') \mapsto \left(s \mapsto \frac{\widetilde{Z}^{\mathrm{RS}}_v(s +n+\frac12,\Phi', W)}{L_v(s+1,\Pi^\vee \times \Pi^{\prime} )} \right)
    \end{equation}
    is continuous. Therefore, the map
    \[
        K_{H,v} \to \cO(\C): k \mapsto \frac{\widetilde{Z}_v^{\mathrm{RS}}(s+n+\frac12, (R(k)\Phi)^\flat, R(k)W )}{L_\tS(s+1,\Pi^\vee \times \Pi^{\prime} )}
    \]
    is continuous. Combining with \eqref{eq:n_n_zeta_integral_1}, the meromorphicity follows from Lemma \ref{lemma: integrate compact}.
\end{proof}

We finally remark by the same argument, Lemma \ref{lem:n_n_zeta_integarl} still holds if we replace $v$ by a finite set $\tS$ of places of $F$.

\subsubsection{Fixed points}
\label{sssec:n_n_fixed_points}

Let $P=M_PN_P$ be a standard parabolic subgroup, we write $M_P$ as  $G_{n_1}\times \cdots \times G_{n_k}$. Let $\pi= \pi_1 \boxtimes \cdots \boxtimes \pi_k$ be a cuspidal unitary automorphic representation of $M_P$ (central character not necessarily trivial on $A_M^{\infty}$) such that the cuspidal datum $\chi$ represented by $(M_P, \pi_0)$ (see \S \ref{sssec:intro_main_result}) is $\Delta$-regular.

We write $\opn{Fix}(\pi)$ for the set of permutations $\sigma: \{1,2,\cdots,k\} \to \{1,2,\cdots,k\}$ such that there exists $1 \le t \le k$ with:
\begin{enumerate}
    \item $n_{\sigma(1)} + \cdots + n_{\sigma(t)}=n, n_{\sigma(t+1)} + \cdots + n_{\sigma(k)}=n$.
    \item $\sigma(1) < \cdots < \sigma(t)$ and $\sigma(t+1) < \cdots < \sigma(k)$.
\end{enumerate}

We also introduce the following notations
\begin{enumerate}
    \item $P_{\sigma}$ the standard parabolic subgroup of $G_{2n+m}$ with $M_{P_{\sigma}} = G_{n_{\sigma(1)}} \times \cdots \times G_{n_{\sigma(k)}}$,
    \item $P_{\sigma,n}$ (resp. $P_{\sigma,n}'$) the standard parabolic subgroup of $G_n$ with Levi subgroup $G_{n_{\sigma(1)}} \times \cdots \times G_{n_{\sigma(t)}}$ (resp. $G_{n_{\sigma(t+1)}} \times \cdots \times G_{n_{\sigma(k)}}$),
    \item $\pi_{\sigma} = \pi_{\sigma(1)} \boxtimes \cdots \boxtimes \pi_{\sigma(k)}$, which is a cuspidal automorphic representation of $M_{P_\sigma}$,
    \item $\pi_{\sigma,n} = \pi_{\sigma(1)} \boxtimes \cdots \boxtimes \pi_{\sigma(t)}$ and $\pi_{\sigma,n}' = \pi_{\sigma(t+1)} \boxtimes \cdots \boxtimes \pi_{\sigma(k)}$.
    \item $\Pi_{\sigma,n} = \opn{Ind}_{P_{\sigma,n}(\bA)}^{G_n(\bA)} \pi_{\sigma,n}$ and $\Pi_{\sigma,n}' = \opn{Ind}_{P_{\sigma,n+m}(\bA)}^{G_{n+m}(\bA)} \pi_{\sigma,n+m}$.
\end{enumerate}

\subsubsection{$L$-functions}

Let $\sigma \in \opn{Fix}(\pi)$, we put
\begin{equation*}
    L(s,T_{\sigma} \check{X}) := L(s,\Pi_{\sigma,n} \times \Pi_{\sigma,n}^{\prime,\vee}) L(s,\Pi_{\sigma,n}^{\vee} \times \Pi'_{\sigma,n}). 
\end{equation*}

\subsubsection{Periods of Eisenstein series}

Let $\varphi \in \Pi= \opn{Ind}_{P(\bA)}^{G(\bA)} \pi= \cA_{P, \pi}$ and write $E(\varphi)(g)=E(g, \varphi, 0)$ for the Eisenstein series of $\varphi$. Then $E(\varphi) \in \cT_{\Delta}([G])$.
\begin{theorem} \label{thm:n_n_Eisenstein}
    We have
    \begin{equation*}
        \begin{split}
        \cP(E(\varphi)) = &(\Delta_{H}^{\tS,*})^{-1} L(1,\pi,\widehat{\fn}_P^-)^{-1} \sum_{\sigma \in \opn{Fix}(\pi)} L^\tS(1,T_{\sigma} \check{X}) L_\tS(1,\pi_\sigma,\widehat{\fn}_{P_{\sigma}}^-)Z_\tS(\lambda, \Phi_\tS, \Omega^{M_{Q_n}}_\tS(N_{\pi, \tS}(\sigma)W^M_{\varphi, \tS})).
        \end{split}
    \end{equation*}
\end{theorem}

Recall the $L$-function $L(s,\pi,\widehat{\fn}_P^-$ defined in \eqref{eq:L_function_on_n_P_-}.

\begin{proof}
 By the constant term formula for Eisenstein series, we have
\[
    (R(k)E(\varphi))_{Q_n}= \sum_{w \in W(P ; Q_n)}E^{Q_n}(M(w)R(k)\varphi), \quad k \in K_H.
\]
By Theorem \ref{thm:n_n}, we can write
\begin{equation}
    \begin{split}
    \cP^*(E(\varphi),\Phi)&= \int_{K_{H}}\widetilde{Z}^{\mathrm{RS}}(n+\frac12, {(R(k)E(\varphi))_{Q_n}},(R(k)\Phi)^\flat) \rd k  \\
    &= \sum_{w \in W(P; Q_n)}\int_{K_{H}}\widetilde{Z}^{\mathrm{RS}}(n+\frac12,  E^{Q_n}(M(w)R(k)\varphi),(R(k)\Phi)^\flat ) \rd k.
    \end{split}
\end{equation}
where the second equality holds because for each $w \in W(P;Q_n)$ and each $k \in K_H$, $E^{Q_n}(M(w)R(k)\varphi) \in \cT_{\widetilde{\mathrm{RS}}}([M_{Q_n}])$.

Assume that there exists $1 \le i < j \le k$ such that $\pi_i^{\vee} = \pi_j$, then by the computation of Fourier coefficient of Eisenstein series \cite[\S 4]{Shahidi}, the Whittaker function of $E^{Q_n}(M(w)R(k)\varphi)|_{[M_Q]}$ vanishes for any $k \in K_H$, therefore $\cP^*(E(\varphi),\Phi)$ vanishes. Therefore, from now on, we assume that $\pi_i^{\vee} \ne \pi_j$ for any $i \ne j$. In particular, for any finite subset $\tS$ of places of $F$, the partial $L$-function $L^{\tS}(s,\pi_i^{\vee} \times \pi_j)$ is regular (and non-vanishing) at $s=1$.

Let $\tS$ be a sufficiently large finite set of places of $F$, which we assume to contain Archimedean places as well as the places where $\Pi$ or $\psi$ is ramified. We also assume $\varphi$ is fixed by $K^\tS$ and $\Phi$ can be written as $\Phi= \Phi_\tS \Phi^\tS$, where $\Phi^\tS$ is the characteristic function of $\cO_{F, 2n}^{\tS}$ and $\Phi_\tS \in \cS(F_{\tS, 2n})$.

Note that there is a bijection between and $W(P; Q_n)$ and $\opn{Fix} (\pi)$, where each $w$ corresponds to the $\sigma$ such that $w M_P w^{-1}= M_{P, \sigma}$. In the following, we fix an arbitrary $w \in W(P; Q_n)$, and corresponding $\sigma \in \opn{Fix}(\pi)$. It's clear that under this correspondence, one can identify the representation $w\pi$ of $wM_P w^{-1}$ with the representation $\pi_\sigma$ of $M_{P, \sigma}$. 

Note that the restriction of $E^{Q_n}(M(w)R(k)\varphi)$ to $[M_{Q_n}]$ belongs to $\lvert \cdot \rvert^{\frac{n}{2}}\Pi_{\sigma,n} \boxtimes \lvert \cdot \rvert^{-\frac{n}{2}}\Pi_{\sigma,n}^\prime$, then it follows from \eqref{eq:Rankin_Selberg_Euler_equal_rank} and \eqref{eq:n_n_zeta_integral_1} that
\begin{equation} \label{eq:n_n_Eisenstein_1}
    \int_{K_H} \widetilde{Z}^{\mathrm{RS}}(n+\frac12, (R(k)\Phi)^\flat,  E^{Q_n}(M(w)R(k)\varphi) ) \rd k =(\Delta_{H}^{\tS, *})^{-1}L^\tS(1,\Pi_{\sigma,n}^{\vee} \times \Pi_{\sigma,n}')  Z_{\tS}(0,W^{M_{Q_n}}_{M(w)E(\varphi),\tS},\Phi_{\tS}).
\end{equation}

By \eqref{eq: Induction Jacquet functional} and \eqref{eq: intertwine Whittaker}, we have
\begin{equation*}
    \begin{split}
     W^{M_{Q_n}}_{E^{Q_n}(M(w)\varphi),\tS}
     &=\frac{1}{L^{\tS}(1,\pi_{\sigma,n}, \widehat{\fn}_{P_{\sigma,n}}^{-}) L^{\tS}(1,\pi'_{\sigma,n}, \widehat{\fn}_{P'_{\sigma,n}}^{-})}  \frac{L(1,\pi_{\sigma},\widehat{\fn}_{P_{\sigma}}^-)}{L(1,\pi,\widehat{\fn}_P^-)} \Omega^{M_{Q_n}}_\tS(N_{\pi,\tS}(w)W^{M_P}_{\varphi,\tS}) \\
     &= \frac{L^{\tS}(1,\Pi_{\sigma,n} \times \Pi^{\prime,\vee}_{\sigma,n})}{L(1,\pi,\widehat{\fn}_P^-)} L_{\tS}(1,\pi_{\sigma},\widehat{\fn}_{P_{\sigma}}^-) \Omega^{Q_n}_\tS(N_{\pi,\tS}(w)W^{M_P}_{\varphi,\tS})
     \end{split}
\end{equation*}

Therefore we can write the left hand side of \eqref{eq:n_n_Eisenstein_1} as
\begin{equation*}
    (\Delta_{H}^{\tS,*})^{-1} \frac{L^\tS(1, \Pi_{\sigma, n}^\vee \times \Pi_{\sigma, n}^\prime) L^{\tS}(1,\Pi^{\vee}_{\sigma,n} \times \Pi'_{\sigma,n})}{L(1,\pi,\widehat{\fn}_P^-)} L_{\tS}(1,\pi_{\sigma},\widehat{\fn}_{P_{\sigma}}^-) Z_{\tS}(0, \Omega^{Q_n}_\tS(N_{\pi,\tS}(w)W^{M_P}_{\varphi,\tS}),\Phi).
\end{equation*}
This finishes the proof
\end{proof}

\section{Periods detecting $(n,n+m)$-Eisenstein series}
\label{sec:higher_corank}

\subsection{Statement of the main results}

\subsubsection{Notations}
\label{sssec:n_m_notation}
 In \S \ref{sec:higher_corank}, fix integers $n \ge 0$ and $m \ge 1$. Let $G = G_{2n+m}$ and let $H = \Sp_{2n}$. We regard $H$ as the subgroup $\begin{pmatrix}
    h &  \\ & 1
\end{pmatrix}, h \in H$ of $G$. We will study period related to the quadruple $\Delta := \Delta_{n,m} = (G,H,0,\iota_{n,m})$, where $\iota_{n,m}: \SL_2 \to G$ is the representation $\mathbf{1}^{n+1} \oplus \opn{Sym}^{m-1}$ of $\SL_2$. 

 Let $N = N_{2n+m}$ denote the upper triangular unipotent subgroup of $G$ and let $N^H := N \cap H$.

 For $0 \le r \le n$, let $P_r := P_r^{n,m}$ be the parabolic subgroup of of $G$ whose Levi component is isomorphic to $G_1^{n-r} \times G_{2r} \times G_1^{n+m-r}$. Let $P_r^H := P_r \cap H$, it is a parabolic subgroup of $H$ whose unipotent radical is $N_r^H := N_r \cap H$. The Levi component of $P_r^H$ is $\Sp_{2r} \times G_1^{n-r}$.

Let $\cP_{2r}$ denote the mirabolic subgroup of $\GL_{2r}$, it consists of elements of $\GL_{2r}$ with last row $(0,\cdots,0,1)$.  Let $\cP^H_{2r} := \cP_{2r} \cap \Sp_{2r}$. We regard $\Sp_{2r}$ as the subgroup $\begin{pmatrix}
    1_r & & \\ & h & \\ & & 1_r
\end{pmatrix}$ of $H$, where $h \in \Sp_{2r}$. We hence regard $\cP^H_{2r}$ as a subgroup of $H$ via the embedding $\cP^H_{2r} \subset \Sp_{2r} \subset H$.

Let $\psi_n$ denote the degenerate character 
\begin{equation*}
    N(\bA) \ni u \mapsto \psi \left( \sum_{\substack{1 \le i \le n+m-1 \\ i \ne n }} u_{i,i+1} \right)
\end{equation*}
of $N(\bA)$ which is trivial on $N(F)$. 

We also denote by $N_{n+1}$ the unipotent radical of the parabolic $P_{n+1}$ of $G$ whose Levi component is $G_{n+1} \times G_1^{n+m-1}$.

For $1 \le r \le n+1$, We write $\psi_{N_r}$ for the restriction of $\psi_n$ to $N_r(\bA)$. For $f \in \cT([G])$, we put
\begin{equation*}
    f_{N_r,\psi}(g) := \int_{[N_r]} f(ug) \psi_{N_r}^{-1}(u) \rd u.
\end{equation*}

\subsubsection{The period}

For $f \in \cS([G])$, we define the period $\cP := \cP_{\Delta}$ on $\cS([G])$ by
\begin{equation*}
    \cP(f) = \int_{[H]} f_{N_{n+1},\psi}(h) \rd h.
\end{equation*}
By Lemma \ref{lemma: constant term}, the integral
\begin{equation*}
    \int_{[H]} \int_{[N_{n+1}]} \lvert f(nh) \rvert \rd n \rd h
\end{equation*}
is absolutely convergent. Hence the integral defining $\cP(f)$ is absolutely convergent.

\subsubsection{Zeta integral}
 For $f \in \cT([G])$, we associate the \emph{degenerate Whittaker coefficient}
\begin{equation*}
    V_f(g) = \int_{[N]} f(ug) \psi_n^{-1}(u) \rd u.
\end{equation*}

Note that $V_f(g) = f_{N_0,\psi}(g)$. Let $Q_{n}$ denote the parabolic subgroup of $G$ of type $(n,n+m)$. For $f \in \cT([G])$, and for $\lambda \in \fa_{Q_n,\C}^*$, we set
\begin{equation*}
    Z(\lambda,f) = \int_{N_{H}(\bA) \backslash H(\bA)} V_f(h) e^{\langle \lambda,H_{Q_n}(h) \rangle} \rd h,
\end{equation*}
provided by the integral is absolutely convergent. Note that $Z(\lambda,f)$ only depends on $s_{\lambda} := \langle \lambda, \alpha^{\vee} \rangle \in \C$.

\begin{lemma} \label{lem:n_m_convergence_1}
    We have the following statements:
    \begin{enumerate}
        \item for any $\lambda \in \fa_{Q_n,\C}^*$, the integral defining $Z(\lambda,f)$ is absolutely convergent, and it defines an entire function on $\fa_{Q_n,\C}^*$,
        \item for any $\lambda \in \fa_{Q_n,\C}^*$, the map $f \mapsto Z(\lambda,f)$ is continuous on $\cS([G])$.
    \end{enumerate}
\end{lemma}

\begin{lemma} \label{lem:n_m_convergence_2}
    Let $N>0$, then there exists $c_N>0$ such that
    \begin{enumerate}
        \item The integral defining $Z(\lambda,f)$ is absolutely convergent when $f \in \cT_{N}([G])$ and $\mathrm{Re}(s_{\lambda}) > c_N$, and defines a holomorphic function of $\lambda$ on the region $\mathrm{Re}(s_{\lambda})>c_N$.
        \item For fix $\lambda$ such that $\mathrm{Re}(s_{\lambda})>c_N$. The map $\cT_N([G]) \ni f  \mapsto Z(\lambda,f) \in \C$ is continuous.
    \end{enumerate}
\end{lemma}
The proof of the two lemmas is parallel to proofs given in \S \ref{ssec:n_n_convergence}, so we leave it to the readers.

\subsubsection{$\Delta$-regular cuspidal datum}

Let $\chi \in \fX(G)$ be a cuspidal data, let $\chi^{M_{Q_n}}$ be the preimage of $\chi$ in $\fX(M_{Q_n}) = \fX(\GL_n \times \GL_{n+m})$. We say that $\chi$ is \emph{$\Delta$-regular}, if for any $\chi' \in \chi^{M_{Q_n}}$ is twisted Rankin-Selberg regular in the sense of \ref{sssec:twisted_higher_corank}. We remark that $\Delta$ here stands for the quadruple defined in \S \ref{sssec:n_m_notation}. Note that any regular cuspidal data is $\Delta$-regular.

Let $\fX_{\Delta} \subset \fX(G)$ denote the set of $\Delta$-regular cuspidal data. We write $\cS_{\Delta}([G])$ (resp. $\cT_{\Delta}([G])$) for $\cS_{\fX_{\Delta}}([G])$ (resp. $\cT_{\fX_{\Delta}}([G])$).

\subsubsection{Main results}

\begin{theorem} \label{thm:n_m}
    We have the following statements
    \begin{enumerate}
        \item The restriction of $\cP$ to $\cS_{\Delta}([G])$ extends (uniquely) by continuity to a functional $\cP^{*}$ on $\cT_{\Delta}([G])$.
        \item For any $f \in \cT_{\Delta}([G])$, the map $\lambda \mapsto Z(\lambda,f)$ extends to an entire function in $\lambda \in \fa^*_{{Q_n},\C}$. Indeed, for any $k \in K_H$, $(R(k)f)_{Q_n}|_{[G_{n} \times G_{n+m}]} \in \cT_{\widetilde{\mathrm{RS}}}([G_n \times G_{n+m}])$, and we have
        \begin{equation}
            Z(\lambda,f) = \int_{K_H} \widetilde{Z}^{\mathrm{RS}}(s_{\lambda}+n+1,(R(k)f)_{Q_n}) \rd k,
        \end{equation}
        here $(R(k)f)_{Q_n}$ means $(R(k)f)_{Q_n}|_{[G_n \times G_{n+m}]}$
        \item We have
        \begin{equation*}
            \cP^{*}(f) = Z(0,f).
        \end{equation*}
    \end{enumerate}
\end{theorem}
The proof of the Proposition will be given in \S \ref{ssec:n_m_proof}.

\subsection{Unfolding} \label{ssec:n_m_unfolding}

In \S \ref{ssec:n_m_unfolding}, we show the following result:

\begin{proposition} \label{prop:n_m_unfolding}
    For any $f \in \cS_{\Delta}([G])$, we have
    \begin{equation*}
        \cP(f) = Z(0,f).
    \end{equation*}
\end{proposition}

\subsubsection{More zeta integrals}

For $f \in \cS([G])$, we put 
\begin{equation*}
    Z_r(f) = \int_{N_r^H(\bA)\cP_{2r}^H(F) \backslash H(\bA)} f_{N_r.\psi_r}(h) \rd h.
\end{equation*}
Note that $Z_0(f) = Z(0,f)$.

\begin{proposition}
    For any $f \in \cS([G])$, the integral defining $Z_r(f)$ is absolutely convergent.
\end{proposition}

\begin{proof}
    The proof of the proposition follows the same line of the proof of Lemma \ref{lemma: Z_r convergence}, and we omit the proof.
\end{proof}

\subsubsection{}

Proposition \ref{prop:n_m_unfolding} will directly follow from the following lemma.

\begin{lemma} \label{lem:n_m_unfolding}
    For any $f \in \cS_{\Delta}([G])$, we have
    \begin{equation*}
         Z_0(f) = Z_1(f) = \cdots = Z_n(f) = \cP(f).
    \end{equation*}
\end{lemma}

\begin{proof}
    It suffices to prove $Z_n(f) = \cP(f)$ and $Z_r(f) = Z_{r+1}(f)$ for any $0 \le r \le n-1$. We prove the latter, and the former one follows from a similar argument.

    Let $r \ge 1$, we denote by $U_r$ the unipotent radical of the parabolic subgroup of $\GL_{2r}$ with Levi component $G_1 \times G_{2r-2} \times G_1$, which we regard $U_r$ as the subgroup $\begin{pmatrix}
        1_{n-r} & & \\ & u & \\ & & 1_{n+m-r}
    \end{pmatrix}, u \in U_r$ of $G$. Let $U_r^H := U_r \cap H$. By an abuse of notation, we write $\psi$ for the character $u \mapsto \psi(u_{12}+u_{2r-1,2r})$ of $U_r(\bA)$.

    Using Fourier analysis on the compact abelian group  $U_{r+1}(\bA)/U_{r+1}^H(\bA)U_{r+1}(F)$, we can write
    \begin{equation*}
        \int_{[U_{r+1}^H]} f_{N_{r+1},\psi}(uh) \rd h = (f_{N_{r+1},\psi})_{U_{r+1}} + \sum_{\gamma \in \cP^H_{2r}(F) \backslash H_{2r}(F)} (f_{N_{r+1},\psi})_{U_{r+1},\psi}.
    \end{equation*}
    where
    \begin{equation*}
    \begin{split}
    	&(f_{N_{r+1},\psi})_{U_r}(g) = \int_{[U_r]} f_{N_{r+1},\psi}(ug) \rd u, \\
    	&(f_{N_{r+1},\psi})_{U_r,\psi}(g) = \int_{[U_r]} f_{N_{r+1},\psi}(ug) \psi(u) \rd u = f_{N_r,\psi}(g).
    \end{split}
    \end{equation*}
    
    Therefore, we formally have
    \begin{equation} \label{eq:n_m_unfolding_1}
    	Z_{r+1}(f) = Z_r(f) + F_r(f),
    \end{equation} 
    where
    \begin{equation*}
    	F_r(f) = \int_{\Sp_{2r}(F)N_r(\bA) \backslash H(\bA)} (f_{N_{r+1},\psi})_{U_r}(h) \rd h.
    \end{equation*}
    
    To verify \eqref{eq:n_m_unfolding_1}, we need to show that the integral defining $F_r(f)$ is absolutely convergent. The proof follows the same line of the proof of Lemma \ref{lemma: F_r converge}, and we omit the proof. Therefore, we are reduced to show that for $0 \le r \le n-1$, we have $F_r(f) = 0$.
    
    Let $R_r$ denote the parabolic subgroup of $G$ with Levi component $G_{n-r} \times G_{2r} \times G_{n+m-r}$. $V_k$ denote the upper triangular unipotent subgroup of $G_k$. Write $R_r^H := R_r \cap H$. Using Iwasawa decomposition $H(\bA)  = R_r^H(\bA) K_H$, we can write the integral defining $F_r(f)$ as
    \begin{equation*}
        \begin{split}
    	F_r(f) =  &\int_{V_{n-r}(\bA) \backslash G_{n-r}(\bA)}  \int_{[\Sp_{2r}]} \int_K  \int_{[V_{n-r}]} \int_{[V_{n+m-r}]} f_{R_r} \left(.\begin{pmatrix} u_1g . & & \\ & h &  \\ & & u_2 \begin{pmatrix} J_{n-r} {}^t g^{-1} J_{n-r} & \\ & 1_{m}  \end{pmatrix}   \end{pmatrix}    k \right) \\
        &\delta_{R_r^H}^{-1} \begin{pmatrix}. g & & \\ & h & \\ & & J_{n-r} {}^tg^{-1}J_{n-r}  \end{pmatrix}^{-1} \rd u_1 \rd u_2 \rd k \rd h \rd g.
        \end{split}
    \end{equation*}
    
    Therefore, the vanishing of $F_r(f)$ is implied by
    \begin{equation} \label{eq:n_m_unfolding_2}
    	\int_{[\Sp_{2r}]} f_{R_r} \begin{pmatrix} 1_{n-r} & & \\ & h & \\ & & 1_{n+m-r}  \end{pmatrix} \rd h
    \end{equation}
    vanishes for any $f \in \cS_{\Delta}([G])$. This follows from Corollary \ref{cor: constant term vanish}.
\end{proof}

\subsection{Proof of Theorem \ref{thm:n_m}}
\label{ssec:n_m_proof}

Let $\fX_{\Delta}^{M_{Q_n}}$ denote the preimage of $\fX_{\Delta}$ in $\fX(M_{Q_n})$, then we have $\fX_{\Delta}^{M_{Q_n}} \subset \fX_{\widetilde{\mathrm{RS}}}(M_{Q_n})$ 

Therefore, it follows from Lemma \ref{lemma: constant term}, for any $k \in K_H$ and $f \in \cT_{\Delta}([G])$, we have $R(k)f|_{[M_{Q_n}]} \in \cT_{\widetilde{\mathrm{RS}}}([M^{Q_n}])$.

Using Iwasawa decomposition as in \eqref{eq:n_n_proof_1}, we see that for any $f \in \cT_{\Delta}([G])$, the equality
\begin{equation} \label{eq:n_m_proof_1}
    Z(\lambda,f) = \int_{K_H} \widetilde{Z}^{\mathrm{RS}}(s_\lambda+n+1,((R(k)f)_{Q_n})|_{[M_{Q_n}]}) \rd k
\end{equation}
holds when $\opn{Re}(s_\lambda) \gg 1$.

By Corollary \ref{cor:twisted_higher_corank}, for $f' \in \cT_{\widetilde{\mathrm{RS}}}([G])$, $\widetilde{Z}^{\mathrm{RS}}(s,f')$ has holomorphic continuation to $s \in \C$ and is continuous in $f'$. Therefore, $\widetilde{Z}^{\mathrm{RS}}(s_{\lambda}+n+1,(R(k)f)_{Q_n})|_{[M_{Q_n}]}$ is defined for any $\lambda \in \fa_{Q_n,\C}^*$. We argue as in \S \ref{ssec:n_n_proof} that the right hand side of \eqref{eq:n_m_proof_1} is holomorphic in $\lambda$, and for any $\lambda \in \fa_{Q_n,\C}^*$, $f \mapsto Z(\lambda,f)$ is continuous in $f \in \cT_{\Delta}([G])$. Therefore (2) is proved.

By Proposition \ref{prop:n_m_unfolding}, the functional $Z(0,\cdot)$ on $\cT_{\Delta}([G])$ coincides with $\cP$ on the dense subspace $\cS_{\Delta}([G])$. Therefore $f \mapsto Z(0,f)$ provides an extension of $\cP$ to $\cT_{\Delta}([G])$, (1) and (3) then follow.

\subsection{Periods of $\Delta$-regular Eisenstein series}
\label{ssec:n_m_Eisenstein}

\subsubsection{Local zeta integral}

Fix a place $v$ of $F$, let $\Pi_M = \Pi_n \boxtimes \Pi_{n+m}$ be an irreducible generic representation of $M_{Q_n}(F_v)$. For $W^M \in \opn{Ind}_{Q_n(F_v)}^{G(F_v)} \cW(\Pi_{M},\psi_v)$ and $\lambda \in \fa_{Q_n,\C}^*$, we define
\[
    Z_v(\lambda, W^M)= \int_{N_H(F_v)\backslash H(F_v)} W^M(h) e^{\langle \lambda, H_{Q_n}(h) \rangle} \rd h,
\]
provided by the integral is absolutely convergent.

\begin{lemma} \label{lem:n_m_zeta_integarl}
     For any $W^M \in \opn{Ind}_{Q_n(F_v)}^{G(F_v)}$, the integral defining $Z_v(\lambda,W^M)$ is absolutely convergent when $\opn{Res}(s_{\lambda}) \gg 0$ and has a meromorphic continuation to $\fa_{Q_n,\C}^*$.
\end{lemma}
We omit the proof which is parallel to the proof of Lemma \ref{lem:n_n_zeta_integarl}. By Iwasawa decomposition, we can write
\begin{equation} \label{eq:n_m_zeta_integral_1}
    Z_v(\lambda,W^M) = \int_{K_H} \widetilde{Z}_v^{\mathrm{RS}}(s_{\lambda}+n+1,(R(k)W^M)|_{M_{Q_n}(F_v)})
\end{equation}

\subsubsection{Fixed points}

Let $P=M_PN_P$ be a standard parabolic subgroup, we write $M_P$ as  $G_{n_1}\times \cdots \times G_{n_k}$. Let $\pi= \pi_1 \boxtimes \cdots \boxtimes \pi_k$ be a cuspidal unitary automorphic representation of $M_P$ (central character not necessarily trivial on $A_M^{\infty}$) such that the cuspidal datum $\chi$ represented by $(M_P, \pi_0)$ is $\Delta$-regular.

Recall the set $\opn{Fix}(\pi)$ defined in \ref{sssec:intro_precise}. 

\begin{enumerate}
    \item $P_{\sigma}$ the standard parabolic subgroup of $G_{2n+m}$ with $M_{P_{\sigma}} = G_{n_{\sigma(1)}} \times \cdots \times G_{n_{\sigma(k)}}$,
    \item $P_{\sigma,n}$ (resp. $P_{\sigma, n+m}$) the standard parabolic subgroup of $G_n$ (resp. $G_{n+m}$) with Levi subgroup $G_{n_{\sigma(1)}} \times \cdots \times G_{n_{\sigma(t)}}$ (resp. $G_{n_{\sigma(t+1)}} \times \cdots \times G_{n_{\sigma(k)}}$),
    \item $\pi_{\sigma} = \pi_{\sigma(1)} \boxtimes \cdots \boxtimes \pi_{\sigma(k)}$, which is a cuspidal automorphic representation of $M_{P_\sigma}$,
    \item $\pi_{\sigma,n} = \pi_{\sigma(1)} \boxtimes \cdots \boxtimes \pi_{\sigma(t)}$ and $\pi_{\sigma,n+m} = \pi_{\sigma(t+1)} \boxtimes \cdots \boxtimes \pi_{\sigma(k)}$.
    \item $\Pi_{\sigma,n} = \opn{Ind}_{P_{\sigma,n}(\bA)}^{G_n(\bA)} \pi_{\sigma,n}$ and $\Pi_{\sigma,n+m} = \opn{Ind}_{P_{\sigma,n+m}(\bA)}^{G_{n+m}(\bA)} \pi_{\sigma,n+m}$.
\end{enumerate}

For $\sigma \in \opn{Fix}(\pi)$, we put
\begin{equation*}
    L(s,T_{\sigma} \check{X}) := L(s,\Pi_{\sigma,n}^{\vee} \times \Pi_{\sigma,n+m}) L(s,\Pi_{\sigma,n} \times \Pi^{\vee}_{\sigma,n+m}).
\end{equation*}

\subsubsection{Periods of Eisenstein series}

Let $\varphi \in \Pi= \opn{Ind}_{P(\bA)}^{G_{2n}(\bA)} \pi= \cA_{P, \pi}$ and write $E(\varphi)(g)=E(g, \varphi, 0)$ for the Eisenstein series of $\varphi$. Note that $E(\varphi) \in \cT_{\Delta}([G])$. 

\begin{theorem} \label{thm:n_m_Eisenstein}
    Let $\tS$ be a sufficiently large finite set of places of $F$, that contains Archimedean places and the places where $\Pi$ or $\psi$ is ramified. We also assume that $\varphi$ is fixed by $K^{\tS}$, and we decompose $W^{M_P}_{\varphi}$ as $W^{M_P}_{\varphi}=W^{M_P}_{\varphi,\tS} W^{M_P,\tS}_{\varphi} $. Then period $\cP^{*}(E(\varphi))$ is equal to
    \begin{equation} \label{eq:n_m_Eisenstein}
       (\Delta_H^{\tS,*})^{-1} L(1,\pi,\widehat{\fn}_P^-)^{-1} \sum_{\sigma \in \opn{Fix}(\pi)} L^{\tS}(1,\Pi,T_\sigma \check{X}) L_{\tS}(1,\pi_{\sigma},\widehat{\fn}_{P_{\sigma}}^-) Z_{\tS}(0,\Omega^{M_{Q_n}}_{\tS}(N_{\pi,\tS}(\sigma)W_{\varphi,\tS}^{M_P}).
    \end{equation}
\end{theorem}

\begin{proof}
    The proof is parallel to the proof of Theorem \ref{thm:n_n_Eisenstein}, so we will be brief.

    By the constant term formula for Eisenstein series andTheorem \ref{thm:n_m}, we can write
\begin{equation} \label{eq:n_m_Eisenstein_2}
    \begin{split}
    \cP^*(E(\varphi)) = \sum_{w \in W(P; Q_n)}\int_{K_{H}}\widetilde{Z}^{\mathrm{RS}}(n+1, E^{Q_n}(M(w)R(k)\varphi) ) \rd k.
    \end{split}
\end{equation}

If there exists $1 \le i < j \le k$ such that $\pi_i^{\vee} = \pi_j$, then both side of \eqref{eq:n_m_Eisenstein} is 0, therefore from now on we assume that $\pi_i \ne \pi_j$ for $i \ne j$.

Let $\tS$ be a sufficiently large finite set of places of $F$, which we assume to contain Archimedean places as well as the places where $\Pi$ or $\psi$ is ramified. We also assume $\varphi$ is fixed by $K^\tS$. Note that there is a bijection between $W(P; Q_n)$ and $\opn{Fix} (\pi)$, where each $w$ corresponds to the $\sigma$ such that $w M_P w^{-1}= M_{P, \sigma}$. In the following, we fix an arbitrary $w \in W(P; Q_n)$, and corresponding $\sigma \in \opn{Fix}(\pi)$.  

Note that the restriction of $E^{Q_n}(M(w)R(k)\varphi)$ to $[M_{Q_n}]$ belongs to $\lvert \cdot \rvert^{\frac{n+m}{2}}\Pi_{\sigma,n} \boxtimes \lvert \cdot \rvert^{-\frac{n}{2}}\Pi_{\sigma,n+m}$, then it follows from \eqref{eq:Rankin_Selberg_Euler} and \eqref{eq:n_m_zeta_integral_1} that
\begin{equation} \label{eq:n_m_Eisenstein_1}
    \int_{K_H} \widetilde{Z}^{\mathrm{RS}}(n+1,  E^{Q_n}(M(w)R(k)\varphi) ) \rd k =(\Delta_{H}^{\tS, *})^{-1}L^\tS(1,\Pi_{\sigma,n}^{\vee} \times \Pi_{\sigma,n+m})   Z_{\tS}(0,W^{M_{Q_n}}_{M(w)E(\varphi),\tS}).
\end{equation}

By \eqref{eq: Induction Jacquet functional} and \eqref{eq: intertwine Whittaker}, we have
\begin{equation} \label{eq:n_m_Eisenstein_3}
    \begin{split}
     W^{M_{Q_n}}_{E^{Q_n}(M(w)\varphi),\tS}
     = \frac{L^{\tS}(1,\Pi_{\sigma,n} \times \Pi^{\vee}_{\sigma,n+m})}{L(1,\pi,\widehat{\fn}_P^-)} L_{\tS}(1,\pi_{\sigma},\widehat{\fn}_{P_{\sigma}}^-) \Omega^{Q_n}_\tS(N_{\pi,\tS}(w)W^{M_P}_{\varphi,\tS})
     \end{split}
\end{equation}
The theorem then follows from \eqref{eq:n_m_Eisenstein_2}, \eqref{eq:n_m_Eisenstein_1}, and \eqref{eq:n_m_Eisenstein_3}.
\end{proof}

\section{Truncation operator and the regularized period}
\label{sec:truncation}

\subsection{Notations}

Let $H=\Sp_{2n}$. We fix an upper triangular Borel subgroup $P_0'$ of $H$, let $\fa_{P_0'} := \fa_0'$, and $\Delta_0' = \Delta_{P_0'}$

Let $G=G_{2n+1}$. For a semi-standard parabolic subgroup $P \subset G$, let $\fa_P^+$ be the subset of $X \in \fa_P$ such that $\langle X,\alpha \rangle > 0$ for any $\alpha \in \Delta_P$.

For any semi-standard parabolic subgroups $P \subset Q$, let $\widehat{\tau}_P^Q$ be the usual characteristic function of a cone on $\fa_P$ defined in \cite[\S 5]{Arthur78}.

\subsection{The case $m=1$}

The case $m=1$ is taking the $\Sp_{2n}$ period of an automorphic form on $\GL_{2n+1}$. In the work \cite{Zydor19} of Zydor, he defined a regularized period of an automorphic form on a reductive group over any reductive subgroup.

Let $G = G_{2n+1}$ and $H = \Sp_{2n}$. Zydor's regularization was written down explicitly in \cite[\S 3.2]{LX} in this case, which we also briefly review here.

Let $\cF'$ be the set of standard parabolic subgroups of $H$. For each $P' \in \cF'$, there is a unique semi-standard parabolic subgroup of $G$ such that $\fa_P^+ \cap \fa_{P'}^+ \ne \varnothing$. If we write $P' = P(\lambda)$ via the dynamical method, where $\lambda$ is a cocharacter of $H$. Then $P$ can also be characterized as $P = P(\lambda^G)$, where $\lambda^G$ denotes the corresponding cocharacter of $G$. In the following, we will also denote by a standard parabolic subgroup of $H$ with a letter with a ${}^{\prime}$, and the corresponding parabolic subgroup of $G$ will be denoted by the same letter without ${}^\prime$.

Let $f \in \cT([G])$, we define
\begin{equation*}
    \Lambda^T f(h) = \sum_{P' \in \cF'} \varepsilon_{P'} \sum_{\gamma \in P'(F) \backslash H(F)} \widehat{\tau}_{P'}(H_{P'}(\gamma h)-T_{P'}) f_P(\gamma h).
\end{equation*}

By \cite[Theorem 3.9]{Zydor19} (see also \cite[Theorem 3.2.2]{LX}), when $T$ is sufficiently positive, $\Lambda^T f \in \cS^0([H])$, moreover, the map $f \in \cT([G]) \mapsto \Lambda^T f \in \cS^0([H])$ is continuous. For such $T$, we define
\begin{equation*}
    \cP^T(f) := \int_{[H]} \Lambda^T f(h) \rd h.
\end{equation*}

More generally, for $Q' \in \cF'$ and $f \in \cT(Q(F) \backslash G(\bA))$ (see \cite[\S 4.3]{BLX} for a definition), we define $\Lambda^{T,Q'} f$ by
\begin{equation*}
    \Lambda^{T,Q'}f(h)= \sum_{\substack{P' \in \cF' \\ P' \subset Q'}} \varepsilon_{P'}^{Q'} \sum_{\gamma \in P'(F) \backslash Q'(F)} \widehat{\tau}_{P'}^{Q'}(H_{P'}(\gamma h)-T_{P'}) f_P(\gamma h).
\end{equation*}
We can similarly show that $\Lambda^{T,Q'} \in \cS^0([H]_{Q'}^1)$ and the map $f \in \cT(Q(F) \backslash G(\bA)) \to \Lambda^{T,Q'}f \in \cS^0([H]_{Q'}^1)$ is continuous. 

There is also a variant of truncation operator for Levi subgroup. Let $Q' \in \cF'$ and $f \in \cT([M_{Q'}])$ and $T \in \fa_0'$, we define
\begin{equation*}
    \Lambda^{T,M_{Q'}} f(h) =  \sum_{\substack{P' \in \cF' \\ P' \subset Q'}} \varepsilon_{P'}^{Q'} \sum_{\gamma \in (M_{Q'} \cap P'(F)) \backslash M_{Q'}(F)} \widehat{\tau}_{P'}^{Q'}(H_{P'}(\gamma h)-T_{P'}) f_{P \cap M_Q}(\gamma h).
\end{equation*}
Since $\delta_{P}^{Q,-1}$ is bounded on $\{ h \in M_{Q'}(\bA) \mid \widehat{\tau}_{P'}^{Q'}(H_{P'}(h) -T) = 1 \}$. By Lemma \ref{lemma: estimate constant term}, for $f \in \cS([M_{Q}])$, the integral
\begin{equation*}
    \int_{[M_{Q'}]_{P \cap M_{Q'}}} \widehat{\tau}_{P'}^{Q'}(H_{P'}(h)-T_{P'}) f_{P \cap M_Q}(h)
\end{equation*}
is absolutely convergent. As a consequence,
\begin{num}
    \item \label{eq:Schwartz_function_truncated_period} For $f \in \cS([M_{Q'}])$ we have
    \begin{equation*}
        \int_{[M_{Q'}]} \Lambda^{T,M_{Q'}} f(h) \rd h = \sum_{\substack{P' \in \cF' \\ P' \subset Q'}} \varepsilon_{P'}^{Q'} \int_{[M_{Q'}]_{P' \cap M_{Q'}}} \widehat{\tau}_{P'}^{Q'}(H_{P'}(h)-T_{P'}) f_{P \cap M_Q}(h) \rd h.
    \end{equation*}
\end{num}
Similarly,
\begin{num}
    \item \label{eq:Schwartz_function_truncated_period_1} For $f \in \cS([M_{Q'}])$ we have
    \begin{equation*}
        \int_{[M_{Q'}]^1} \Lambda^{T,M_{Q'}} f(h) \rd h = \sum_{\substack{P' \in \cF' \\ P' \subset Q'}} \varepsilon_{P'}^{Q'} \int_{[M_{Q'}]^1_{P' \cap M_{Q'}}} \widehat{\tau}_{P'}^{Q'}(H_{P'}(h)-T_{P'}) f_{P \cap M_Q}(h) \rd h.
    \end{equation*}
\end{num}

We say that $T \in \fa_0' \to \infty$ if $\langle T, \alpha \rangle \to \infty$ for any $\alpha \in \Delta_0'$. Therefore, when $T \to \infty$, $\tau_{P'}(H_{P}(h)-T_{P'}) \to 0$ for any $h \in H(\bA)$. Therefore, by the dominated convergence theorem, we see that
\begin{num}
    \item \label{eq:Schwartz_function_truncated_period_limit}    For any $f \in \cS([G])$, we have
    \begin{equation*}
        \lim_{T \to \infty} \cP^T(f) = \cP(f).
    \end{equation*}
\end{num}

Let $M$ be a Levi subgroup of $G$. We write $\fX_{\Delta}^M$ for the preimage of $\fX_{\Delta}$ in $\fX(M)$. Let $\cS_{\Delta}([M]) := \cS_{\fX_{\Delta}^M}([G])$.
\begin{lemma} \label{lem:Levi_Delta_period_vanish}
    Let $Q'$ be a proper parabolic subgroup of $H$. Then for any $f \in \cS_{\Delta}([M_Q])$, we have
    \begin{equation*}
        \int_{[M_{Q'}]^1} f(h) \rd h = 0.
    \end{equation*}
\end{lemma}

\begin{proof}
    Let $\chi \in \fX_{\Delta}^{M_Q}$, and $f \in \cS_{\chi}([M_Q])$, it suffices to show $\int_{[M_{Q'}]^1} f = 0$.

    Assume that $M_{Q'} = G_{n_1} \times \cdots \times G_{n_k} \times \Sp_{2r}$, then $\int_{[M_Q']^1}$ is the product of
    \begin{itemize}
        \item The ``twisted diagnal period" on $\cS([G_{n_i} \times G_{n_i}])$
            \begin{equation*}
               f \mapsto \int_{[G_{n_i}]^1} f(g,w_{\ell} {}^t g^{-1}w_{\ell}) \rd g,
            \end{equation*}
        where $w_{\ell}$ denotes the longest Weyl element as usual.
        \item The symplectic period on $\cS([G_{2r}])$:
        \begin{equation*}
            f \mapsto \int_{[\Sp_{2r}]} f(h) \rd h.
        \end{equation*}
    \end{itemize}
    Then from Theorem \ref{thm:symplectic_period} and the definition of $\fX_{\Delta}^{M_Q}$, it is easy to see that at least one of the integral above is vanishing.
\end{proof}

Combining Lemma \ref{lem:Levi_Delta_period_vanish} and \eqref{eq:Schwartz_function_truncated_period_1}, we see that
\begin{num}
   \item \label{eq:Levi_truncated_period_vanish} Let $f \in \cS_{\Delta}([M_{Q'}])$ and $T \in \fa_0'$ be sufficiently positive. Then
    \begin{equation*}
        \int_{[M_{Q'}]^1} \Lambda^{T,M_{Q'}} f(h) \rd h = 0.
    \end{equation*}
\end{num}

\begin{proposition}
    Let $f \in \cT_{\Delta}([G])$, then $\cP^T(f)$ is a constant of $T$, and this constant is equal to $\cP^*(f)$ in Theorem \ref{thm:n_m}.
\end{proposition}

\begin{proof}
    For $P' \in \cF'$, let $\Gamma'_{P'}$ be the function on $\fa_{P'} \times \fa_{P'}$ defined in \cite[\S 2]{Arthur81}. The function $\Gamma'_{P'}$ is compactly supported in the first variable when the second variable stays in a compact subset and
    \begin{equation*}
        \widetilde{\tau}_{P'}(H-X) = \sum_{\substack{P' \in \cF' \\ P' \subset Q'}} \varepsilon_{Q'} \widehat{\tau}_{P'}^{Q'}(H) \Gamma'_{Q'}(H,X).
    \end{equation*}
    From this, for $f \in \cT([G])$, $T,T' \in \fa_0'$ sufficiently positive, we can write 
    \begin{equation*}
        \Lambda^{T+T'}f(h) = \sum_{Q' \in \cF'} \sum_{\delta \in Q'(F) \backslash H(F)} \Gamma_{Q'}'(H_{Q'}(\delta h)-T_{Q'},T'_{Q'}) \Lambda^{T,Q'}f(\delta h). 
    \end{equation*}
    Therefore
    \begin{equation*}
        \cP^{T+T'}(f) - \cP^T(f) = \sum_{\substack{Q' \in \cF' \\ Q' \ne H}} \int_{[H]_{Q'}} \Gamma'_{Q'}(H_{Q'}(h)-T_{Q'},T'_{Q'}) \Lambda^{T,Q'}f(h) \rd h.
    \end{equation*}
    It remains to show that for any $f \in \cT_{\Delta}([G])$ and any $Q' \ne H \in \cF'$, the integral
    \begin{equation*}
        \int_{[H]_{Q'}^1} \Lambda^{T,Q'}f(h) \rd h
    \end{equation*}
    vanishes. As $\cS_{\Delta}([G])$ is dense in $\cT_{\Delta}([G])$ and the integral above is continuous in $f$. Therefore it suffices to show the vanishing for $f \in \cS_{\Delta}([G])$. However, this directly follows from \eqref{eq:Schwartz_function_truncated_period}. The final statement then follows from \eqref{eq:Schwartz_function_truncated_period_limit}
\end{proof}

\appendix

\section{Computation of the fixed point and the tangent space}
\label{appendix:fixed_point}

In the Appendix, we do an exercise in linear algebra. We show that, under the hypothetical Langlands correspondence, the fixed points of the $L$-parameter and the $L$-function $L(T_x \check{X})$ coincide with the concrete description in \S \ref{sssec:intro_precise} and \S \ref{sssec:n_n_fixed_points}. In particular, the analogue of Theorem \ref{thm:intro} for function field matches with the Conjecture \ref{conj:BZSV}.

\subsection{The global Langlands correspondence}

We will assume the following properties of the hypothetical global Langlands correspondence:

\begin{enumerate}
    \item There exists a locally compact topological group $\cL_F$, such that there is a bijection of isomorphism classes:
    \begin{equation*}
        \{ \text{$n$-dim continuous irreducible rep. of $\cL_F$} \} \longleftrightarrow \{ \text{cuspidal automorphic rep. of $G_n(\bA)$} \}.
    \end{equation*}
    For a cuspidal automorphic representation $\pi$ of $G_n$, we write $\phi_{\pi}$ the corresponding representation $\cL_F \to \GL_n(\C)$, and called it the \emph{L-parameter} of $\pi$.

    \item Let $P=MN$ be a standard parabolic subgroup of $G_n$. Let $\pi$ be a unitary cuspidal automorphic representation of $M$. By the correspondence above, we have an $L$-parameter $\cL_F \to \check{M}$ of $\pi$.

    Let $\Pi = \opn{Ind}_{P(\bA)}^{G_n(\bA)} \pi$, realized as Eisenstein series on $G_n(\bA)$. Then the $L$-parameter of $\Pi$ is given by (or defined to be) $\cL_F \to \check{M} \to \GL_n(\C)$.
\end{enumerate}

\subsection{The fixed points}

Let $\Gamma$ be a group. Let $n>0, m \ge 0$ be integers. Assume that we have a $2n+m$-dimensional semisimple complex representation $\phi$ of $\Gamma$. $\Gamma$ acts on $\check{X} := \GL_{2n+m}(\C)/\GL_n(\C) \times \GL_{n+m}(\C)$ via the representation $\Gamma \to \GL_{2n+m}(\C)$ composes with the natural action of $\GL_{2n+m}(\C)$ on $\check{X}$. We write $\opn{Fix}(\phi)$ for the set of fixed point of $\Gamma$ on $\check{X}$.

Assume that $\phi$ decomposes as
\begin{equation*}
    \phi = \bigoplus_{i=1}^k \phi_i,
\end{equation*}
where $\phi_i$ is an $n_i$-dimensional irreducible representation of $\Gamma$. We remark that $\phi_i$ and $\phi_j$ may be isomorphic for $i \ne j$. 

We may identify $\check{X}$ with the set of pairs $(V,W)$ where $V,W$ are subspaces of $\C^{2n+m}$ with $\dim V = n, \dim W = n+m$ and $\C^{2n+m} = V \oplus W$, where the action of $\GL_{2n+m}(\C)$ is given by $g \cdot (V,W) = (g V,g W)$. The set $\opn{Fix}(\phi)$ then corresponds to decompose the representation into a direct sum of an $n$-dimensional invariant subspace and an $n+m$-dimensional invariant subspace.

Considering the following condition:
   \begin{num}
        \item \label{eq:appendix_1} For any subset $I \subset \{1,2,\cdots,k\}$ such that $\sum_{i \in I} n_i=n$, we have $\phi_i \not \cong \phi_j$ for any $i \in I$ and $j \not \in I$.
    \end{num}

    If the condition \eqref{eq:appendix_1} does not hold. Take a subset $I$ such that $\sum_{i \in I} n_i = n$ and $\phi_i \cong \phi_j$ for some $i \in I$ and $j \not \in I$. Then the subrepresentation $\phi_i \oplus \phi_j$ has infinitely many decomposition into irreducible representation. Take any such decomposition $\phi_i \oplus \phi_j = \rho \oplus \rho'$, then the pair
    \begin{equation*}
        \left( \sum_{\substack{ s \in I \\ s \ne i }} \phi_s + \rho, \sum_{\substack{ t \not \in I \\ t \ne j }} \phi_t + \rho' \right)
    \end{equation*}
    is a fixed point. Therefore there are infinitely many fixed points.

    Conversely, if the condition \eqref{eq:appendix_1} holds. Let $\C^{2n+m} = V \oplus W$ be a decomposition of $\Gamma$-representation. Then \eqref{eq:appendix_1} implies that each isotypic part of $\phi$ must completely lie inside $V$ or $W$. Since isotypic part is canonical, then $\opn{Fix}(\phi)$ is finite. To conclude, we have shown the following lemma
    
\begin{lemma}
    The set $\opn{Fix}(\phi)$ is discrete (in the Zariski topology, so equivalent to finite) if and only if the condition \eqref{eq:appendix_1} holds.
\end{lemma}

From the discussion above, it is easy to see the following
\begin{lemma}
    When $\opn{Fix}(\phi)$ is finite, the set $\opn{Fix}(\phi)$ is in bijection with the set
    \begin{equation*}
       \left \{ I \subset \{1,2,\cdots,k\} \mid  \sum_{i \in I} n_i = n \right \}
    \end{equation*}
\end{lemma}

Finally, the following lemma describe the representation given by the tangent space of fixed point.
\begin{lemma}
    When $x=(V,W) \in \opn{Fix}(\phi)$. Then the representation of $\Gamma$ at $T_x \check{X}$ is isomorphic to $V^{\vee} \otimes W \oplus V \otimes W^{\vee}$.
\end{lemma}

\begin{proof}
    It suffices to show that for $x=(V,W) \in \check{X}$, then as a $\GL(V) \times \GL(W)$ representation, $T_x \check{X} \cong V^\vee \otimes W \oplus V \otimes W^{\vee} = \Hom(V,W) \oplus \Hom(W,V)$.

    Let $\C[\varepsilon]$ be the ring of dual number. Then $T_x \check{X}$ can be identified with the pair of free $\C[\varepsilon]$-submodule $(\bV,\bW)$ of $\C[\varepsilon]^{2n+m}$ such that $\C[\varepsilon]^{2n+m} = \bV \oplus \bW$ and $\bV \otimes_{\C[\varepsilon]} \C = V$, $\bW \otimes_{\C[\varepsilon]} \C= W$. Then it is direct to check that all such $(\bV,\bW)$ is of the form $(V + \varepsilon V + \varepsilon SV, W + \varepsilon W + \varepsilon TW)$ for $(S,T) \in \Hom(V,W) \times \Hom(W,V)$. This finishes the proof.
\end{proof}

\printbibliography

\end{document}